\documentclass[12pt,draft]{amsart}
\usepackage[utf8]{inputenc}
\usepackage[T1]{fontenc}
\usepackage{url}
\usepackage{color}
\usepackage{hyperref}
\usepackage{enumerate}
\usepackage{array}
\usepackage{nicefrac}
\usepackage[section]{placeins}
\usepackage{amsmath,amssymb,enumerate,amsthm}
\usepackage{latexsym}
\usepackage{setspace}
\usepackage{mathtools}
\usepackage{longtable}
\usepackage{mathdots}
\usepackage{graphicx, subfig}
\usepackage{anysize}
\usepackage[linguistics,edges]{forest}

\usepackage{enumitem}

\usepackage[obeyDraft,textsize=tiny]{todonotes}
\usepackage{xpatch}

\presetkeys{todonotes}{
    color=blue!20,
    bordercolor=blue!70!black
}{}

\usepackage{xcolor}

\definecolor{claudiagreen}{HTML}{006400}

\numberwithin{equation}{section}
\theoremstyle{plain}

\newtheorem{theorem}{Theorem}[section]
\newtheorem{lemma}[theorem]{Lemma}
\newtheorem{proposition}[theorem]{Proposition}
\newtheorem{corollary}[theorem]{Corollary}

\theoremstyle{definition}
\newtheorem{definition}[theorem]{Definition}

\newtheorem{example}[theorem]{Example}

\theoremstyle{remark}
\newtheorem{remark}[theorem]{Remark}

\newcommand{\Gr}{\mathrm{Gr}}
\newcommand{\R}{\mathbb{R}}
\newcommand{\N}{\mathbb{N}}
\newcommand{\Z}{\mathbb{Z}}
\newcommand{\Q}{\mathbb{Q}}
\newcommand{\C}{\mathbb{C}}

\newcommand{\e}{\mathfrak{e}}

\renewcommand{\H}{\mathbb{H}}

\newcommand{\SL}{{\mathrm{SL}}}
\newcommand{\PSL}{{\mathrm{PSL}}}

\newcommand{\SO}{{\mathrm{SO}}}

\newcommand{\tr}{\mathrm{tr}}
\newcommand{\sign}{\mathrm{sign}}

\renewcommand{\a}{\mathfrak{a}}

\renewcommand{\O}{\mathrm{O}}
\renewcommand{\cosh}{\mathrm{cosh}}

\renewcommand{\det}{\mathrm{det}}
\renewcommand{\Re}{\mathrm{Re}}
\renewcommand{\Im}{\mathrm{Im}}
\renewcommand{\d}{\mathrm{d}}

\usepackage{tikz}
\usetikzlibrary{patterns}

\title[Traces of Green's functions on $\H^q$ over point-hypersurface pairs]{Traces of Green's functions on $\H^q$ over point-hypersurface pairs}

\author{Claudia Alfes}
\address{Universit\"at Bielefeld, Fakult\"at f\"ur Mathematik, Postfach 100 131, 33501 Bielefeld, Germany}
\email{alfes@math.uni-bielefeld.de}
\thanks{C.~Alfes' research is funded by the Deutsche Forschungsgemeinschaft (DFG, German Research Foundation) -- SFB-TRR 358/1 2023 -- 491392403.}

\author{Jens Funke}
\address{Durham University, Department of Mathematical Sciences, Upper Mountjoy, Stockton Rd, Durham DH1 3LE, United Kingdom}
\email{jens.funke@durham.ac.uk}
\thanks{}

\author{Sebasti\'an Herrero}
\address{Universidad de Santiago de Chile, Departamento de Matem\'atica y Ciencia de la Computaci\'on, Av.~Libertador Bernardo O'Higgins 3363, Santiago, Chile} %, and ETH, Mathematics Dept., CH-8092, Z\"urich, Switzerland}
\email{sebastian.herrero.m@gmail.com}
\thanks{S.~Herrero's research is supported by ANID/FONDECYT/Regular 1250734}

\author{\"Ozlem Imamo\={g}lu}
\address{ETH, Department of Mathematics, CH-8092, Z\"urich, Switzerland}
\email{ozlem@math.ethz.ch}
\begin{document}
\begin{abstract}
The famous algebraicity conjecture of Gross and Zagier \cite{GZ86} for the values at CM points of higher Green's functions was recently proven in \cite{BLY25}. Motivated by this the values of the automorphic Green's function at pairs of special points in hyperbolic $3$-space $\mathbb{H}^3$ were studied in \cite{HIPS25} and an analogous algebraicity result for traces of Green's functions was established. 
In this paper, we consider higher automorphic Green's functions $G_{s}(P,Q)$ on hyperbolic $q$-space $\H^q$ for $q\geq 2$ when the spectral parameter $s$ is a positive even integer, and prove an explicit finite formula for their traces over geodesic hypersurfaces in one of the variables. This formula provides a geometric interpretation of the traces of Green's functions for every choice of a point in $\H^q$ for the second variable. As an application, we obtain rationality results when the second variable is evaluated at a special point providing further evidence that such results for traces of Green's functions do not depend on any underlying complex structure. 
\end{abstract}

\maketitle

% \section*{Notation}

% \begin{enumerate}
%     \item $\H^2$, $\H^3$ and $\H^q$.
%     \item $\kappa=\frac{1-q}{2}$ is the relevant weight.
%     \item $L$ of signature $(1,q)$ with isotropic vector $\ell$. 
%     \item $V=L\otimes \R$, $V_1=\{X_1\in V:(X_1,X_1)=1,(X_1,\ell)>0\}$
%     \item $L=\R\ell+\R\ell'+K$, $(x,y)\in \H^q$ where $x\in W=K\otimes \R$ and $y>0$
%     \item The classical model is $\H^q\simeq \{(x_1,\ldots,x_{q-1},y)\in \R^q: y>0\}$
%     \item $X\in V$
%     \item $z\in \Gr(L)$ and  $z=\R X_1$ with $X_1\in V_1$
%     \item $\C[L'/L]$ with basis $\e_\mu$ and $\langle \cdot ,\cdot \rangle$ is linear in first variable, conjugate-linear in the second variable (as in \cite[p. 17]{brhabil})
%     \item $\tau=u+iv$
%     \item $\mu,\gamma \in L'/L$
%     \item Fourier coefficients of $f$ are $a_f^{\pm}(\mu,m)$
%     \item Iterated raising operator $R^j_{\kappa-2j}$; Green's function at $s=2+2j$ is $G_{2+2j}(\cdot,\cdot)$
% \end{enumerate}

\setcounter{tocdepth}{1}

\tableofcontents

\section{Introduction}

Let $\H:=\{\tau\in \C: \Im(\tau)>0\}$ be the complex upper half-plane and let $\Gamma:=\PSL_2(\Z)$ be the modular group. The automorphic  Green's function  $G_s(\tau_1,\tau_2)$ on $\H$ for $\Gamma$ 
%is defined as 
%$$G_s(\tau_1,\tau_2):=\frac{1}{2\pi}\sum_{\gamma %\in \Gamma}Q_{s-1} \left(\cosh \, d(\tau_1,\gamma %\tau_2)\right)$$
%for $\tau_1,\tau_2\in \H$ with $\Gamma \tau_1\neq %\Gamma\tau_2$ and $s\in \C$ with $\Re(s)>1$, where %$d(\cdot,\cdot)$ denotes the hyperbolic distance and 
%$$Q_{s-1}(t)=\int_0^\infty (t+\sqrt{t^2-1}\,\cosh %\,x)^{-s}  \d x \qquad (t>1)$$
%is the Legendre function of the second kind. The %function $G_s(\tau_1,\tau_2)$ 
admits a meromorphic continuation to $s\in \C$ with  %$s=1$ being the only pole in the half-plane $%\Re(s)>\frac{1}{2}$. The pole at $s=1$ is 
a simple pole at $s=1$ and constant residue. The constant term of the Laurent expansion at $s=1$ is essentially $ \log|j(\tau_1)-j(\tau_2)|$, where $j$ is the classical Klein invariant. Interpreting this constant term in the Laurent expansion as $G_1(\tau_1,\tau_2)$ and using the theory of complex multiplication, one obtains that for $\tau_1,\tau_2$ CM points the value $G_1(\tau_1,\tau_2)$ is, up to an explicit multiplicative factor, the logarithm of an algebraic integer. Gross and Zagier \cite{GZ85} further exploited this relation to give a factorization of differences of singular moduli.

At the end of their highly influential paper \cite{GZ86} Gross and Zagier conjectured that a similar algebraicity result holds for the values at CM points of the higher Green's functions $G_k(\tau_1,\tau_2)$, where $k\geq 2$ is an integer. 
This conjecture was proven by Bruinier, Li and Yang \cite{BLY25} after partial results by Gross, Kohnen and Zagier \cite{GKZ87}, Mellit \cite{Mel08}, Viazovska \cite{Via11,Via19}, Bruinier, Ehlen and Yang \cite{BEY21} and Li \cite{Li22,Li23}.

Motivated by these works, the values of the automorphic Green's function at pairs of special points in hyperbolic $3$-space $\mathbb{H}^3$ were studied in \cite{HIPS25} and an analogous algebraicity result for traces of Green's functions was proven. 

In this paper, we turn our attention to the case of pairs of a point and a geodesic hypersurface in the $q$-dimensional hyperbolic space $\H^q$ for any $q\geq 2$. 

%Here $q$ denotes the real dimension, hence $\H$ corresponds to the case $q=2$.
We begin by expressing the traces of the integrals of Green's functions over geodesic hypersurfaces, taken with respect to one of the variables, as regularized Siegel theta lifts of harmonic Maass--Poincar\'e series. Taking rational linear combinations of traces leads to regularized Siegel theta lifts of weakly holomorphic modular forms. We then derive an explicit finite formula for these theta lifts by adapting the methods of \cite{borcherds}, following the approach of Bruinier and Schwagenscheidt \cite{BS21}. This formula provides a geometric interpretation of our rational linear combination of traces of Green's functions at every point of $\H^q$ for the remaining variable, in terms of the hyperbolic distances to the finitely many geodesic hemispheres enclosing that point. As an application, we establish rationality results when the remaining variable is evaluated at a special point. As a technical caveat, the techniques used in this paper apply only to higher Green's functions $G_s(P,Q)$ where $s\geq 2$ is an even integer.

\bigskip

\emph{The case $q=2$.}
Before we describe our results in general, we start with the more familiar case of $\H=\H^2$.  Consider the lattice
\begin{equation}\label{eq:L-binary}
    L:= \left\{ X = \left(\begin{matrix}
b & c \\ -a & -b
\end{matrix}\right) \colon a,b,c \in \Z \right\}
\end{equation}
with the quadratic form $Q(X)=\det(X)$ of signature $(1,2)$. Its dual lattice is
$$L'= \left\{ \left(\begin{matrix}
b/2 & c \\ -a & -b/2
\end{matrix}\right) \colon a,b,c \in \Z \right\}.$$
The modular group $\Gamma=\PSL_2(\Z)$ acts on $L'$ via $gX:=gXg^{-1}$. This action preserves the quadratic form $Q$ and is trivial on $L'/L$. For $m\in \frac{1}{4}\Z$ put $L'_m:=\{X\in L':Q(X)=m\}$.

The assignment  $ \left(\begin{smallmatrix}
b/2 & c \\ -a & -b/2
\end{smallmatrix}\right)\mapsto [a,b,c]:= ax^2+bxy+cy^2$ identifies $L'$ with the set $\mathcal{Q}$ of integral binary quadratic forms and this identification is $\Gamma$-equivariant\footnote{The action of $\Gamma$ on $\mathcal{Q}$ is $(g\cdot [a,b,c])(x,y):=[a,b,c]((x,y)g^{-t})$.}. Note that the discriminant of $[a,b,c]$ is $b^2-4ac=-4\det(X)$, hence $L'_m$ is in bijection with $\mathcal{Q}_{-4m}:=\{[a,b,c]\in \mathcal{Q}:b^2-4ac=-4m\}$.

Every $X=\left(\begin{smallmatrix}
b/2 & c \\ -a & -b/2
\end{smallmatrix}\right)\in L'$ with  $d:=-4\det(X)<0$  defines a CM point
$$\tau_X:=\frac{-b}{2a}+\frac{i\sqrt{|d|}}{2|a|}\in \H.$$
Similarly, every $Y=\left(\begin{smallmatrix}
b/2 & c \\ -a & -b/2
\end{smallmatrix}\right)\in L'$ with  $D:=-4\det(Y)>0$  defines a geodesic
$$S_Y=S_{[a,b,c]}:=\{\tau \in \H: a|\tau|^2+b\Re(\tau)+c=0\}.$$
%joining two points $w_Y,w_Y'\in \R\cup \{\infty \}$. 
Let $\Gamma_Y=\Gamma_{[a,b,c]}$ be the stabilizer of $Y$ in $\Gamma$. On $S_Y$ we have the $\Gamma_Y$-invariant arc element
$$\d \sigma:=\frac{ |\d \tau|}{\Im(\tau)}.$$
The quotient $\Gamma_Y\backslash S_Y$ is a real-analytic curve that has finite length with respect to $\d \sigma$ if and only if $D$ is not a perfect square. 

We let $G_s(\tau_1,\tau_2)$ be the automorphic  Green's function on $\H$ for $\Gamma$ given by 
$$G_s(\tau_1,\tau_2):=\frac{1}{2\pi}\sum_{\gamma \in \Gamma}Q_{s-1} \left(\cosh \, d(\tau_1,\gamma \tau_2)\right) \qquad (\Re(s)>1).$$
%for $\tau_1,\tau_2\in \H$ with $\Gamma \tau_1\neq \Gamma\tau_2$ and $s\in \C$ with $\Re(s)>1$. 
Here $d(\cdot,\cdot)$ denotes the hyperbolic distance and 
$Q_{s-1}(t)=\int_0^\infty (t+\sqrt{t^2-1}\,\cosh \,x)^{-s}  \d x$ is (for $t>1$) the Legendre function of the second kind. 
%The function $G_s(\tau_1,\tau_2)$ 
%admits a meromorphic continuation to $s\in \C$ with % $s=1$ being the only pole in the half-plane $%\Re(s)>\frac{1}{2}$. The pole at $s=1$ is 
% a simple pole at $s=1$ and constant residue.
For every discriminant $D>0$ we define
$$\tr_{D}\,G_{s}(\tau_1,\cdot):=\sum_{\substack{[Y]\in \Gamma\backslash L_{-D/4}'}}\int_{\Gamma_Y\backslash S_Y}G_{s}(\tau_1,\tau_2)\d \sigma(\tau_2).$$
Note that the sum is finite since $\Gamma\backslash L'_{-D/4}\simeq \Gamma\backslash \mathcal{Q}_D$ is finite. In Section \ref{sec:traces_Greens} we prove that each of the integrals above is absolutely convergent for every $\tau_1\in \H$ (see Proposition \ref{prop:tracegreenshypergeometric}).

The following theorem is a consequence of a special case of Theorem \ref{thm:main_general} in Section \ref{sec:traces_Greens}.  To state it, for every integer $k$ we let $S_{k+\frac{1}{2}}^+(\Gamma_0(4))$ denote the space of weight $k+\frac{1}{2}$ holomorphic cusp forms for  $\Gamma_0(4)$ satisfying Kohnen's plus condition\footnote{For $k\in \Z$, a modular form $f$ of weight $k+\frac{1}{2}$ for $\Gamma_0(4)$ is said to satisfy Kohnen's plus condition if its Fourier coefficients $a_f(n)$ vanish whenever $(-1)^kn\equiv 2,3$ mod 4.}. 

For $z\in \C$ and $r\in \N$ let $(z)_r:=z(z+1)\cdots (z+r-1)$ and $(z)_0:=1$.

\begin{theorem}\label{thm:1.1}
Let $j\in \N_0$ and assume $S_{\frac{5}{2}+2j}^+(\Gamma_0(4))=\{0\}$. Then, for every non-square discriminant $D>0$ there exists $C\in \sqrt{D}\,\Q$, depending on both $j$ and $D$, such that for every $\tau=x+iy\in \H$ we have
\begin{equation*}
       \mathrm{tr}_{D}\,G_{2+2j}(\tau,\cdot)=\frac{C}{y^{1+2j}}- 2\sum_{\substack{[a,b,c]\in \mathcal{Q}_D\\  a(y^2+x^2)+bx+c>0 >a}}
\sum_{r=0}^j \binom{j}{r}
\frac{(j+\frac32)_r}{(\frac32)_r}\sinh^{1+2r} d(\tau,S_{[a,b,c]}).
       %\sum_{r=0}^j (-1)^r
   %{j \choose r}\left(j+\frac{3}{2}\right)_{r}
  %\Gamma\left(-\frac{1}{2}-r\right)\sinh^{1+2r} d(x+iy,S_Y).
    \end{equation*}
 %    \begin{equation}\begin{split}
 %    & \frac{1}{(4\pi)^{j+1}}\Psi^j(f,(x,y))\\
 %   &=  j!   a_f(0)\frac{(j+\frac{3}{2})_{j}}{(2+2j)!y^{1+2j}}  B_{2+2j} \\
 % & \quad  +\frac{D^{j+\frac{1}{2}}}{4^j\sqrt{\pi}} \sum_{\substack{Y=a\ell'+c\ell+b\in L'\\ Q(Y)=-D/4\\ a(y^2+x^2)+2bx+c>0 >a}} \sum_{r=0}^j (-1)^r
 %   {j \choose r}\left(j+\frac{3}{2}\right)_{r}
 %  \Gamma\left(-\frac{1}{2}-r\right)\sinh^{1+2r} d(x+iy,S_Y).
 %  \end{split}
% \end{equation}
Moreover, the double sum on the right hand side is a finite sum.
%over $[a,b,c]\in \mathcal{Q}_D$ is a finite sum and runs over geodesics semicircles $S\subset \H$ so that $x+iy$ in the bounded component of $\H\setminus S$.
\end{theorem}

\begin{remark}
    \begin{enumerate}
        \item For $[a,b,c]\in \mathcal{Q}_D$ with $D>0$ and $a\neq 0$ the set $\H\setminus S_{[a,b,c]}$ has one bounded connected component and one unbounded connected component with respect to the Euclidean distance in $\H$. For $\tau=x+iy\in \H$ the condition $a(x^2+y^2)+bx+c>0>a$ means that $\tau$ lies in the bounded component of $\H\setminus S_{[a,b,c]}$. The finiteness of the double sum in Theorem \ref{thm:1.1} follows from the fact that for every fixed $\tau\in \H$ there are only finitely many $[a,b,c]\in \mathcal{Q}_D$ satisfying such a condition (see Lemma \ref{lem:finiteness_geodesics} in Section \ref{sec:finite_sum_resps}).
        \item Given $j\in \N_0$ and $D>0$ a non-square discriminant, our formula for $\mathrm{tr}_{D}\,G_{2+2j}(\tau,\cdot)$ in Theorem \ref{thm:1.1} still holds under the more general assumption that the $D$-th Fourier coefficient of every form in $ S_{\frac{5}{2}+2j}^+(\Gamma_0(4))$ vanishes. This condition implies that there exists a weakly holomorphic modular form $f\in M^!_{-\frac{1}{2}-2j}(\Gamma_0(4))$, satisfying Kohnen's plus condition, with Fourier expansion $f(\tau)=q^{-D}+a_f(0)+O(q)$ and rational Fourier coefficients. The constant $C$ appearing in  Theorem \ref{thm:1.1} is 
        $$C=   \frac{a_f(0)}{D^{\frac{1}{2}+j}}\frac{2^{2j}j!(j+\frac{3}{2})_{j}}{(2+2j)!}  B_{2+2j},$$
        where $B_n$ denotes the $n$-th Bernoulli number.
        %in order to present a simplified version of our main result. In general, we take linear combinations of traces with coefficients coming from rational relations for $S_{\frac{5}{2}+2j}^+(\Gamma_0(4))$ (or equivalently, coming from principal parts of weakly holomorphic modular forms of weight $-\frac{1}{2}-2j$ satisfying Kohnen's plus condition). See 
    \end{enumerate}
\end{remark}

For every $\tau=x+iy\in \H$ and $[a,b,c]\in \mathcal{Q}_D$ with $D>0$ we have
$$\sinh d(\tau,S_{[a,b,c]})=\frac{|a(x^2+y^2)+bx+c|}{y\sqrt{D}}.$$
%(see Corollary \ref{cor-distance} in Section \ref{sec:points_and_geodesic_hypersurfaces}). 
Hence, as a direct consequence of Theorem \ref{thm:1.1} we obtain the following.

\begin{theorem}\label{2.1}
Let $j\in \N_0$ and assume $S_{\frac{5}{2}+2j}^+(\Gamma_0(4))=\{0\}$. Then, for every non-square discriminant $D>0$ and every CM point $\tau\in \H$ of discriminant $d<0$ we have
    \begin{align*}
        \tr_D\,G_{2+2j}(\tau,\cdot)\in  \sqrt{D|d|}\, \Q.
    \end{align*}
\end{theorem}

\begin{example}
Choose $j=0$ and $D=5$ in which case $\tr_D$ has only one summand corresponding to the quadratic form $Y_0=[-1,1,1]$. There is a unique form $f\in M_{-\frac{1}{2}}^{!}(\Gamma_0(4))$ with Fourier expansion $f(\tau)=q^{-5}+48+O(q)$. Then, for every $\tau=x+iy\in \H$ we get
\begin{align*}
       \tr_5\, G_2(\tau,\cdot) =  \int_{\Gamma_{Y_0}\backslash S_{Y_0}}G_{2}(\tau,\tau_2)\d \sigma(\tau_2)
       =\frac{4\sqrt{5}}{5y}- 2\sum_{\substack{[a,b,c]\in \mathcal{Q}_5\\  a(y^2+x^2)+bx+c>0 >a}}\sinh d(\tau,S_{[a,b,c]}).
    \end{align*}
    We can compute the finite sum on the right hand side explicitly for any $\tau \in \H$. Since  $\tr_5\, G_2(\tau,\cdot) $ is $\Gamma$-invariant, we can assume that $\tau$ is in $\mathcal{F}$, the standard fundamental domain for $\PSL_2(\Z)$. Note that the only two geodesics $S_{[a,b,c]}$ crossing $\mathcal{F}$ with $[a,b,c]\in \mathcal{Q}_5$ are $S_{[-1,1,1]}$ and $S_{[-1,-1,1]}$; see Figure \ref{fig:D=5}. 

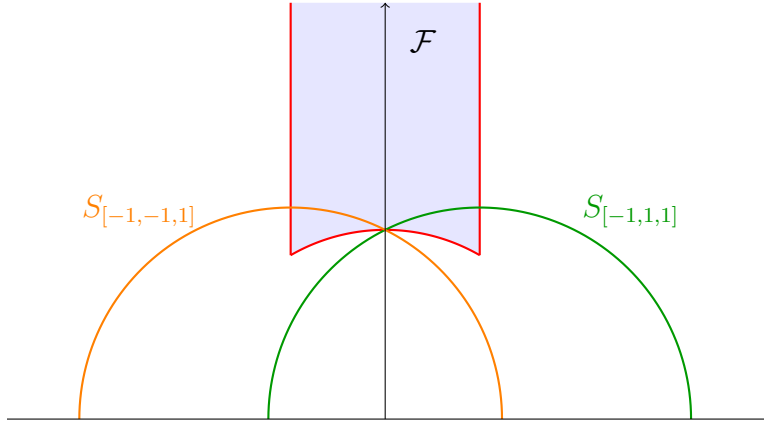
\begin{figure}[h]
\centering
\begin{tikzpicture}[scale=2.5]
  % Clip to upper half plane
  \clip (-2, 0) rectangle (2, 2.2);
  % Shade the fundamental domain
  % The domain: -1/2 <= x <= 1/2, and outside the unit circle
    \fill[blue!10]
    (-0.5, 2.2) -- (-0.5, 0.8660147)
    arc[start angle=120, end angle=60, radius=1]
    -- (0.5, 2.2) -- cycle;
  % Actually, let's construct the exact fundamental domain
  % F = { z : |z| >= 1, -1/2 <= Re(z) <= 1/2 }
  %\fill[blue!10, even odd rule]
    % Outer rectangle covering the region
  %  (-0.5, 0) rectangle (0.5, 2.2)
    % Inner unit disk to subtract
  %  (0,0) circle (1);

  % The unit circle arc (|z| = 1) bounding the domain below
  \draw[thick, red, domain=60:120, samples=100]
    plot ({cos(\x)}, {sin(\x)});

  % Vertical line x = -1/2
  \draw[thick, red] (-0.5, 0.866) -- (-0.5, 2.2);

  % Vertical line x = 1/2
  \draw[thick, red] (0.5, 0.866) -- (0.5, 2.2);

  % Arc 1: x^2 + y^2 - x - 1 = 0
  % => (x - 1/2)^2 + y^2 = 5/4
  % Center (1/2, 0), radius sqrt(5)/2
  % This arc passes through (1/2, 1) and (-1/2, ...)
  % Let's find intersections with the fundamental domain
  % At x = -1/2: 1/4 + y^2 + 1/2 - 1 = 0 => y^2 = 1/4 => y = 1/2
  % At x = 1/2: 1/4 + y^2 - 1/2 - 1 = 0 => y^2 = 5/4 => y = sqrt(5)/2
  \draw[thick, green!60!black, domain=0:180, samples=100]
    plot ({0.5 + 1.118*cos(\x)}, {1.118*sin(\x)});

  % Arc 2: x^2 + y^2 + x - 1 = 0
  % => (x + 1/2)^2 + y^2 = 5/4
  % Center (-1/2, 0), radius sqrt(5)/2
  \draw[thick, orange, domain=0:180, samples=100]
    plot ({-0.5 + 1.118*cos(\x)}, {1.118*sin(\x)});

  % Axes
  \draw[->] (-2.2, 0) -- (2.2, 0) node[right] {};
  \draw[->] (0, 0) -- (0, 2.2) node[above] {};

  % Labels
  %\node at (0.75, 0.3) {$|z|=1$};
  \node at (1.3, 1.1) [green!60!black] {$S_{[-1,1,1]}$};
  \node at (-1.3, 1.1) [orange] {$S_{[-1,-1,1]}$};
  %\node at (-0.5, -0.1) [below] {$-\frac{1}{2}$};
  %\node at (0.5, -0.1) [below] {$\frac{1}{2}$};
  \node at (0, -0.1) [below] {$0$};
  \node at (0.2, 2) {$\mathcal{F}$};
\end{tikzpicture}
\caption{Standard fundamental domain $\mathcal{F}$ for $\PSL_2(\Z)$ and the only two geodesics $S_{[a,b,c]}$  crossing $\mathcal{F}$ with $[a,b,c]\in \mathcal{Q}_5$.}
\label{fig:D=5}
\end{figure}
% Letting
% \begin{eqnarray*}
%     \mathcal{R}_{[-1,1,1]}&:=\{\tau=x+iy\in \mathcal{F}:-(x^2+y^2)+x+y>0\},\\
%     \mathcal{R}_{[-1,-1,1]}&:=\{\tau=x+iy\in \mathcal{F}:-(x^2+y^2)-x+y>0\},\\
% \end{eqnarray*}
The geodesics $S_{[-1,1,1]}$ and $S_{[-1,-1,1]}$ induce a decomposition of $\mathcal{F}$ into three regions according to the conditions $-(x^2+y^2)+x+1>0$ (region $\mathcal{R}_1$), $-(x^2+y^2)-x+1>0$ (region $\mathcal{R}_2$), and the complement of these two in $\mathcal{F}$ (region $\mathcal{R}_3$); see Figure \ref{fig:2}.

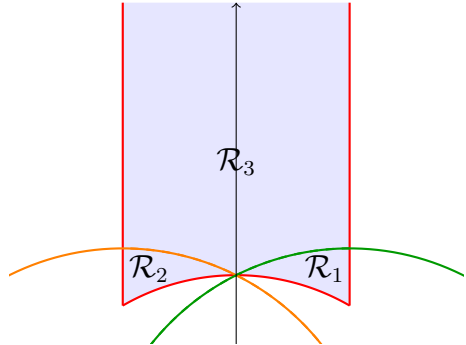
\begin{figure}[h]
\centering
\begin{tikzpicture}[scale=3]
  % Clip to upper half plane
  \clip (-1, 0.7) rectangle (1, 2.2);
  % Define the three regions of the fundamental domain
  % R1: left of the orange geodesic
    \fill[blue!10]
    (-0.5, 2.2) -- (-0.5, 0.8660147)
    arc[start angle=120, end angle=60, radius=1]
    -- (0.5, 2.2) -- cycle;

  % R2: between the two geodesics (inside the central region)
  % \fill[green!20]
  %   (0, 1)
  %   arc[start angle=90, end angle=60, radius=1]
  %   -- (0.5, 0.8660147)
  %   arc[start angle=60, end angle=0, radius=1.118]
  %   -- (0.5 + 1.118, 0)
  %   -- (-0.5 - 1.118, 0)
  %   arc[start angle=180, end angle=90, radius=1.118]
  %   -- (0, 1) -- cycle;

  % R3: right of the green geodesic
  % \fill[red!20]
  %   (0.5, 2.2) -- (0.5, 0.8660147)
  %   arc[start angle=60, end angle=90, radius=1]
  %   -- (0, 1)
  %   arc[start angle=90, end angle=0, radius=1.118]
  %   -- (0.5, 2.2) -- cycle;

  % Redraw the unit circle arc bounding the domain below
  \draw[thick, red, domain=60:120, samples=100]
    plot ({cos(\x)}, {sin(\x)});

  % Vertical boundaries x = -1/2 and x = 1/2
  \draw[thick, red] (-0.5, 0.866) -- (-0.5, 2.2);
  \draw[thick, red] (0.5, 0.866) -- (0.5, 2.2);

  % Geodesic 1: x^2 + y^2 - x - 1 = 0
  \draw[thick, green!60!black, domain=0:180, samples=100]
    plot ({0.5 + 1.118*cos(\x)}, {1.118*sin(\x)});

  % Geodesic 2: x^2 + y^2 + x - 1 = 0
  \draw[thick, orange, domain=0:180, samples=100]
    plot ({-0.5 + 1.118*cos(\x)}, {1.118*sin(\x)});

  % Dashed partition boundaries inside F
  \draw[dashed, thick, green!60!black, domain=90:180, samples=100]
    plot ({0.5 + 1.118*cos(\x)}, {1.118*sin(\x)});
  \draw[dashed, thick, orange, domain=0:90, samples=100]
    plot ({-0.5 + 1.118*cos(\x)}, {1.118*sin(\x)});

  % Axes
  \draw[->] (-2.2, 0) -- (2.2, 0);
  \draw[->] (0, 0) -- (0, 2.2);

  % Labels for the three regions
  \node at (-0.385, 1.03) {$\mathcal{R}_2$};
  \node at (0, 1.5) {$\mathcal{R}_3$};
  \node at (0.39, 1.03) {$\mathcal{R}_1$};

  % Labels for the geodesics
  %\node at (1.3, 1.1) [green!60!black] {$S_{[-1,1,1]}$};
  %\node at (-1.3, 1.1) [orange] {$S_{[-1,-1,1]}$};

  % Other labels
  %\node at (0, -0.1) [below] {$0$};
  %\node at (0.2, 2) {$\mathcal{F}$};
\end{tikzpicture}
\caption{The fundamental domain \(\mathcal{F}\) partitioned into three regions \(\mathcal{R}_1, \mathcal{R}_2, \mathcal{R}_3\) by the geodesics $S_{[-1,-1,1]}$ and $S_{[-1,1,1]}$.}
\label{fig:2}
\end{figure}

For $\tau=x+iy\in \mathcal{F}$ we have
$$\sum_{\substack{[a,b,c]\in \mathcal{Q}_5\\  a(y^2+x^2)+bx+c>0 >a}}\sinh d(\tau,S_{[a,b,c]})
=\begin{cases}
    \frac{-(x^2+y^2)+x+1}{y\sqrt{5}} & \text{if }\tau \in \mathcal{R}_1,\\
    \frac{-(x^2+y^2)-x+1}{y\sqrt{5}} & \text{if }\tau \in \mathcal{R}_2,\\
     0 &\text{if }\tau \in \mathcal{R}_3.
\end{cases}$$
Choosing $\tau$ as the CM points $2i, i\in \mathcal{R}_3$ and $\frac{1+i\sqrt{3}}{2}\in\mathcal{R}_1$, corresponding to $[1,0,4],[1,0,1]$ and $[1,-1,1]$ of discriminant $d=-16,-4$ and $-3$, respectively, we get
    \begin{align*}
       \tr_5 G_2(2i,\cdot)=\frac{2\sqrt{5}}{5}, \quad \tr_5\, G_2(i,\cdot) =\frac{4\sqrt{5}}{5} , \quad \tr_5 \,G_2\left(\frac{1+i\sqrt{3}}{2},\cdot\right)=\frac{2\sqrt{15}}{5}.
    \end{align*}
    The decimal expansions of these numbers are $0.8944\ldots$, $1.7888\ldots,$ and $1.5491\ldots$, respectively. Approximating the Green's function $G_2(\tau_1,\tau_2)$ using only the 2042 elements of $\PSL_2(\Z)$ with entries bounded by 20 in absolute value gives 
    \begin{align*}
       \tr_5 G_2(2i,\cdot)\approx 0.8912, \quad \tr_5\, G_2(i,\cdot) \approx 1.7873 , \quad \tr_5 \,G_2\left(\frac{1+i\sqrt{3}}{2},\cdot\right)\approx 1.5466.
    \end{align*}
\end{example}

\bigskip

\noindent \emph{The general case.}
Next, we consider the general case of  $\H^q$ and the pairs of points and geodesic hypersurfaces. 

Let $L$ denote an even lattice of signature $(1,q)$ with quadratic form $Q$ and associated bilinear form $(\cdot,\cdot)$. We further assume that $L$ splits a hyperbolic plane over $\Z$, meaning that there exist isotropic vectors $\ell,\ell'\in L$ with $(\ell,\ell')=1$ and a negative definite sublattice $K$ of rank $q-1$ of $L$ such that  $L=(\Z \ell \oplus \Z \ell') \oplus^{\perp} K$ where $\oplus^\perp$ denotes orthogonal sum. Note that $L'=(\Z \ell \oplus \Z \ell') \oplus^{\perp} K'$ and $L'/L\simeq K'/K$.

Let $W:=K\otimes \R$ and $V:=L\otimes \R=\R \ell\oplus\R\ell'\oplus W$. A general element $X\in V$ will be written as $X=a\ell'+c\ell +b$ with $a,c\in \R$ and $b\in W$. The space
$$\H_W:= W\times \R_{>0}=\{P=(x,y):x\in W, y>0\}$$
is a $q$-dimensional hyperbolic space. Choosing an orthogonal basis $\{w_1,\ldots,w_{q-1}\}$ of $W$ with $Q(w_i)=-1$ for all $i$, gives the \emph{upper half-space model}
\[\H_W\simeq \H^q:=\{(x_1,\ldots,x_{q-1},y)\in \R^q:y>0\}.\]
We can identify $\H_W$ with $\Gr(L)=\Gr(V)$, the Grassmannian of positive definite lines in $V$, as follows: for every $P=(x,y)\in \H_W$ let
$$X_1(P):=\frac{1}{\sqrt{2}y}x+\frac{1}{\sqrt{2}y}\ell'+\left(\frac{y^2-Q(x)}{\sqrt{2}y}\right)\ell\in V.$$
Then $Q(X_1(P))=1$ and the map $P\mapsto \R X_1(P)$ defines our identification $\H_W\simeq \Gr(L)$. 

Every $X=a\ell'+c\ell+b\in V$ with  $Q(X)>0$ defines a point $P_X\in \H_W$ given by
\begin{equation*}
P_X:=\left(\frac{b}{a},\frac{\sqrt{Q(X)}}{|a|}\right)
\end{equation*}
and every $Y=a\ell'+c\ell+b\in V$ with  $Q(Y)<0$ defines a geodesic hypersurface $S_Y\subset  \H_W$ given by
\begin{equation*}
    S_Y:=\{(x,y)\in \H_W: a(y^2-Q(x))+(b,x)+c=0\}.
\end{equation*}

The special orthogonal group $\SO(V)$ has two connected components. We denote by $\SO(V)^+$ the connected component of the identity and define $\SO(L)^+:= \SO(L)\cap \SO(V)^+$. Let $\Gamma$ be a finite-index subgroup of $\SO(L)^+$ acting trivially on $L'/L$. Then $\Gamma$ acts discontinuously on $\Gr(V)\simeq \H_W$ and this action preserves orientation. 

For $s\in \C$ with $\Re(s)>1$ and $P,Q\in \H_W$ we let 
\begin{equation}\label{eq:Green_H_W}
G_s(P,Q)=\sum_{\gamma \in \Gamma}g_s(\cosh \, d(P,\gamma Q))
\end{equation}
be the automorphic Green's function associated to $\Gamma$. Here
$$g_s(t):=(2\pi)^{-\frac{q}{2}}e^{-\pi i \frac{(q-2)}{2}} (t^2-1)^{-\frac{q-2}{4}}Q_{s-2+\frac{q}{2}}^{\frac{q-2}{2}}(t)$$
with $Q_{\nu}^\mu$ the Legendre function of the second kind. The series \eqref{eq:Green_H_W} converges absolutely and uniformly for $(P,Q)$  on compact subsets of $(\H_W\times \H_W)\setminus \{(P,Q): P\equiv Q \text{ mod }\Gamma\}$.

For every $\mu\in L'/L$ and every $m \in \Z-Q(\mu)$ with $m>0$, the group $\Gamma$ acts on the set $L_{\mu,-m}:=\{Y\in L+\mu:Q(Y)=-m\}$ with finitely many orbits. For each $Y\in L_{\mu,-m}$ the action of $\Gamma$ on $\H_W$ restricts to an action of the stabilizer subgroup $\Gamma_Y$ on $S_Y$; the quotient $\Gamma_Y\backslash S_Y$ has finite hyperbolic volume if $q\geq 3$. For $q=2$ this holds if and only if the orthogonal complement of $Y$ in $L\otimes \mathbb{Q}$ is anisotropic. On $S_Y$ we have the Euclidean measure $\d\omega$ and the hyperbolic hypersurface measure
$\d\sigma:=\frac{\d\omega}{y^{q-1}}.$
For every $\mu\in L'/L$, every $m \in \Z-Q(\mu)$ with $m>0$ and every $P\in \H_W$ we define
\begin{equation*}
\tr_{\mu,m}\,G_{s}(P,\cdot)
:=
\sum_{\substack{[Y]\in \Gamma \backslash (L+\mu)\\Q(Y)=-m}}
\int _{\Gamma_Y \backslash S_Y}
G_s(P,Q)\d\sigma(Q).
\end{equation*}
Note that the sum over $[Y]$ is a finite sum. Moreover, we show in Section \ref{sec:traces_Greens}  that each integral in the sum converges absolutely when $\Re(s)>1$ (see Proposition \ref{prop:tracegreenshypergeometric}). Also, note that $\tr_{\mu,m}\,G_{s}(P,\cdot)=\tr_{-\mu,m}\,G_{s}(P,\cdot)$.

\begin{remark}\label{rmk:traces_D_vs_mu,m}
    In the case $q=2$ and $L$ as in \eqref{eq:L-binary} we have $L'/L\simeq \Z/2\Z$. Then $\tr_{D}=\tr_{0,-D/4}$ if $D$ is even and $\tr_{D}=\tr_{1,-D/4}$ if $D$ is odd.
\end{remark}

The following theorem is also consequence of Theorem \ref{thm:main_general} in Section \ref{sec:traces_Greens}.  Here, for every $\kappa\in \frac{1}{2}\Z$ we let $S_{\kappa,L^-}$ denote the space of vector-valued cusp forms  of weight $\kappa\in \frac{1}{2}\Z$ for the Weil representation $\rho_{L^-}$ (see Section \ref{sec:preliminaries} for details).

Recall that $K$ is a negative definite sublattice of $L$ of rank $q-1$, with dual lattice $K'$. Since $L'/L\simeq K'/K$, for every $\mu=(a\ell'+c\ell+b)+L\in L'/L$ we simply write $K+\mu:=K+b$ for the corresponding element in $K'/K$. We also define $L\pm \mu:= (L+\mu)\cup (L-\mu)$ and $K\pm \mu:=(K+\mu)\cup (K-\mu)$.

\begin{theorem}\label{thm:2.1}
Let $j\in \N_0$ and assume $S_{\frac{3+q}{2}+2j,L^-}=\{0\}$. Then, for every $\mu\in L'/L$ and every $m \in \Z-Q(\mu)$ with $m>0$ there exists $C\in \sqrt{m}\,\Q$, depending on $j,\mu$ and $m$, such that for every $P=(x,y)\in \H_W$ we have
\begin{equation*}\begin{split}
           \mathrm{tr}_{\mu,m}\,G_{2+2j}(P,\cdot)&= \frac{C}{y^{1+2j}}+y\frac{\sqrt{m}\,\# \{Y\in K+\mu:Q(Y)=-m\}}{\left(q-1+2j\right)(j+1)}\\
           & \qquad +\frac{j!}{4}\sum_{r=0}^j \frac{\left(j+\frac{1+q}{2}\right)_r}{(j-r)!(2+2r)! m^{\frac{1}{2}+r}y^{1+2r}}\,  \varepsilon_\mu \sum_{\substack{Y\in K\pm \mu \\ Q(Y)=-m}}\mathbb{B}_{2+2r}((Y,x))\\
           & \qquad \qquad -\sum_{r=0}^j {j\choose r}\frac{\left(j+\frac{1+q}{2}\right)_r}{(3/2)_r} \,\varepsilon_\mu \sum_{\substack{Y=a\ell'+c\ell+b\in L\pm \mu \\ Q(Y)=-m\\a(y^2-Q(x))+(b,x)+c>0>a}}\sinh^{1+2r}d(P,S_Y).
\end{split}
    \end{equation*}
Here 
$$\varepsilon_\mu:=\begin{cases} 1 & \text{if }\mu\neq -\mu,\\ 2& \text{if }\mu =-\mu,\end{cases}$$ 
and $\mathbb{B}_{2+2r}(t)$ denotes the $1$-periodic function on $\R$ that agrees with the Bernoulli polynomial $B_{2+2r}(t)$ on $[0,1)$.  All the sums on the right hand side of the above formula are  finite sums.
\end{theorem}

\begin{remark}
     \begin{enumerate}
         \item For every $Y=a\ell'+c\ell+b\in L'$ with $m>0$ and $a\neq 0$ the set $\H_W\setminus S_Y$ has one bounded connected component and one unbounded connected component with respect to the Euclidean distance in $\H_W$. For $P=(x,y)\in \H_W$ the condition $a(y^2-Q(x))+(b,x)+c>0>a$ means that $P$ lies in the bounded component of $\H_W\setminus S_{Y}$. For fixed $P\in \H_W$ there are only finitely many $Y\in L'$ with $Q(Y)=-m$ satisfying such  condition by Lemma \ref{lem:finiteness_geodesics} in Section \ref{sec:finite_sum_resps}.
         \item For each $Y\in K\pm \mu$ with $Q(Y)=-m$ the term $\mathbb{B}_{2+2r}((Y,x))$ is, locally around $P$, a polynomial in $y\sinh d(P,S_{\widetilde{Y}})$ where $S_{\widetilde{Y}}$ is any of the vertical hypersurfaces closest to $P$ and parallel to $S_Y$; see Section \ref{sec:geometric-interpretation} for details. 
        \item Theorem \ref{thm:2.1} still holds under the more general assumption that the Fourier coefficient of index $(\mu,m)$ of every vector-valued cusp form in $S_{\frac{3+q}{2}+2j,L^-}$ vanishes. Such a condition implies that there exists a vector-valued weakly holomorphic modular form $f\in M^!_{\frac{1-q}{2}-2j,L}$ with Fourier expansion $f(\tau)=q^{-m}(\e_\mu+\e_{-\mu})+O(1)$ and rational Fourier coefficients. The constant $C$ appearing in  Theorem \ref{thm:2.1} is 
        $$C=   \frac{a_f(0,0)}{m^{\frac{1}{2}+j}}\frac{j!(j+\frac{1+q}{2})_{j}}{4(2+2j)!}  B_{2+2j},$$
        where $a_f(0,0)$ denotes the Fourier coefficient of $f$ of index $(0,0)\in L'/L\times \Q$.
        \item Theorem \ref{thm:1.1} follows from Theorem \ref{thm:2.1} by choosing $q=2$, $L$ as in \eqref{eq:L-binary}, $m=D/4$ and $\mu=0$ or $1$ depending on whether $D$ is even or odd (see Remark \ref{rmk:traces_D_vs_mu,m}). Then $(K',Q)\simeq (\frac{1}{2}\Z,-x^2)$ and since $D$ is not a square there are no vectors $Y\in K\pm \mu$ satisfying $Q(Y)=-m$. Here we also use that $S_{\frac{5}{2}+2j,L^-}\simeq S^+_{\frac{5}{2}+2j}(\Gamma_0(4))$ 
        via $f_0(\tau)\e_0+f_1(\tau)\e_1\mapsto f_0(4\tau)+f_1(4\tau)$.% defines an isomorphism between $M_{\frac{1-q}{2}-2j,L}^!$ and the subspace of forms in $M^{!}_{\frac{1-q}{2}-2j}(\Gamma_0(4))$ satisfying Kohnen's plus condition.
     \end{enumerate}
 \end{remark}

For every $P=(x,y)\in \H_W$ and $Y=a\ell'+c\ell+b\in V$ with $Q(Y)=-m<0$ we have
$$\sinh d(P,S_Y)=\frac{|a(y^2-Q(x))+(b,x)+c|}{2y\sqrt{m}}$$
by Corollary \ref{cor-distance} in Section \ref{sec:points_and_geodesic_hypersurfaces}. If $P=P_X$ with $X\in L\otimes \Q$, then $P$ is called a \emph{special point}. As a direct consequence of Theorem \ref{thm:2.1} we obtain the following  result.

\begin{theorem}\label{thm:2.2}
Let $j\in \N_0$ and assume $S_{\frac{3+q}{2}+2j,L^-}=\{0\}$. Then, for every $\mu\in L'/L$, every $m \in \Z-Q(\mu)$ with $m>0$ and every $X\in L\otimes \Q$ with $Q(X)>0$, we have
    \begin{align*}
        \tr_{\mu,m}\,G_{2+2j}(P_X,\cdot)\in  \sqrt{mQ(X)}\, \Q.
    \end{align*}
\end{theorem}

The method of proof of our results makes use of the machinery of Siegel theta lifts in signature $(1,q)$. More explicitly, the proofs consist of the following major steps.

\begin{enumerate}
    \item We integrate the automorphic Green's function $G_s(P,Q)$ over geodesic hypersurfaces, in the variable $Q$,  and obtain a formula for $\tr_{\mu,m} G_{2+2j}(P, \cdot)$ in terms of an infinite sum of hypergeometric functions.
    \item We show that the trace $\tr_{\mu,m} G_{2+2j}(P, \cdot)$ is the regularized theta lift of $R_\kappa^j F_{\mu,m}$ where $F_{\mu,m}$ is a harmonic Maass--Poincar\'e series of weight $\frac{1-q}{2}-2j$ and $R_\kappa^j=R_{\kappa+2j-2}\circ \cdots \circ R_\kappa$ denotes the iterated version of the raising operator (see \eqref{eq:lowering_and_raising_operators} in Section \ref{sec:HarmonicMaassForms}). %To this end we unfold  $F_{\mu,m,\kappa,s^*}$ and observe that the obtained infinite sum of hypergeometric functions can also be written in terms of the integral of the Green's function over hypersurfaces.
    \item We prove a general finite formula for the evaluation of the regularized theta lift of $R^jf$, valid at any point $z\in \Gr(L)$, for every weight $\frac{1-q}{2}-2j$ weakly holomorphic vector-valued modular form $f$.
    %\item Using the method of \cite{borcherds} and \cite{BS21}, we calculate the theta lift by unfolding the theta kernel. In the case that the input $F_m$ is a weakly holomorphic form, this leads to a finite formula expression for the evaluation of the lift at any point.
    \item When $S_{\frac{3+q}{2}+2j,L^-}=\{0\}$ the form $f=F_{\mu,m}$ is weakly holomorphic and has rational Fourier coefficients. Combining the previous steps, we deduce our finite formula for $\tr_{\mu,m} G_{2+2j}(P, \cdot)$, valid for any $P\in \H_W$, when $S_{\frac{3+q}{2}+2j,L^-}=\{0\}$. From this, we deduce  both an interpretation of these traces in terms of hyperbolic distances to geodesic hypersurfaces and our algebraicity result for the evaluation at special points.
\end{enumerate}

Our results on the algebraicity of the values at special points of the traces of Green's functions over geodesic hypersurfaces can also be obtained using the splitting of the theta kernel, as was done for point-point pairs in previous works; see, e.g., \cite{BEY21} and \cite{HIPS25}. In Section \ref{sec:theta_splitting} we briefly present such a computation using results of Males \cite{males} and Li and Schwagenscheidt \cite{LS22}.

In the case $q=2$, the regularized theta lift occurring in step (2) also shows up in the study of cycle integrals of certain meromorphic modular forms attached to positive definite binary quadratic forms, notably in \cite{ABMS21}. In Section \ref{sec:cycle_integral_meromorphic_mod_forms} we explain this connection and show how to use the results in this paper to give an alternative proof of some of the algebraicity results for such meromorphic modular forms.

This paper is organized as follows. Section \ref{sec:preliminaries} covers the necessary preliminaries on Weil representations, harmonic Maass forms, Maass--Poincar\'e series and rationality results for spaces of cusp forms. Section \ref{sec:quad_spaces_(1,q)_hyp_geometry} introduces the geometric setting of quadratic spaces of signature $(1,q)$ and their relationship to hyperbolic geometry. It includes the hyperbolic distance formulas that are used in our computations. Section \ref{sec:Siegel} presents the Siegel theta lift in signature $(1,q)$ and our finite sum representation for the lift of weakly holomorphic modular forms, providing also a geometric interpretation. In Section \ref{sec:Hypersurface integrals of Green's functions} we give basic properties of automorphic Green's functions and compute their integrals over geodesic hypersurfaces. This leads to an evaluation of the traces as an infinite sum of hypergeometric series, and its connection to a theta lift. There we prove our main finiteness and rationality results. Finally, Section \ref{sec:remarks} presents an alternative evaluation of the theta lift at special points and discusses connections to cycle integrals of meromorphic modular forms attached to definite binary quadratic forms.

\subsection*{Acknowledgments} This project was initiated during a research visit by C. Alfes, S. Herrero and J. Funke to work with \"O. Imamo\=glu at ETH Zurich. They thank \"O.~Imamo\=glu for the invitation and the Department of Mathematics of ETH Zurich for providing a stimulating research environment. J. Funke also thanks the Forschungsinstitut f\"ur Mathematik at ETH Zurich for their generous support of many research visits over the years. 

\subsection*{Statement on the use of AI}
All mathematical results in this paper were obtained by the human authors. 
AI was used to identify minor mathematical inconsistencies and typographical errors.

\section{Preliminaries}\label{sec:preliminaries}

In this section we let $\H:=\{z\in \C: \Im(z)>0\}$ denote the usual hyperbolic plane.

\subsection{The Weil representation}\label{Section: Weil representation}
The {\it metaplectic extension $\widetilde{\Gamma} \coloneqq \mathrm{Mp}_2(\Z)$ of $\SL_2(\Z)$} is defined as
\begin{equation*}
\widetilde{\Gamma} \coloneqq \left\{ (\gamma, \phi) \colon \gamma = \left(\begin{matrix}
a & b \\ c & d
\end{matrix}\right)\in \SL_2(\Z), \phi\colon \H \rightarrow \C \text{ holomorphic}, \phi^2(\tau) = c\tau+d  \right\}.
\end{equation*}
It is generated by $\widetilde{T} := \left(\left( \begin{smallmatrix} 1 & 1 \\ 0 & 1 \end{smallmatrix} \right),1\right)$ and $\widetilde{S} := \left(\left( \begin{smallmatrix} 0 & -1 \\ 1 & 0 \end{smallmatrix}\right) ,\sqrt{\tau}\right)$. By $\widetilde{\Gamma}_\infty$ we denote the subgroup generated by
$\widetilde{T}$.

Let $L$ be an even lattice of signature $(p,q)$ with quadratic form $Q$ and associated bilinear form $(\cdot,\cdot)$. Let $L'$ denote its dual lattice, $\C[L'/L]$ be the group ring of $L'/L$ with standard basis elements $\e_{\mu}$ for $\mu \in L'/L$, and let $\langle\cdot,\cdot\rangle$ be the sesquilinear form on $\C[L'/L]$ given by
\[\left\langle\sum_{\mu \in L'/L}a_\mu \e_\mu,\sum_{\mu \in L'/L}b_\mu \e_\mu\right\rangle=\sum_{\mu \in L'/L}a_\mu \overline{b_\mu}.\]

%natural bilinear form on $\C[L'/L]$ given by $\langle\e_{\mu},\e_{\nu}\rangle = \delta_{\mu,\nu}$, which is conjugate-linear in the second variable.

The {\it Weil representation} $\rho_L$ of $\widetilde{\Gamma}$ associated to $L$ acts on $\C[L'/L]$ by
\begin{align*}
\rho_L\left(\widetilde{T}\right)(\e_\mu) \coloneqq e(Q(\mu)) \e_\mu, \qquad
\rho_L\left(\widetilde{S}\right)(\e_\mu) \coloneqq \frac{e\left(\frac18(q-p)\right)}{\sqrt{|L'/L|}} \sum_{\nu \in L'/L} e(-(\nu,\mu)) \e_{\nu}.
\end{align*}
In particular, the element $\widetilde{Z}:=\widetilde{S}^2=\left(\left( \begin{smallmatrix} -1 & 0 \\ 0 & -1 \end{smallmatrix} \right),i\right)$ acts as $\rho_L\left(\widetilde{Z}\right)(\e_\mu)=e\left(\frac14(q-p)\right)\e_{-\mu}.$

The Weil representation $\rho_{L^-}$ associated to the lattice $L^- = (L,-Q)$ gives the \textit{dual Weil representation} associated to $L$.

\subsection{Harmonic Maass forms}\label{sec:HarmonicMaassForms}
Let $\kappa \in \frac{1}{2} \Z$. We define the {\it slash-operator} by
\[
f\mid_{\kappa,\rho_{L}}(\gamma,\phi) (\tau):= \phi(\tau)^{-2\kappa}\rho_{L}^{-1}(\gamma,\phi)f(\gamma\tau),
\]
where $f\colon \H \rightarrow \C[L'/L]$ and $(\gamma,\phi) \in \widetilde{\Gamma}$. In \cite{brfu04} Bruinier and Funke introduced the notion of weak Maass forms. A smooth function $f \colon \H \rightarrow \C[L'/L]$ is called a \emph{harmonic weak Maass form} of weight $\kappa$ with respect to $\rho_L$ if it is annihilated by the \emph{weight $\kappa$ Laplace operator}
\[
\Delta_{\kappa} \coloneqq -v^2\left(\frac{\partial^2}{\partial u^2} + \frac{\partial^2}{\partial v^2} \right) + i\kappa v\left(\frac{\partial}{\partial u} + i\frac{\partial}{\partial v} \right) \qquad (\text{where }\tau=u+iv),
\]
if it is invariant under the slash-operator $\mid_{\kappa,\rho_L}$, and if it has at most exponential growth at the cusp, that is, $f(\tau)=O(e^{Cv})$ as $v\to \infty$ for some constant $C>0$. We denote this space by $H_{\kappa,L}$. We let $M_{\kappa,L}^!$, $M_{\kappa,L}$ and $S_{\kappa,L}$ denote the subspaces of $H_{\kappa,L}$ of weakly holomorphic modular forms, holomorphic modular forms, and cusp forms, respectively.

\begin{remark}\label{rmk:weight}
  Throughout the paper, we assume that $2\kappa\equiv p-q$ (mod 4)\footnote{We will consider lattices $L$ of signature $(1,q)$ and forms of weight $\frac{1-q}{2}-2j$ with $j\in \N_0$, or $L^-$ of signature $(q,1)$ and forms of weight $2-\left(\frac{1-q}{2}-2j\right)$.}. This implies that every $f\in H_{\kappa,L}$ of the form $f=\sum_{\mu \in L'/L}f_\mu \e_\mu$ satisfies
$$f=f|_{\kappa,\rho_L}\widetilde{Z}=\sum_{\mu \in L'/L}f_{-\mu}\e_\mu,$$
hence $f_{\mu}=f_{-\mu}$ for all $\mu \in L'/L$.  
\end{remark}

The antilinear differential operator
\[
\xi_\kappa(f) \coloneqq 2iv^\kappa \overline{\frac{\partial f}{\partial \overline{\tau}}}
\]
maps
\[
\xi_\kappa \colon H_{\kappa,L}\longrightarrow M_{2-\kappa,L^-}^!,
\]
where $L^-$ denotes the lattice $L$ with the quadratic form $-Q$. 

Following Bruinier and Funke, we also consider the subspaces $H_{\kappa,L}^{\text{hol}}$ and $H_{\kappa,L}^{\text{cusp}}$ of harmonic weak Maass forms that map to holomorphic modular forms and cusp forms, respectively, under the $\xi$-operator.

Every $f\in H_{\kappa,L}^{\text{hol}}$ of weight $\kappa\in \frac{1}{2}\Z,\, \kappa\neq1$, splits as a sum $f=f^+ + f^-$ of a \emph{holomorphic} and a \emph{non-holomorphic part} $f^-$, having Fourier expansions of the form
\begin{equation}\label{eq:fourierexpansion}
\begin{split}
f^+(\tau) &= \sum_{\mu \in L'/L} \sum_{\substack{n\in \Z+Q(\mu)\\n \gg -\infty}} a_f^+(\mu,n) e(n\tau) \e_\mu, \\
f^-(\tau) &= \sum_{\mu \in L'/L} \left(a_f^-(\mu,0)v^{1-\kappa}+  \sum_{\substack{n\in \Z+Q(\mu)\\n <0}} a_f^-(\mu,n)\Gamma(1-\kappa, 4\pi |n| v) e(n\tau)\right) \e_\mu,
\end{split}
\end{equation}
where $\Gamma(s,x):= \int_x^\infty t^{s-1}e^{-t} dt$ denotes the incomplete Gamma function. The coefficients $a_f^\pm(\mu,n)$ are invariant under $\mu\mapsto -\mu$ by Remark \ref{rmk:weight}. When $f\in M^!_{\kappa,L}$ we have $f=f^+$ and we write $a_f^+(\mu,n)=:a_f(\mu,n)$.

\begin{definition}\label{def:rationa_principal_part}
    Let $f\in M^!_{\kappa,L}$ with Fourier coefficients $a_f(\mu,n)$. The finite sum
    $$\sum_{\mu \in L'/L} \sum_{\substack{n\in \Z+Q(\mu)\\n <0}} a_f(\mu,n) e(n\tau) \e_\mu$$
    is called the \emph{principal part} of $f$. We say that \emph{$f$ has rational principal part} if $a_f(\mu,n)\in \Q$ for all $\mu \in L'/L$ and all $n\in \Z+Q(\mu)$ with $n<0$.
\end{definition}

We also define the \emph{lowering} and \emph{raising operators}
\begin{equation}\label{eq:lowering_and_raising_operators}
    L_\kappa \coloneqq -2iv^2\frac{\partial}{\partial \overline{\tau}} \text{ and }R_\kappa = 2i\frac{\partial}{\partial \tau}+ \frac{\kappa}{v},
\end{equation}
which lower, respectively raise the weight of a smooth function transforming like a modular form of weight $\kappa$ for $\rho_L$ by two. For $n\in \N_0$ we also let $L_\kappa^n=L_{\kappa-2n+2}\circ \cdots \circ L_\kappa$ and $R_\kappa^n=R_{\kappa+2n-2}\circ \cdots \circ R_\kappa$ denote the corresponding iterated operators.

\subsection{Maass--Poincar\'{e} series}\label{Section: Maass--Poincare series}
Let $\kappa \in \frac{1}{2} \Z$ with $\kappa < 0$ satisfying $2\kappa\equiv p-q$ (mod 4), and denote by $M_{\mu, \nu}$ the usual $M$-Whittaker function (see \cite[Equation~13.1.32]{abramowitz1988handbook}). We define, for $s \in \C$ and $y \in \R \backslash \{0\}$,
\begin{align*}%\label{curlyM}
\mathcal{M}_{\kappa,s} (y) \coloneqq |y|^{-\frac{\kappa}{2}} M_{-\mathrm{sgn}(y) \frac{\kappa}{2}, s-\frac{1}{2}} (|y|).
\end{align*}
Following Bruinier \cite[Section 1.3]{brhabil}, for $\mu \in L'/L$ and $m \in \Z - Q(\mu)$ with $m > 0$ we define the vector-valued Maass--Poincar\'{e} series
\begin{align}\label{eq:PCseries}
F_{\mu, m, \kappa,s}(\tau) \coloneqq \frac{1}{2 \Gamma(2s)} \sum_{(\gamma,\phi) \in \widetilde{\Gamma}_\infty \backslash \widetilde{\Gamma}} \left(\mathcal{M}_{\kappa,s} (4 \pi m v) e(-mu) \e_\mu\right) \mid_{\kappa, \rho_L} (\gamma,\phi)(\tau).
\end{align}
The series converges absolutely for $\Re(s) > 1$, transforms like a modular form of weight $\kappa$ for $\rho_L$, and is an eigenform of the Laplace operator $\Delta_\kappa$ with eigenvalue $s(1-s)+(\kappa^2-2\kappa)/4$. Hence, the special value
\[F_{\mu,m,\kappa}(\tau):=F_{\mu, m, \kappa,1-\frac{\kappa}{2}}(\tau)\]
defines a harmonic weak Maass form in $H^{\text{cusp}}_{\kappa,L}$ whose Fourier expansion starts with 
\[F_{\mu,m,\kappa}(\tau)=e(-m\tau)(\e_{\mu}+\e_{-\mu})+O(1).\]
%for some constant $\mathfrak{c} \in \C[L'/L]$ at the harmonic point $s = 1-\frac{\kappa}{2}$. 
In particular, for $\kappa<0$, every harmonic Maass form $f \in H^{\text{cusp}}_{\kappa,L}$ with Fourier expansion~\eqref{eq:fourierexpansion} can be written as a linear combination
\begin{align}\label{Equation: basis}
f(\tau) = \frac{1}{2}\sum_{\mu \in L'/L}\sum_{\substack{m\in \Z-Q(\mu)\\m > 0}}a_f^+(\mu,-m)F_{\mu,m,\kappa}(\tau).
\end{align}

%\todo[inline]{I think we need to restrict to $p=1$ or impose $2\kappa\equiv p-1\pmod{4}$, otherwise the principal part might look different.} Seba: I added Remark \ref{rmk:weight}

The following lemma follows by induction from \cite[Proposition~3.4]{BEY21}.
\begin{lemma}\label{Lemma: R_k^n acting on F_m,h}
	For $j \in \N_0$ we have that
	\begin{equation*}
	R_\kappa^j \left(F_{\mu, m, \kappa,s}\right)(\tau) = (4\pi m)^j \frac{\Gamma\left(s+j+\frac{\kappa}{2}\right)}{\Gamma\left(s+\frac{\kappa}{2}\right)} F_{\mu, m,\kappa+2j,s}(\tau).
	\end{equation*}
\end{lemma}

We also recall the formula
$$R^j_\kappa=\sum_{r=0}^j(-1)^r{j \choose r}(\kappa +r)_{j-r}\Im(\tau)^{r-j}(4\pi D)^r,$$
where $D:= \frac{1}{2\pi i}\frac{\partial}{\partial \tau}$ and $(x)_s:=x(x+1)\cdots (x+s-1)$; see, for instance \cite[Lemma 5.4]{BFOR17}. Using $D(e(m\tau))=me(m\tau)$ and $(x)_s=(-1)^s(1-x-s)_s$, this implies
\begin{equation}\label{eq:raising_of_exponential}
\begin{split}
        R^j_{\kappa-2j}(e(m\tau)) %&=\left(\sum_{a=0}^r(-1)^a{r \choose a}(\kappa-2r +a)_{r-a}\Im(\tau)^{a-r}(4\pi m)^a \right) e(m\tau)\\
        &=(-1)^j\left(\sum_{r=0}^j{j \choose r}(j+1-\kappa)_{r}\Im(\tau)^{-r}(4\pi m)^{j-r} \right) e(m\tau).
\end{split}
\end{equation}
Note that $R^j_{\kappa-2j}(e(m\tau))$ is $e(m\tau)$ times a polynomial in $\Im(\tau)^{-1}$.

%{\claudia What is meant here? Note that $R^j_{\kappa-2j}(e(m\tau))$ is a polynomial in $\Im(\tau)^{-1}\cdot e(m\tau)$.?? Seba: Is it better now?}

\subsection{Operators on vector-valued modular forms and Rankin--Cohen brackets}\label{section:operatorsandrankin} 
For an even lattice $L$ we let $A_{\kappa,L}$ be the space of $\C[L'/L]$-valued smooth modular forms  of weight $\kappa$ with respect to the representation $\rho_L$. 

Let $K \subset L$ be a sublattice of finite index. Since $K \subset L \subset L' \subset K'$ it follows that $L/K \subset L'/K \subset K'/K$, hence there is a natural map $L'/K \rightarrow L'/L,\,  \mu \mapsto \bar{\mu}$. For $\mu \in K'/K$, $f \in A_{\kappa, L}$, and $g \in A_{\kappa,K}$ define
\begin{equation*}
(f_K)_\mu \coloneqq \begin{dcases}
f_{\bar{\mu}}& \text{ if } \mu \in L'/K, \\
0& \text{ if } \mu \not\in L'/K,
\end{dcases}\qquad 
\left(g^L\right)_{\bar{\mu}} \coloneqq \sum_{\alpha \in L/K} g_{\alpha + \mu},
\end{equation*}
where $\mu$ is a fixed preimage of $\bar{\mu}$ in $L'/K$. The proof of the following lemma is contained in \cite[Section~3]{BY09}.

\begin{lemma}\label{Lemma: restriction and trace maps}
	There are two natural maps
	\begin{equation*}
	\textup{res}_{L/K} \colon A_{\kappa, L} \rightarrow A_{\kappa, K}, \hspace{10pt} f \mapsto f_K,
\qquad\qquad
	\tr_{L/K} \colon A_{\kappa, K} \rightarrow A_{\kappa, L}, \hspace{10pt} g \mapsto g^L,
	\end{equation*}
	such that for any $f \in A_{\kappa,L}$ and $g \in A_{\kappa,K}$, we have $
	\langle f, g^L \rangle = \langle f_K, g \rangle$.
	\end{lemma}

Now let $K$ and $L$ be even lattices. For $n\in\N_0$ and functions $f \in A_{\kappa,K}$ and $g \in A_{\ell,L}$ with $\kappa,\ell \in \frac{1}{2} \Z$ we define the $n$-th \emph{Rankin--Cohen bracket} by
\begin{align*}\label{Definition: RC brackets}
[f,g]_n \coloneqq \sum_{\substack{r,s \geq 0\\ r+s=n}} (-1)^r \binom{\kappa+n-1}{s}\binom{\ell+n-1}{n-s}  (D^rf) \otimes (D^sg),
\end{align*}
where $D^rf =\frac{1}{(2\pi i )^r}\frac{\partial^rf}{\partial \tau^r}$ and the tensor product of two vector-valued forms $f = \sum_{\mu}f_\mu \e_\mu \in A_{\kappa,K}$ and $g = \sum_{\nu}g_\nu \e_\nu \in A_{\ell,L}$ is defined by 
\[
f \otimes g \coloneqq \sum_{\mu , \nu} f_\mu g_\nu \e_{\mu \oplus  \nu} \in A_{\kappa+\ell,K \oplus L}.
\]
%\todo[inline]{$\e_{\mu + \nu}=\e_{\mu} \otimes \e_{\nu}$ here?} Since the lattice is $K\oplus L$, I think the canonical basis is $\e_{\mu\oplus \nu}$.

% \begin{proposition}[{\cite[Proposition~3.6]{BEY21}}]\label{Proposition: lowering of RC brackets}
% 	Let $f \in H_{\kappa,K}$ and $g \in H_{\ell,L}$ be harmonic Maass forms. For $n \in \N_0$ we have
% 	\begin{align*}
% 	&(-4\pi)^n L_{\kappa+\ell+2n}[f,g]_n= \binom{\ell+n-1}{n} R_\kappa^n f \otimes L_\ell g + (-1)^n \binom{\kappa+n-1}{n}L_\kappa f \otimes R_\ell^n g.
% 	\end{align*}
% \end{proposition}
%\todo[inline]{Use polynomial binomial coefficients? Gamma-functions are not define for $\kappa=0,-1-,...$}

\subsection{Rational relations for spaces of cusp forms}\label{sec 
Rational relation}
A sequence
% \begin{equation}\label{eq rational relation}
%    \{\lambda(\mu,m)\}_{\mu \in L'/L,m\in \Q_{>0}}\subseteq \mathbb{Q} 
% \end{equation}
\begin{equation*}%\label{eq rational relation}
   \{\lambda(\mu,m)\}_{\mu \in L'/L,m\in \Z+Q(\mu), m>0}\subseteq \mathbb{Q} 
\end{equation*}
is called a \emph{rational relation} for~$S_{\kappa,L}$ if the following conditions are satisfied:
\begin{enumerate}
    \item For each~$\mu\in L'/L$ we have~$\lambda(-\mu,m)=\lambda(\mu,m)$,
    \item For each~$\mu\in L'/L$ we have~$\lambda(\mu,m)=0$ for all but finitely many~$m\in \Q_{>0}$,
    \item For every cusp form $f(\tau)=\sum_{\mu \in L'/L}\sum_{\substack{ m\in \Z+Q(\mu)\\ m>0}}a_f(\mu,m)e(m\tau)\e_{\mu}\in S_{\kappa,L} $ we have
    \[\sum_{\mu \in L'/L}\sum_{\substack{ m\in \Z+Q(\mu)\\ m>0}}\lambda(\mu,m)a_f(\mu,m)=0.\]
\end{enumerate}
Note that, when $S_{\kappa,L}=\{0\}$, condition (3) is satisfied trivially.

We recall the Bruinier--Funke pairing considered in \cite{brfu04}. Let $\kappa<0$. For a harmonic weak Maass form $h\in H^{\text{hol}}_{\kappa,L}$ and a cusp form $g\in S_{2-\kappa,L^-}$ it is defined as
\begin{equation}{\label{bfpairing_def}}
    \{ g,h\}:= \int_\mathcal{F} \langle g(\tau), (\xi_\kappa h) (\tau) \rangle v^{2-\kappa}\d\mu(\tau),
\end{equation}
where $\tau=u+iv$, $\d\mu(\tau):=\frac{\d u\d v}{v^2}$ and $\mathcal{F}$ denotes the standard fundamental domain for $\SL_2(\Z)$. By \cite[Proposition 3.5]{brfu04} we have 
\begin{equation}{\label{bfpairing_formula}}
    \{ g,h\} =\sum_{\mu\in L'/L}\sum_{\substack{n\in \Z+Q(\mu)\\n<0}} a_h^+(\mu,n) a_g(\mu,-n).
\end{equation}

As straightforward consequence of McGraw's rationality theorem \cite{McG03} we have

\begin{lemma}\label{lemma:mcgrawextended}
    Let $\kappa<0$ and suppose that $f\in M^!_{\kappa,L}$ has rational principal part (see Definition \ref{def:rationa_principal_part}). Then all Fourier coefficients of $f$ are rational.
\end{lemma}

Using this lemma we show that certain linear combinations of Maass--Poincar\'{e} series have rational coefficients. 

\begin{proposition}\label{prop:rational_relations_define_weakly_holomorphic_forms}
    Assume $\kappa<0$ and let $\{\lambda(\mu,m)\}$ be a rational relation for $S_{2-\kappa,L^-}$. Then
    $$f(\tau):=\sum_{\mu \in L'/L}\sum_{\substack{m\in \Z-Q(\mu)\\m>0}}
    \lambda(\mu,m)F_{\mu,m,\kappa}(\tau)\in M^!_{\kappa,L}.$$
    % {\claudia: Do we want to take $\frac12$ of this? (Since we get twice something in the proof, this should be fine.} Seba: I think it is OK as it is.
    Moreover, $f$ has rational Fourier coefficients.
    %\todo{$k=\kappa$ everywhere?} True! corrected :)
\end{proposition}
\begin{proof}
Clearly $f\in H^{\text{cusp}}_{\kappa,L}$ and it has Fourier expansion
\begin{equation}\label{eq:fourier_expansion_rational_relation}
\begin{split}
    f(\tau)&=\sum_{\mu \in L'/L}\sum_{\substack{m\in \Z-Q(\mu)\\m>0}} \lambda(\mu,m)e(-m\tau)(\e_{\mu}+\e_{-\mu}) + O(1)\\
    &=2\sum_{\mu \in L'/L}\sum_{\substack{m\in \Z-Q(\mu)\\m>0}} \lambda(\mu,m)e(-m\tau)\e_{\mu} + O(1).
    \end{split}
\end{equation}
Let $g(\tau)=\sum_{\mu \in L'/L}\sum_{\substack{m\in \Z-Q(\mu)\\ m > 0}}a_g(\mu,m)e(m\tau)\e_{\mu}\in S_{2-\kappa,L^-}$. Applying the Bruinier--Funke pairing \eqref{bfpairing_def} to $f$ and $g$ and using \eqref{bfpairing_formula} gives
\[
 \{g, f\}= 2\sum_{\mu\in L'/L}\sum_{\substack{m\in \Z-Q(\mu)\\ m>0}} \lambda(\mu,m)a_g(\mu,m).
\]
%\todo[inline]{Is there a 2 missing?} Added!
Since $\lambda$ is a rational relation for $S_{2-\kappa,L^-}$ the right hand side vanishes for every cusp form $g\in S_{2-\kappa,L^-}$.
Thus $f\in M^!_{\kappa,L}$ by \cite[Theorem 1.1]{brfu04}. Moreover, since all $\lambda(\mu,m)$ are rational, the principal part of $f$ is rational. Lemma \ref{lemma:mcgrawextended} implies the desired result.
\end{proof}

% For a  prime discriminant $D<0$ and the lattice $L$ of determinant $|D|$, the assignment
% $$\sum_{\mu\in L'/L}f_{\mu}(\tau)\e_{\mu} \mapsto \sum_{\mu\in L'/L}f_{\mu}(|D|\tau)$$
% defines a linear map~$S_{k,L}\to S^+_{k}(\Gamma_0(|D|),\chi_D)$ (see \cite{bruinierbundschuh}). This implies the following lemma.

% \begin{lemma}\label{lem vector-valued to scalar mfs}
% Assume  $D<0$ is a prime discriminant and the even lattice $L$ has determinant $|D|$. Then, for every rational relation~$\{\lambda(t)\}_{t\in \N}$ for~$S^+_{k}(\Gamma_0(|D|),\chi_D)$ the sequence~$\lambda'(m,\mu)=\lambda(m|D|)$ defines a relation for~$S_{k,L}$ (that is independent of $\mu\in L'/L$).
% \end{lemma}

\section{Quadratic spaces of signature $(1,q)$ and hyperbolic geometry}\label{sec:quad_spaces_(1,q)_hyp_geometry}

\subsection{Hyperbolic models}\label{sec:hyperbolic-model}
Let $q\geq 2$ be an integer and $(V,Q)$ a real quadratic space of signature $(1,q)$ with associated bilinear form $(\cdot,\cdot)$.  By Witt's theorem %\cite[Lemma 1.3]{shimurabook}, 
there are isotropic vectors $\ell,\ell'\in V$ with $(\ell,\ell')=1$ and a negative definite subspace $W\subseteq V$ such that 
\[V=(\R \ell\oplus \R\ell')\oplus^\perp W\]
where $\oplus^\perp$ denotes orthogonal sum. A general element $X\in V$ will be written as $X=a\ell'+c\ell +b$ with $a,c\in \R$ and $b\in W$. We also write $a_X\coloneqq(X,\ell)=a$.

Let $\mathrm{Gr}(V)$   denote the Grassmannian of positive definite lines in $V$. We can identify $\Gr(V)$ with
$$V_1:=\{X_1\in V: (X_1,X_1)=1, (X_1,\ell)>0\}$$
via %$V_1\to \Gr(V), 
$X_1\mapsto \R X_1$. We further identify $V_1$ with the \emph{upper half-space model}
$$\H_W:= W\times \R_{>0}=\{P=(x,y):x\in W, y>0\}$$
via the map $\H_W \to V_1$ sending $P=(x,y)$ to
\begin{equation}\label{def:X_1(P)}
X_1(P):=\frac{1}{\sqrt{2}y}x+\frac{1}{\sqrt{2}y}\ell'+\left(\frac{y^2-Q(x)}{\sqrt{2}y}\right)\ell\in V_1.
\end{equation}

Choosing an orthogonal basis $\{w_1,\ldots,w_{q-1}\}$ of $W$ with $Q(w_i)=-1$ for all $i$, gives the \emph{classical model}
\begin{equation}\label{eq:classical_model}
    \H_W\simeq  \H^q:=\left\{(x_1,\ldots,x_{q-1},y)\in \R^{q}:y>0\right\}.
\end{equation}
For two points $P=(x,y),\, P'=(x',y')\in \H_W$, the hyperbolic distance $d(P,P')$ satisfies
\begin{equation}\label{eq:cosh_d_H_W}
    \cosh \, d(P,P')=\frac{|Q(x-x')|+y^2+(y')^2}{2yy'},
\end{equation}
which in the classical model~\eqref{eq:classical_model} becomes
\[\cosh \, d( (x_1,\ldots,x_{q-1},y),(x_1',\ldots,x_{q-1}',y'))=\frac{\sum_{i=1}^{q-1}(x_i-x_i')^2+y^2+(y')^2}{2yy'}.\]

\begin{example}\label{ex:V_signature_(1,2)}
In the case $q=2$ we can take %and $W=\R w_1$ we have $\mathcal{C}_W=\R \oplus \R w_1$ and $\mathcal{V}_W=\R$, hence the matrix representation of $V$ is
    \begin{equation*}
 V=\left\{X=\begin{pmatrix}
    b & c \\ -a & -b
\end{pmatrix}:a,b,c\in \R\right\} \text{ with }Q(X)=ac-b^2
\end{equation*}
and $W=\left\{\left(\begin{smallmatrix}
    b & 0 \\ 0 &- b
\end{smallmatrix}\right):b\in \R\right\}\simeq \R$ via $\left(\begin{smallmatrix}
    b & 0 \\ 0 & -b
\end{smallmatrix}\right)\mapsto -b$, so that $\H_W\simeq \R \times \R_{>0}\simeq \H$. Choosing
$$\ell:=\left(\begin{matrix}
0 & 1 \\ 0 & 0
\end{matrix}\right), \quad \ell':=\left(\begin{matrix}
0 & 0 \\ -1 & 0
\end{matrix}\right)$$
we get
$$X_1(\tau)=\frac{1}{\sqrt{2}y}\begin{pmatrix}
    -x & x^2+y^2 \\ -1 & x
\end{pmatrix} \quad \text{ for }\tau=x+iy\in \H.$$
Note that every $g\in\SL_2(\R)$ acts as an isometry on $V$ via $X\mapsto gXg^{-1}$. If we also consider the action of $\SL_2(\R)$ on $\H$ by  linear fractional transformations, then it is easy to check that the map $\tau \mapsto X_1(\tau)$ is $\SL_2(\R)$-equivariant.
\end{example}

%\todo[inline]{$\H_W^2$ and $\H_W^3$ are not defined} corrected

\begin{example}\label{ex:V_signature_(1,3)} 
In the case $q=3$ %and $W=\R w_1+\R w_2$ we have $\mathcal{C}_W=\R \oplus \R w_1\oplus \R w_2 \oplus \R w_1w_2 $ and $\mathcal{V}_W=\R\oplus \R w_1\simeq \C$, hence the matrix representation of $V$ becomes isomorphic to
we can take
    \begin{equation*}
V= \left\{X=\begin{pmatrix}
    c & b \\ \overline{b} & a
\end{pmatrix}:a,c\in \R,b\in \C \right\} \text{ with }Q(X)=ac-|b|^2  
\end{equation*}
and $W=\left\{\left(\begin{smallmatrix}
    0 & b \\ \overline{b} & 0
\end{smallmatrix}\right):b\in \C\right\}\simeq \C$ via $\left(\begin{smallmatrix}
    0 & b \\ \overline{b} & 0
\end{smallmatrix}\right)\mapsto b$, so that $\H_W\simeq \C \times \R_{>0}$ can be identified with the subset $\{z+rj:z\in \C,r>0\}$ of Hamilton's quaternions $\R[i,j,k]$. Choosing
$$\ell:=\left(\begin{matrix}
1 & 0 \\ 0 & 0
\end{matrix}\right), \quad \ell':=\left(\begin{matrix}
0 & 0 \\ 0 & 1
\end{matrix}\right)$$
we get
$$X_1(P)=\frac{1}{\sqrt{2}r}\begin{pmatrix}
    |z|^2+r^2 & z \\ \overline{z} & 1
\end{pmatrix} \quad \text{ for }P=z+rj \text{ where } z\in \C,r>0.$$
In this setting every $g\in\SL_2(\C)$ acts as an isometry on $V$ via $X\mapsto gX\overline{g}^t$. If we also consider the action of $\SL_2(\C)$ on $\{P=z+rj:z\in \C,r>0\}$ given by
\[\begin{pmatrix}
    \alpha & \beta \\ \gamma & \delta
\end{pmatrix} P:=(\alpha P+\beta)(\gamma P+\delta)^{-1}\]
(where sums, products and inverses are taken in the division algebra $\R[i,j,k]$) then one can check that the map $P \mapsto X_1(P)$ is $\SL_2(\C)$-equivariant (see \cite[Chapter 9, Proposition 1.2]{EGM98} or \cite[Section 4]{HIPS25}).
\end{example}

\subsection{Points and geodesic hypersurfaces}\label{sec:points_and_geodesic_hypersurfaces}

Every $X=a\ell'+c\ell+b\in V$ with  $Q(X)>0$ defines a point $P_X\in \H_W$ given by
\begin{equation*}%\label{eq:def_P_X}
P_X:=\left(\frac{b}{a},\frac{\sqrt{Q(X)}}{|a|}\right).    
\end{equation*}
A straightforward computation using \eqref{def:X_1(P)} gives the identity
\begin{equation}\label{eq:X_1(P_X)}
X_1(P_X)=\frac{\sign(a)}{\sqrt{2Q(X)}}X.    
\end{equation}
In particular, if $X\in V_1$ then $X_1(P_X)=X$.

Every $Y=a\ell'+c\ell+b\in V$ with  $Q(Y)<0$ defines a geodesic hypersurface $S_Y\subset  \H_W$ given by
\begin{equation}\label{eq:def_S_Y}
    S_Y:=\{(x,y)\in \H_W: a(y^2-Q(x))+(b,x)+c=0\}.
\end{equation}
Note that for $P=(x,y)\in \H_W$ and $Y=c\ell +a\ell'+b\in V$ we have
\begin{equation}\label{eq:(X_1(P),Y)}
(X_1(P),Y)=\frac{a(y^2-Q(x))+(b,x)+c}{\sqrt2\,y}.
\end{equation}
Hence, under the identification $\Gr(V)\simeq \H_W$, $S_Y$ corresponds to
\begin{equation}\label{eq:Y_perp}
Y^\perp:=\{z\in \Gr(V): (X,Y)=0 \text{ for all }X\in z\}.
\end{equation}
It is easy to see that $\H_W\setminus S_Y$ has two connected components. When $a\neq 0$, $S_Y$ is a hemisphere of real dimension $q-1$ in $\H_W$ of radius $\frac{\sqrt{|Q(Y)|}}{|a|}$ centered at the point $(b/a,0)\in W\times \{0\}$. In this case $\H_W\setminus S_Y$ has one bounded component and one unbounded component with respect to the Euclidean distance in $W\times \R$. When $a=0$, $S_Y$ is a vertical half-hyperplane.

%If $a>0$ we parametrize $S_Y$ so that the associated normal vector field points to the bounded component of $\H_W\setminus S_Y$. If $a<0$ we take a parametrization with normal vector field pointing towards the unbounded component of $\H_W\setminus S_Y$. This defines an orientation of $S_Y$.

%$$P_X:=\left(\frac{b}{a},\frac{\sqrt{Q(X)}}{|a|}\right).$$

%Given $Y\in V$ with $Q(Y)<0$ and $a_Y> 0$, and a point $P\in \H_W\setminus S_Y$, we define $\sign_{Y}(P)$ as 1 if $P$ is in the bounded component of $ \H_W\setminus S_Y$, and as $-1$ otherwise. If $a_Y<0$, we define $\sign_{Y}(P)$ as $-1$ if $P$ is in the bounded component of $ \H_W\setminus S_Y$, and $1$ otherwise. Then $\sign_{-Y}(P)=-\sign_{Y}(P)$. 

%\subsection{Orientation of geodesic hypersurfaces}\label{sec:orientation}

% Note that for $P=(x,y)\in \H_W$ and $Y=c\ell +a\ell'+b\in V$ we have
% \begin{equation}\label{eq:(X_1(P),Y)}
% (X_1(P),Y)=\frac{a(y^2-Q(x))+(b,x)+c}{\sqrt2\,y}.
% \end{equation}

\begin{example}
    Let $V$ and $W\simeq \R$ be as in Example \ref{ex:V_signature_(1,2)} and $Y=\begin{pmatrix}
        b & c \\ -a & -b
    \end{pmatrix}\in V$ with $Q(Y)=ac-b^2<0$. Under $\H_W\simeq \H$ we identify $S_Y$ with the geodesic curve 
    $$\{\tau\in \H:a|\tau|^2+2b\Re(\tau)+c=0\}\subset \H.$$
If $a\neq 0$ then $S_Y$ is a semi-circle. If $a=0$ then $S_Y$ is the vertical semi-line of equation $\Re(\tau)=-\frac{c}{2b}$. 
\end{example}

For every $Y\in V$ with $Q(Y)<0$ we define 
the \emph{positive side} (resp.~\emph{negative side}) \emph{of $S_Y$ with respect to $Y$} as the component of $\H_W\setminus S_Y$ on which $(X_1(P),Y)>0$ (resp. $(X_1(P),Y)<0$). For $P\in \H_W$, define
\begin{equation*}
\sign_Y(P):=\sign((X_1(P),Y))
\end{equation*}
with the convention that $\sign(0):=0$. Then the two connected components of $\H_W\setminus S_Y$ are $\{P\in \H_W:\sign_Y(P)=1\}$ and $\{P\in \H_W:\sign_Y(P)=-1\}$. Note that
\[
\sign_{-Y}(P)=-\sign_Y(P).
\]
If $a_Y>0$, the set $\{P\in \H_W:\sign_Y(P)=1\}$ is the unbounded component of $\H_W\setminus S_Y$; if $a_Y<0$, it is the bounded component.

\begin{lemma}\label{l:sinh-distance-point-to-hypersurface}
Let $X,Y\in V$ with $Q(X)>0$ and $Q(Y)<0$. Let $P_{X}\in \H_W$ and $S_{Y}\subset \H_W$ be the associated  point and geodesic hypersurface, respectively. Then, the hyperbolic distance between $P_{X}$ and $S_{Y}$ satisfies the formula
$$\sinh  d(P_{X},S_{Y})=\frac{(X,Y)\cdot \sign_{Y}(P_{X})\cdot \sign(a_X)}{2\sqrt{Q(X)|Q(Y)|}}.$$
\end{lemma}
\begin{proof} 
Using \eqref{eq:X_1(P_X)}, we are left to prove
$$\sinh  d(P_{X},S_{Y})=\frac{|(X,Y)|}{2\sqrt{Q(X)|Q(Y)|}}.$$
Substituting $X$ and $Y$ by $\frac{1}{\sqrt{Q(X)}}X$ and $\frac{1}{\sqrt{|Q(Y)|}}Y$, respectively, we can assume $Q(X)=1$ and $Q(Y)=-1$. Choose $t>0$ such that $t^{-1}-t=(X,Y)$.
Then $Q(t\ell'+t^{-1}\ell)=1$, $Q(\ell'-\ell)=-1$ and $(t\ell'+t^{-1}\ell,\ell'-\ell)=(X,Y)$. This implies that there exists $g\in \O(V)$ such that $g(X)=t\ell'+t^{-1}\ell$ and $g(Y)=\ell'-\ell$. Since $P_{g(X)}=(0,t^{-1})$ and $S_{g(Y)}=\{(x,y)\in \H_W: y^2-Q(x)=1\}$, we have (by \eqref{eq:cosh_d_H_W})
$$\cosh\, d(P_{X},S_{Y})=\cosh\, d(P_{g(X)},S_{g(Y)})=\cosh\,d((0,t^{-1}),(0,1))=\frac{1}{2}(t^{-1}+t).$$
This implies
$$\sinh\, d(P_{X},S_{Y})=\frac{1}{2}|t^{-1}-t|=\frac{|(X,Y)|}{2\sqrt{Q(X)|Q(Y)|}}.$$
This proves the lemma.
\end{proof}

\begin{corollary}\label{cor-distance}

Let $(x,y)\in \H_W$ and $Y=c\ell +a\ell'+b\in V$ with $Q(Y)<0$. Then we have 
\[\sinh d((x,y),S_Y)=\frac{|a(y^2-Q(x))+(x,b)+c|}{2y\sqrt{|Q(Y)|}}.
\]
\end{corollary}
\begin{proof}
    Let $P=(x,y)$ and  $X=X_1(P)$. Then $Q(X)=\frac{1}{2}$, $P_X=P$ and the result follows from Lemma \ref{l:sinh-distance-point-to-hypersurface} and \eqref{eq:(X_1(P),Y)}. 
\end{proof}

\begin{corollary}\label{cor:1/cosh}
Let $z\in \Gr(V)$ correspond to $P\in \H_W$, and let $Y\in V$ with $Q(Y)=-m$ for some $m>0$. Then
$$\frac{m}{|Q(Y_{z^\perp})|}=\frac{1}{\cosh \, d(P,S_Y)^2},$$
where $Y_z$ and $Y_{z^\perp}$ denote the orthogonal projections of $Y$ to $z$ and $z^\perp$, respectively.
\end{corollary}
\begin{proof}
From $Y=Y_z+Y_{z^\perp}$ we get $-m=Q(Y)=Q(Y_z)+Q(Y_{z^\perp})$. Also, if we write $z=\R X_1$ with $X_1\in V_1$, then $P=P_{X_1}$ and $Y_z=(X_1,Y)X_1$, so $Q(Y_z)=\frac{1}{2}(X_1,Y)^2$. From Lemma \ref{l:sinh-distance-point-to-hypersurface} we get
$$\sinh  d(P,S_{Y})^2=\frac{(X_1,Y)^2}{2|Q(Y)|}=\frac{Q(Y_z)}{m}=-1-\frac{Q(Y_{z^\perp})}{m}.$$
This implies
$$\cosh \, d(P,S_Y)^2=1+\sinh \, d(P,S_Y)^2=\frac{|Q(Y_{z^\perp})|}{m},$$
as desired.
\end{proof}

\section{A Siegel theta lift}\label{sec:Siegel}

\subsection{Siegel theta functions in signature $(1,q)$}\label{section theta functions}
We let $L$ be an even lattice of signature $(1,q)$. By $\mathrm{Gr}(L)=\Gr(V)$ we denote the Grassmannian of positive definite $1$-dimensional subspaces of $V=L\otimes \mathbb R$. 

The Siegel theta function associated with $L$ is defined by
\begin{equation*}%\label{eq:Siegel_theta_L}
\Theta_{L}(\tau,z) = \Im(\tau)^{q/2}\sum_{\gamma \in L'/L}\sum_{X \in L+\gamma}e\left(Q(X_{z})\tau+Q(X_{z^{\perp}})\overline{\tau}\right)\e_{\gamma},
\end{equation*}
where $\tau \in \H$ and $z \in \Gr(L)$, and $X_{z},X_{z^\perp}$ denote the orthogonal projections of $X$ to $z$ and $z^\perp$, respectively. The Siegel theta function transforms like a modular form of weight $\frac{1-q}{2}$ for $\rho_{L}$ in $\tau$ (see \cite[Theorem~4.1]{borcherds}) and is invariant in $z$ under the subgroup of the orthogonal group $\mathrm{O}(L)$ fixing $L'/L$. 

\begin{definition}\label{def:special_point}
    An element $z\in \mathrm{Gr}(L)$ is called a \emph{special point} if it is defined over $\mathbb Q$, that is, if $z=\R X$ for some $X \in L\otimes \mathbb Q$.
\end{definition}
%\todo[inline]{This definition only works if $p=1$, move into the next section?} Done!

\subsection{Siegel theta lift in signature $(1,q)$}\label{sec: Siegel theta}

Let $q\geq 2$, let $(L,Q)$ be as in the previous section, and put $\kappa:=\frac{1-q}{2}$.
%We now let $(L,Q)$ be an even lattice of signature $(1,q)$ with $q\geq 2$ and put $\kappa=(1-q)/2.$ %\todo[inline]{Claudia: Jan uses $k$ here instead of $\kappa$. I am okay with both} 
Let $j\in \N_0$ and $f\in H^{\text{cusp}}_{\kappa-2j,L}$. Following \cite{borcherds} we define for $z\in \Gr(L)$ the regularized theta lift
\[
\Psi^j(f,z)=\lim_{T\to \infty}\int_{\mathcal{F}_T} \langle R^j_{\kappa-2j}f(\tau), \Theta_L(\tau,z)\rangle \Im(\tau)^\kappa \d\mu(\tau),
\]
where $\mathcal{F}_T$ denotes the standard truncated fundamental domain for $\SL_2(\mathbb Z)$. When $j=0$ we simply write $\Psi^0(f,z)=:\Psi(f,z)$.

Following \cite[Section 2.2]{brhabil}, for each $\mu\in L'/L$ and $m\in \Z-Q(\mu)$ with $m>0$ we  consider the lift of the Maass--Poincar\'{e} series $F_{\mu,m,\kappa,s}$ (defined in \eqref{eq:PCseries}) 
\begin{equation}\label{eq:Phi_mu_m}
    \Phi_{\mu,m} \left(z,s\right):=\lim_{T\to \infty}\int_{\mathcal{F}_T} \left\langle F_{\mu,m,\kappa,s}\left(\tau\right),\Theta_L(\tau,z) \right\rangle  \Im(\tau)^{\kappa}\d\mu(\tau).
\end{equation}
By \cite[Proposition 2.8]{brhabil} this regularized integral converges when $s\in \C$ satisfies $\Re(s)>\max\left\{1, \frac{1+q}{4}\right\}$ and $z\in \Gr(L)\setminus H(\mu,m)$,
where
\begin{equation*}%label{eq:def_H(mu,m)}
H(\mu,m) := \bigcup_{\substack{Y \in L+\mu\\Q(Y)=-m}} Y^\perp
\end{equation*}
and $Y^\perp$ is defined in \eqref{eq:Y_perp}. Moreover, by \cite[Theorem 2.12, $(i)$]{brhabil} the function $\Phi_{\mu,m} \left(z,s\right)$ %has analytic continuation to $\Re(s)>1,s\neq \frac{1+q}{4}$ and this extension 
is real-analytic in $z\in  \Gr(L)\setminus H(\mu,m)$. 

For $f\in H^{\text{cusp}}_{\kappa-2j,L}$, it follows from \eqref{Equation: basis} and Lemma \ref{Lemma: R_k^n acting on F_m,h} that the theta lift $\Psi^j(f,z)$ is real-analytic outside the union of the sets $H(\mu,m)$ for those $\mu\in L'/L$ and $m>0$ with $a_f^+(\mu,-m)\neq 0$, where  $a_f^+(\mu,-m)$ is the Fourier coefficient of $f$ as in \eqref{eq:fourierexpansion}.

A direct application of \cite[Theorem 2.14]{brhabil} gives the following representation of $\Phi_{\mu,m}(z,s)$ as an infinite series.

\begin{theorem}{\cite[Theorem 2.14]{brhabil}}\label{thm:Siegel lift of PC series}
    For every $\mu\in L'/L$, $m\in \Z-Q(\mu)$ with $m>0$, $z\in \Gr(L)\setminus H(\mu,m)$ and $s\in \C$ with $\Re(s)>\max\left\{1,\frac{1+q}{4}\right\}$, we have
    \begin{align*}
        \Phi_{\mu,m}(z,s)&= 2\left(4\pi m \right)^{s-\frac{1-q}{4}} \frac{\Gamma\left(s+\frac{q-3}{4}\right)}{\Gamma\left(2s\right)}
        \\
        &\quad\times\sum_{\substack{Y\in L+\mu\\Q(Y)=-m}}\left(4\pi |Q(Y_{z^\perp})|\right)^{\frac{3-q}{4}-s} \, _2F_1\left(s+\frac{q-3}{4}, s-\frac{q-1}{4},2s;\frac{m}{|Q(Y_{z^\perp})|}\right).
    \end{align*}
    The series on the right hand side converges normally for $z\in \Gr(L)\setminus H(\mu,m)$. 
\end{theorem}
%Note that $1-\frac{\kappa}{2}= \frac{q+3}{4}$.

\subsection{Finite sum representation}\label{sec:finite_sum_resps}

As in the previous section, let $(L,Q)$ be an even lattice of signature $(1,q)$ with $q\geq 2$ and $\kappa=\frac{1-q}{2}$. We further assume that $L$ splits a hyperbolic plane over $\Z$. That is,  there exist  isotropic vectors $\ell,\ell'\in L$ with $(\ell,\ell')=1$ and a negative definite sublattice $K$ of rank $q-1$ of $L$ such that
\begin{equation*}%\label{eq:L_Lorentzian_lattice}
    L=(\Z \ell \oplus \Z \ell') \oplus^{\perp} K.
\end{equation*} 
Note that $L'/L\simeq K'/K$. 

Let $V:=L\otimes \R= (\R \ell \oplus \R \ell') \oplus^{\perp} W$ where $W:=K\otimes \R$. As in Section \ref{sec:hyperbolic-model} we identify $\Gr(L)=\Gr(V)$ with
$$V_1:=\{X_1\in V: (X_1,X_1)=1, (X_1,\ell)>0\}$$
and also identify $V_1$ with $\H_W= W\times \R_{>0}=\{(x,y):x\in W, y>0\}$.

Let $j\in \N_0$ and $f\in M^!_{\kappa-2j,L}$. Since $L$ has signature $(1,q)$, the regularized theta lift $\Psi^j(f,(x,y))$ defines a continuous function on $(x,y)\in \H_W$ (see, e.g., \cite[Theorem 6.2]{borcherds}). The purpose of this section is to prove a finite sum representation for $\Psi^j(f,(x,y))$, valid for all $(x,y)\in \H_W$, that generalizes \cite[Theorem 3.3]{BS21}. We remark that \cite[Theorem 3.3]{BS21}, corresponding to the case $j=0$, is stated in the more general case where $f\in H^{\text{cusp}}_{\kappa,L}$. In that case, the resulting formula for $\Psi(f,(x,y))$ has an extra sum that is infinite, coming from the non-holomorphic coefficients of $f$.

Recall that for every $r\in \N_0$ the $r$-th Bernoulli polynomial $B_r(x)$ is defined by the generating series
\[\frac{ze^{xz}}{e^z-1}=\sum_{r=0}^\infty \frac{B_r(x)}{r!}z^r \qquad (\text{for }z\in \C \text{ with }|z|<2\pi).\]
The first few are $B_0(x)=1$, $B_1(x)=x-\frac{1}{2}$ and $B_2(x)=x^2-x+\frac{1}{6}$. In general, $B_r(x)$ is a polynomial of degree $r$ with rational coefficients. Following \cite[p.~531]{borcherds} (see also \cite[p.~64]{brhabil}) we consider for $r\geq 2$ the unique 1-periodic function $\mathbb{B}_{r}:\R\to \R$ satisfying $\mathbb{B}_r(x)=B_r(x)$ for $0\leq x<1$. We have
\begin{equation*}
    \mathbb{B}_{r}(x)=-r!\sum_{n\in \Z\setminus \{0\}}\frac{e(nx)}{(2\pi i n)^r}
\end{equation*}
(see \cite[p.~37]{EMT81}). In particular, for every integer $r\geq 0$ we have
\begin{equation}\label{eq:formula_B_2+2r}
    (-1)^r\frac{(2\pi)^{2+2r}}{(2+2r)!}\mathbb{B}_{2+2r}(x)=\sum_{n\in \Z \setminus \{0\}}\frac{e(nx)}{n^{2+2r}}.
\end{equation}

%\todo[inline]{Claudia: I think this is $2\sum_{n=1}^\infty\frac{\cos(2\pi n x)}{n^{2+2a}}$ (Later, we can pair $Y$ and $-Y$ to get this I think}

Our finite sum representation of the lift $\Psi^j(f,(x,y))$ for $f\in M^!_{\kappa-2j,L}$ is given in the following theorem.

\begin{theorem}\label{thm:finite_sum_representation}
For $j\in \N_0$, $f\in M^!_{\kappa-2j,L}$ with Fourier expansion
\begin{equation*}%\label{eq:Fourier_exp_f}
f(\tau)= \sum_{\mu \in L'/L}\sum_{\substack{m \in \Z+Q(\mu) \\ m \gg -\infty}}a_{f}(\mu,m)e(m\tau)\e_{\mu}
\end{equation*}
and $(x,y)\in \H_W$ we have
    \begin{equation}\label{eq:finite_sum_representation}\begin{split}
    & \frac{1}{(4\pi)^{j+1}}\Psi^j(f,(x,y))\\
   &= \frac{y}{\left(\frac{q-1}{2}+j\right)(j+1)}\sum_{Y\in K'}|Q(Y)|^{j+1} a_{f}(Y,Q(Y))\\
   &\quad +  j! \sum_{Y \in K'}  a_f(Y,Q(Y))\sum_{r=0}^j\frac{(j+1-\kappa)_{r}|Q(Y)|^{j-r}}{(j-r)!(2+2r)!y^{1+2r}}  \mathbb{B}_{2+2r}\left((Y,x)\right) \\
 & \quad  +\frac{(-1)^j}{\sqrt{\pi}} \sum_{\substack{Y=a\ell'+c\ell+b\in L'\\ Q(Y)<0\\ a(y^2-Q(x))+(b,x)+c>0 >a}}a_{f}(Y,Q(Y)) \sum_{r=0}^j
   {j \choose r}(j+1-\kappa)_{r}
  Q(Y)^{j-r}\\
  & \qquad \qquad \qquad  \qquad \times \Gamma\left(-\frac{1}{2}-r\right)\frac{1}{ 4^r y^{1+2r}}(a(y^2-Q(x))+(b,x)+c)^{1+2r}.
  \end{split}
\end{equation}
Moreover, all the sums appearing in this formula are finite sums.
%\todo[inline]{ Is there a $(-1)^j$ in front of the last summand? See $\beta$ later} Corrected
\end{theorem}
%\todo[inline]{In the proof we get $(-1)^r (1+j-\kappa)_r$. This is $(\kappa-j-r)_r$, right? (not $(j+r-\kappa)_r$ in the second term of the sum above} Corrected!

\begin{remark}\label{rmk:Gamma(-1/2-r)}
For all integers $r\geq 0$ one has $\Gamma\left(-\frac{1}{2}-r\right)=(-1)^{r+1}\frac{2^{r+1}\sqrt{\pi}}{(2r+1)!!}$ where $(2r+1)!!=(2r+1)(2r-1)\cdots 3 \cdot 1$.    
\end{remark}

%\todo[inline]{1) We use $a$ twice here (once when defining $Y$, but also as a summation index; 2) Does the theorem hold for all $(x,y)$ directly or should we say something about analytic cont.?}

% \medskip

% We could rewrite the above formula using
% $$(-1)^{j-r}{j \choose r}(\kappa-j-r)_{r}\Gamma\left(-\frac{1}{2}-r\right)\frac{1}{4^r}=(-1)^{j+r+1}\,\frac{2\sqrt{\pi}}{(2r+1)!}\,\frac{j!}{(j-r)!}\,(j+1-\kappa)_r.$$

For the proof of Theorem \ref{thm:finite_sum_representation} we first establish the Fourier expansion for $\Psi^j(f,(x,y))$ for sufficiently large $y$.

\begin{proposition}\label{prop:Fourier_expansion_lift}
Let $f\in M_{\kappa-2j,L}^!$ be as in Theorem \ref{thm:finite_sum_representation} and let $m_0\in \Q$, $m_0>0$, be such that $a_f(\mu,-m)=0$ for all $m>m_0$. Then for all $(x,y)\in\H_W$ with $y>\sqrt{2m_0}$ we have
\begin{equation}\label{eq:final_fourier_expansion_lift}
\begin{split}
  & \frac{1}{(4\pi )^{j+1}} \Psi^j(f,(x,y)) \\
  &= \frac{y}{\left(\frac{q-1}{2}+j\right)(j+1)}\sum_{Y\in K'}|Q(Y)|^{j+1} a_{f}(Y,Q(Y))\\
   &\qquad +  j! \sum_{Y \in K'}  a_f(Y,Q(Y)) \sum_{r=0}^j \frac{(j+1-\kappa)_{r}}{(j-r)!(2+2r)! y^{1+2r}} |Q(Y)|^{j-r} \mathbb{B}_{2+2r}\left((Y,x)\right).
\end{split}
\end{equation}
%\todo[inline]{Power of $Q(Y)$ in the last line should be $j-a$?} You are right!
\end{proposition}
\begin{proof}
We start by using \eqref{Equation: basis} and Lemma \ref{Lemma: R_k^n acting on F_m,h}  to write 
\begin{align}\label{eq:raising_of_weakly_holomorphic}
(R^j_{\kappa-2j}f)(\tau) = (4\pi)^j \frac{j!}{2}\sum_{\mu \in L'/L}\sum_{\substack{m\in \Z-Q(\mu) \\m > 0}}m^ja_{f}(\mu,-m) F_{\mu,m,\kappa ,s_0}\left(\tau\right),
\end{align}
with $s_0=1-\frac{\kappa}{2}+j$.
%{\claudia with $s_0=1-\frac{\kappa}{2}+j$}.
For $z\in \Gr(L)$ outside the union of the $H(\mu,m)$ with $a_f(\mu,-m)\neq 0$, this implies 
\begin{align}\label{eq:theta_lift_with_raising_operator}
\Psi^j(f,z) = (4\pi )^j \frac{j!}{2} \sum_{\mu \in L'/L}\sum_{\substack{m\in \Z-Q(\mu) \\m > 0}} m^j a_{f}(\mu,-m) \Phi_{\mu,m}(z,s_0),
\end{align}
where $\Phi_{\mu,m}(z,s)$ is defined in \eqref{eq:Phi_mu_m}. Now, for each $\mu\in L'/L$ and $m\in \Z-Q(\mu)$ with $m>0$ write
$$F_{\mu,m,\kappa,s}\left(\tau\right)=\sum_{\gamma \in L'/L} F_{\mu,m,\kappa,s,\gamma}(\tau)  \e_\gamma$$
and
$$F_{\mu,m,\kappa,s,\gamma}(\tau)=\sum_{n\in \Z+Q(\gamma)}c_{\kappa,m,\mu}(\gamma,n;y,s) e(nx).$$
%\todo[inline]{$F_{\mu,m,\kappa}\left(\tau,s\right)$ was $F_{\mu,m,\kappa,s}\left(\tau\right) $ before, which notation do we want to take?} Corrected :)
%\todo[inline]{Claudia: I think from here on, $k $ should be $\kappa$. You are right!}
As in \cite[p.~57]{brhabil} we can write  $c_{\kappa,m,\mu}(\gamma,0;y,s)=b_{\kappa,m,\mu}(\gamma,0,s)y^{1-s-\frac{\kappa}{2}}$. Applying \cite[Theorems 2.14 and 2.15]{brhabil} we get, when $y>\sqrt{2m}$%\todo{ absolute value missing?} m is positive here
, the formula
\begin{equation}\label{eq:Fourier_expansion_Phi_mu,m}
\begin{split}
   \Phi_{\mu,m} \left((x,y),s_0\right)=& \frac{8\pi m y}{\left(\frac{q-1}{2}+j\right)(j+1)!} \#\{Y\in \mu+K: Q(Y)=-m\}\\
   & + \frac{2}{\pi^{1+j}y^{1+2j}}j!\, b_{\kappa,m,\mu}(0,0,s_0)\zeta(2j+2)\\
&    +2y\sum_{Y \in K'\setminus \{0\}}\sum_{n=1}^\infty e\left(n(Y,x)\right)\\
 & \qquad \times \int_0^\infty c_{\kappa,m,\mu}(Y,Q(Y);t,s_0)t^{-2}\exp\left(-\frac{\pi n^2y^2}{t}+2\pi t Q(Y)\right)dt.
 \end{split}
\end{equation}
%\todo[inline]{Claudia: Two issues? a) Sum should be over $K'$?, b) I think the $y$ in the integral is actually two different $y$'s one being the integration variable from Jan and the other one coming from $z_\nu$ (in Jan's notation), so the integral should be
% \[
% \int_0^\infty c_{\kappa,m,\mu}(Y,Q(Y);t,s_0)t^{-2}\exp\left(-\frac{\pi n^2 y^2}{t}+2\pi t Q(Y)\right)dt
% \]
% This would later affect the power of $Q(Y)$ not being $a$ but $j-a$ I think} I hope it is correct now.
Replacing $f$ in~\eqref{eq:raising_of_weakly_holomorphic} by its Fourier expansion and using~\eqref{eq:raising_of_exponential} we get the equality
\begin{equation}\label{eq:coeff_linear_comb_Niebur}
\begin{split}
 & a_f(\gamma,n)(-1)^j \left(\sum_{r=0}^j{j \choose r}(j+1-\kappa)_{r}t^{-r}(4\pi n)^{j-r} \right)e^{-2\pi nt} \\
 &= (4\pi )^j \frac{j!}{2}\sum_{\mu \in L'/L}\sum_{\substack{m\in \Z-Q(\mu) \\m > 0}} m^ja_{f}(\mu,-m)c_{\kappa,m,\mu}(\gamma,n;t,s_0).
 \end{split}
\end{equation}
for all $\gamma\in L'/L$, $n\in \Z+Q(\gamma)$ and $t>0$. Combining \eqref{eq:theta_lift_with_raising_operator}, \eqref{eq:Fourier_expansion_Phi_mu,m} and \eqref{eq:coeff_linear_comb_Niebur} leads to
\begin{equation}\label{eq:lift_expansion}
\begin{split}
 \Psi^j(f,(x,y)) =&   \frac{(4\pi )^{j+1}y}{\left(\frac{q-1}{2}+j\right)(j+1)}\sum_{Y\in K'}|Q(Y)|^{j+1} a_{f}(Y,Q(Y))\\
 & +(-1)^j\frac{2}{\pi^{1+j}y^{1+2j}}j!a_f(0,0)(j+1-\kappa)_j\zeta(2j+2)\\
 & + (-1)^j j! \sum_{Y \in K'\setminus \{0\}}  a_f(Y,Q(Y))\\
 & \qquad \sum_{r=0}^j\frac{(4\pi Q(Y))^{j-r}(j+1-\kappa)_{r}}{(j-r)!\pi^{1+r}y^{1+2r}} 2\sum_{n=1}^\infty \frac{e(n(Y,x))}{n^{2+2r}}.
 % & + (4\pi)^{j+1}(-1)^j \, j! \sum_{Y \in K'\setminus \{0\}}  a_f(Y,Q(Y))\\
 % & \qquad \sum_{a=0}^j\frac{(-1)^a(j+1-\kappa)_{a}}{(j-a)!(2(1+a))!y^{1+2a}} Q(Y)^a \mathbb{B}_{2(1+a)}\left((Y,x))\right).
\end{split}
\end{equation}
Using \eqref{eq:formula_B_2+2r} and that $a_f(-Y,Q(-Y))=a_f(Y,Q(Y))$, it is easy to see that the sum of the second and third lines in \eqref{eq:lift_expansion} equals
\[(4\pi)^{j+1} j! \sum_{Y \in K'}  a_f(Y,Q(Y))\sum_{r=0}^j\frac{(j+1-\kappa)_{r}}{(j-r)!(2+2r)!y^{1+2r}} |Q(Y)|^{j-r} \mathbb{B}_{2+2r}\left(\left(Y,x\right)\right).\]
%\todo[inline]{Shouldn't we also see the $Y=0$ contribution here?} Yes. The sum was supposed to be over K'
This implies formula \eqref{eq:final_fourier_expansion_lift} for all $(x,y)\in \H_W$ with $y>\sqrt{2m_0}$ outside the union of the sets $H(\mu,m)$ with $a_f(\mu,-m)\neq 0$. Since both sides of \eqref{eq:final_fourier_expansion_lift} are continuous on $(x,y)$, the equality extends to all $(x,y)\in \H_W$ with $y>\sqrt{2m_0}$. This completes the proof. 
\end{proof}

As a next step towards the proof of Theorem \ref{thm:finite_sum_representation}, we study the singularities\footnote{Recall that the function $\Psi^j(f,z)$ is continuous on all of $\Gr(L)$ and real-analytic outside the union of the sets $H(\mu,m)$ for those $\mu\in L'/L$ and $m>0$ with $a_f(\mu,-m)\neq 0$, but in general not real-analytic at points of this union.} of the lift $\Psi^j(f,z)$. As in \cite[p.~51]{brhabil}, given an open subset $U\subset \Gr(L)$ and functions $g,h$ defined on a dense open subset of $U$, we say that $g$ has a singularity of type $h$ if $g-h$ can be continued to a real-analytic function on $U$.

\begin{proposition}\label{prop:singularities_of_lift}
Around any given point $(x_0,y_0)\in \H_W$ the function $\Psi^j(f,z)$ has a singularity of type
\begin{align*}
    S_{(x_0,y_0)}(x,y) & = (-1)^j(4\pi)^{j+\frac{1}{2}} \sum_{\substack{Y=a\ell'+c\ell +b \in L'\\ Q(Y)<0\\ a(y_0^2-Q(x_0))+(b,x_0)+c=0 }}a_{f}(Y,Q(Y)) \sum_{r=0}^j
   {j \choose r}(j+1-\kappa)_{r}
  \\
  & \qquad \qquad \qquad  \qquad  \times \Gamma\left(-\frac{1}{2}-r\right)\frac{Q(Y)^{j-r}}{(2y)^{1+2r}}|a(y^2-Q(x))+(b,x)+c|^{1+2r}.
\end{align*}
\end{proposition}
\begin{proof}
    We use \cite[Theorem 6.2]{borcherds} with $M=L$, $p=1$, $\alpha=\beta=0$ and $F=R^j_{\kappa-2j}f$.  In Borcherds' notation the Fourier coefficients of $F$ are (using \eqref{eq:raising_of_exponential})
$$c_\mu(n,r):=a_{f}(\mu,n) (-1)^{j}{j \choose r}(j+1-\kappa)_{r}(4\pi n)^{j-r} \quad \text{if }r\in \{0,\ldots,j\}$$
and $c_\mu(n,r)=0$ if $r>j$. We then apply \cite[Theorem 6.2]{borcherds} to  see that  around any given point $z_0\in \Gr(L)$ the function $\Psi^j(f,z)$ has a singularity of type
\begin{align*}
   %&= \sum_{\substack{Y\in L'\cap z_0^\perp\\ Y\neq0}}\sum_{t=0}^jc_Y(Q(Y),t)(4\pi Q(Y_z))^{\frac{1}{2}+t}\Gamma\left(-\frac{1}{2}-t\right)\\
  %&=\sum_{\substack{Y\in L'\cap z_0^\perp\\ Y\neq0}}\sum_{t=0}^ra_{f}(Q(Y),\mu) (-1)^{r-t}{r \choose t}(k-r -t)_{t}(4\pi Q(Y))^{r-t}(4\pi Q(Y_v))^{\frac{1}{2}+t}\Gamma\left(-\frac{1}{2}-t\right)\\
  &(-1)^{j}(4\pi)^{j+\frac{1}{2}}\sum_{\substack{Y\in L'\cap z_0^\perp\\ Y\neq0}}a_{f}(Y,Q(Y))  \sum_{r=0}^j
   {j \choose r}(j+1 -\kappa)_{r}
  Q(Y)^{j-r}Q(Y_{z})^{\frac{1}{2}+r}\Gamma\left(-\frac{1}{2}-r\right).
\end{align*}
If $z_0\in \Gr(L)$ corresponds to $(x_0,y_0)\in \H_W$, then $Y\in z_0^\perp$ is equivalent to $(x_0,y_0)\in S_Y$. Moreover, since $z_0$ is a positive definite line and $(V,Q)$ has signature $(1,q)$, the subspace $z_0^\perp \subset V$ is negative definite and therefore $Q(Y)\leq 0$. Together with $Y\neq 0$, this implies $Q(Y)<0$. Now, writing $Y=a\ell'+c\ell+b$, the condition $(x_0,y_0)\in S_Y$ becomes $a(y_0^2-Q(x_0))+(b,x_0)+c=0$ by \eqref{eq:def_S_Y}. Writing $P=(x,y)\in \H_W$ for the point corresponding to $z\in \Gr(L)$, we 
have that $Y_z=(Y,X_1(P))X_1(P)$ and $Q(Y_z)=\frac{1}{2}(Y,X_1(P))^2$. Combined with \eqref{eq:(X_1(P),Y)}, this implies the desired result. 
%\todo[inline]{I do not understand where the condition $a<0$ comes from, why is $a=0$ not possible?}
%where in the last equality we used that $a_f(-Y,Q(-Y))=a_f(Y,Q(Y))$. 
\end{proof}

We first prove two lemmas needed for the proof of Theorem \ref{thm:finite_sum_representation}.

\begin{lemma}\label{lem:finiteness_geodesics}
    Fix $(x_0,y_0)\in \H_W$ and $m_0>0$. Then, the sets 
    $$\mathcal{Y}_0:=\{Y=a\ell'+c\ell+b\in L':Q(Y)<0,|Q(Y)|\leq m_0, a(y_0^2-Q(x_0))+(b,x_0)+c=0\}$$
    and
    $$\mathcal{Y}:=\{Y=a\ell'+c\ell+b\in L':Q(Y)<0,|Q(Y)|\leq m_0, a(y_0^2-Q(x_0))+(b,x_0)+c\geq 0>a\}$$
    are both finite. Moreover, if $y_0>\sqrt{m_0}$ then $\mathcal{Y}=\emptyset$.
\end{lemma}
\begin{proof}
    We can split $\mathcal{Y}_0$ as a disjoint union $\mathcal{Y}_0^+\cup \mathcal{Y}_0^0\cup \mathcal{Y}_0^-$ according to whether $a>0$, $a=0$ or $a<0$. In particular
    $$\mathcal{Y}_0^0=\{Y=c\ell+b\in L':-m_0\leq Q(b)<0,(b,x_0)+c=0\}.$$
    Since $b\in K'$ and $K'$ is negative definite, there is a finite subset $\mathcal{S}\subset K'$ such that $b\in \mathcal{S}$ for all $Y=c\ell+b\in \mathcal{Y}_0^0$. Moreover, for every such $Y$ the value of $c$ is determined by $b$. This shows that $\mathcal{Y}_0^0$ is finite. Since $Y\mapsto -Y$ defines a bijection between $\mathcal{Y}_0^+$ and $\mathcal{Y}_0^-$, and  since $\mathcal{Y}_0^-\subseteq \mathcal{Y}$, it remains to prove that $\mathcal{Y}$ is finite and that $\mathcal{Y}=\emptyset $ if $y_0>\sqrt{m_0}$. We start by recalling the Cauchy--Schwarz inequality 
    $$|(b,x)|\leq 2\sqrt{|Q(b)|}\sqrt{|Q(x)|} \text{ for all }b,x\in K\otimes \R,$$
    which, combined with the inequality between the arithmetic and geometric mean, gives
    $$|(b,x)|\leq \frac{Q(b)}{a}+ aQ(x) \text{ for all }b,x\in K\otimes \R \text{ and all }a\in \R_{<0}.$$
    For $Y=a\ell'+c\ell+b\in \mathcal{Y}$ this implies
    $$0\leq a(y_0^2-Q(x_0))+(b,x_0)+c\leq ay_0^2+\frac{Q(Y)}{a}\leq a y_0^2-\frac{m_0}{a}$$
    and then $|a|y_0^2\leq \frac{m_0}{|a|}$. This is possible only for finitely many values of $a\in \Z_{<0}$. Moreover, one has $y_0^2\leq \frac{m_0}{a^2}\leq m_0$, hence $\mathcal{Y}=\emptyset$ if $y_0>\sqrt{m_0}$. Now, fixing one of the finitely many possibilities for $a\in \Z_{<0}$, we have
    $$0<-a(y_0^2-Q(x_0))\leq (b,x_0)+c=(b,x_0)+\frac{Q(Y)-Q(b)}{a}\leq (b,x_0)-\frac{Q(b)}{a}-\frac{m_0}{a}$$
    and then $\frac{m_0}{a}\leq (b,x_0)-\frac{Q(b)}{a}$. Since $K'$ is negative definite, this gives only finitely many possibilities for $b\in K'$. Finally, fixing $a\in \Z_{<0}$ and $b\in K'$, the restriction $|Q(Y)|=|Q(b)+ac|\leq m_0$ gives finitely many possibilities for $c\in \Z$. This proves that $\mathcal{Y}$ is finite and completes the proof of the lemma.
\end{proof}

\begin{lemma}\label{lemma:type_sing_B_{2+2r}}
    Let $r\geq 0$ be an integer. Then, the functions
    $$\mathbb{B}_{2+2r}(t)+(r+1)|t|^{1+2r}\quad \text{and}\quad |t|^{1+2r}-2\max\{0,t^{1+2r}\}$$
    are real-analytic around $t=0$.
\end{lemma}
\begin{proof}
    By definition, for $t\in [0,1)$ we have $\mathbb{B}_{2+2r}(t)=B_{2+2r}(t)$, while for $t\in [-1,0)$ we have $\mathbb{B}_{2+2r}(t)=B_{2+2r}(t+1)$. By \cite[p.~36, Equation (6)]{EMT81} we have $B_{2+2r}(t+1)=B_{2+2r}(t)+(2+2r)t^{1+2r}$. For $t\in [-1,1)$ it follows that
    $$\mathbb{B}_{2+2r}(t)+(r+1)|t|^{1+2r}=B_{2+2r}(t)+(r+1)t^{1+2r}.$$
    This implies the result for $\mathbb{B}_{2+2r}(t)+(r+1)|t|^{1+2r}$ since $B_{2+2r}(t)$ is a polynomial in $t$. The result for the second function follows from the equality $|t|^{1+2r}-2\max\{0,t^{1+2r}\}=-t^{1+2r}$ for all $t\in \R$.
    % and
    % $$\mathbb{B}_{2+2r}(t)+(2r+2)\max\{0,t\}^{1+2r}=B_{2+2r}(t)+(2+2r)t^{1+2r}.$$
\end{proof}

\begin{proof}[Proof of Theorem \ref{thm:finite_sum_representation}]
We start by defining, for $(x,y)\in \H_W$, the functions
\begin{align*}
   \alpha(x,y) &:= \, j! \sum_{Y \in K'\setminus\{0\}}  a_f(Y,Q(Y)) \sum_{r=0}^j \frac{(j+1-\kappa)_{r}}{(j-r)!(2+2r)! y^{1+2r}} |Q(Y)|^{j-r} \mathbb{B}_{2+2r}\left((Y,x)\right),\\
    \beta(x,y)
  &:=  4(-1)^j(4\pi)^{-\frac{1}{2}} \sum_{\substack{Y=a\ell'+c\ell+b \in  L'\\ Q(Y)<0,a<0}}a_{f}(Y,Q(Y)) \sum_{r=0}^j
   {j \choose r}(j+1-\kappa)_{r}\\
  & \qquad \qquad \qquad  \quad   \times \Gamma\left(-\frac{1}{2}-r\right)\frac{Q(Y)^{j-r}}{(2y)^{1+2r}}\max\{0,a(y^2-Q(x))+(b,x)+c\}^{1+2r},
\end{align*}
and
\begin{equation}\label{eq:F_function}
\begin{split}
    F(x,y) & :=\frac{y}{\left(\frac{q-1}{2}+j\right)(j+1)}\sum_{Y\in K'}|Q(Y)|^{j+1} a_{f}(Y,Q(Y))\\
    & \qquad + j!\, a_f(0,0) \frac{(j+1-\kappa)_j}{(2+2j)!y^{1+2j}}\mathbb{B}_{2+2j}(0)\\
    & \qquad \qquad + \alpha(x,y)+\beta(x,y).
    \end{split}
\end{equation}
For given $(x,y)\in \H_W$, the sums defining $\alpha(x,y)$, $\beta(x,y)$ and $F(x,y)$ are finite by Proposition \ref{prop:Fourier_expansion_lift} and Lemma \ref{lem:finiteness_geodesics}. The term corresponding to $Y=0$ in the second sum on the right hand side of \eqref{eq:finite_sum_representation} equals the second line of \eqref{eq:F_function}. Moreover, the last line of \eqref{eq:finite_sum_representation} is equal to $\beta(x,y)$. Hence, it suffices to prove that $\frac{1}{(4\pi)^{j+1}}\Psi^j(f,(x,y))=F(x,y)$ for all $(x,y)\in \H_W$.

We will first show that the difference $\frac{1}{(4\pi)^{j+1}}\Psi^j(f,(x,y))-F(x,y)$ is real-analytic on $\H_W$. By Proposition \ref{prop:singularities_of_lift}, around  a given point $(x_0,y_0)\in \H_W$ the function $\frac{1}{(4\pi)^{j+1}}\Psi^j(f,(x,y))$ has a singularity of type $T_{(x_0,y_0)}(x,y)+U_{(x_0,y_0)}(x,y)$ where
\begin{align*}
  T_{(x_0,y_0)}(x,y)& := (-1)^j(4\pi)^{-\frac{1}{2}} \sum\limits_{\substack{Y=c\ell  +b \in L'\\ Q(Y)<0\\ (b,x_0)+c=0 }}a_{f}(Y,Q(Y)) \sum_{r=0}^j
   {j \choose r}(j+1-\kappa)_{r}
  \\
  & \qquad \qquad \qquad  \qquad  \times \Gamma\left(-\frac{1}{2}-r\right)\frac{Q(Y)^{j-r}}{(2y)^{1+2r}}|(b,x)+c|^{1+2r} ,\\
  U_{(x_0,y_0)}(x,y)& := 2(-1)^j(4\pi)^{-\frac{1}{2}} \sum\limits_{\substack{Y=a\ell'+c\ell +b \in L'\\ Q(Y)<0,a<0\\ a(y_0^2-Q(x_0))+(b,x_0)+c=0 }}a_{f}(Y,Q(Y)) \sum_{r=0}^j
   {j \choose r}(j+1-\kappa)_{r}
  \\
  & \qquad \qquad \qquad  \qquad  \times \Gamma\left(-\frac{1}{2}-r\right)\frac{Q(Y)^{j-r}}{(2y)^{1+2r}}|a(y^2-Q(x))+(b,x)+c|^{1+2r}.
\end{align*}
These are finite sums by Lemma \ref{lem:finiteness_geodesics}. 
To prove that $\frac{1}{(4\pi)^{j+1}}\Psi^j(f,(x,y))-F(x,y)$ is real-analytic around $(x_0,y_0)$, it is enough to check that $T_{(x_0,y_0)}(x,y)-\alpha(x,y)$ and $U_{(x_0,y_0)}(x,y)-\beta(x,y)$ are both real-analytic around that point. Since $\mathbb{B}_{2+2r}(x)$ is real-analytic for $x\in \R\setminus \Z$, around $(x_0,y_0)$ the function $\alpha(x,y)$ has a singularity of type
$$ j! \sum_{\substack{Y \in K'\setminus\{0\}\\ (Y,x_0)\in \Z}}  a_f(Y,Q(Y)) \sum_{r=0}^j \frac{(j+1-\kappa)_{r}}{(j-r)!(2+2r)! y^{1+2r}} |Q(Y)|^{j-r} \mathbb{B}_{2+2r}\left((Y,x)\right).$$
Since $L'=\Z \ell +\Z \ell'+K'$ and every $Y=c\ell+b\in L'$ satisfies $(Y,x_0)=(b,x_0)$, the map $c\ell+b\mapsto b$ defines a bijection between the sets $\{Y=c\ell+b\in L':Q(Y)<0,(b,x_0)+c=0\}$ and $\{Y\in K'\setminus \{0\}:(Y,x_0)\in \Z\}$.
%\todo[inline]{Are these sets really equal? This should be an isomorphism? Also: $-\ell+b$ is in the LHS, but not the right? Do you want to require that $(Y,x_0)\in\Z$ instead of vanishing?} I hope it is correct now :)
Thus, around $(x_0,y_0)$ the difference $T_{(x_0,y_0)}(x,y)-\alpha(x,y)$ has a singularity of type
\begin{align*}
    & (-1)^j j!\sum\limits_{\substack{Y=c\ell +b \in L'\\ Q(Y)<0\\ (b,x_0)+c=0 }}a_{f}(Y,Q(Y)) \sum_{r=0}^j
   \frac{(j+1-\kappa)_{r}}{(j-r)!}\frac{Q(Y)^{j-r}}{y^{1+2r}}
  \\
  & \qquad \qquad \qquad  \qquad  \times \left( \frac{\Gamma\left(-\frac{1}{2}-r\right)}{\sqrt{\pi}r!\,2^{2+2r}}|(b,x)+c|^{1+2r}-\frac{(-1)^r}{(2+2r)!}\mathbb{B}_{2+2r}((Y,x))\right).
\end{align*}
%\todo[inline]{Move $ \frac{j!}{(j-r)!}$ inside the sum?} Yes 
Using  Remark \ref{rmk:Gamma(-1/2-r)} we get $\frac{\Gamma\left(-\frac{1}{2}-r\right)}{\sqrt{\pi}r!\,2^{2+2r}}=\frac{1}{2}\frac{(-1)^{r+1}}{(2r+1)!}$, thus the above sum is a finite linear combination of terms of the form
\begin{equation}\label{eq:|(b,x)+c|^{1+2r}+B_{2+2r}((Y,x))}
  \frac{1}{y^{1+2r}}\left((r+1)|(b,x)+c|^{1+2r}+\mathbb{B}_{2+2r}((Y,x))\right)  
\end{equation}
with $Y=c\ell +b\in L'$ satisfying $Q(Y)<0$ and $(b,x_0)+c=0$. Since $x\in W\subseteq \ell^\perp$, we have $\mathbb{B}_{2+2r}((Y,x))=\mathbb{B}_{2+2r}((b,x))=\mathbb{B}_{2+2r}((b,x)+c)$; thus by Lemma \ref{lemma:type_sing_B_{2+2r}} all terms of the form \eqref{eq:|(b,x)+c|^{1+2r}+B_{2+2r}((Y,x))} are real-analytic around $(x_0,y_0)$. This proves that $T_{(x_0,y_0)}(x,y)-\alpha(x,y)$ is real-analytic around $(x_0,y_0)$. Similarly, around $(x_0,y_0)$ the difference $U_{(x_0,y_0)}(x,y)-\beta(x,y)$ has a singularity whose type is a finite linear combination of terms of the form
$$\frac{1}{y^{1+2r}}\left(|a(y^2-Q(x))+(b,x)+c|^{1+2r}-2\max\{0,a(y^2-Q(x))+(b,x)+c\}^{1+2r}\right).$$
Hence,  our claim for $U_{(x_0,y_0)}(x,y)-\beta(x,y)$ follows from Lemma \ref{lemma:type_sing_B_{2+2r}}. This proves that $\frac{1}{(4\pi)^{j+1}}\Psi^j(f,(x,y))-F(x,y)$ is real-analytic around any given point  $(x_0,y_0)\in \H_W$.

Now, let $m_0\in \Q_{>0}$ such that $a_f(\mu,-m)=0$ for all $m>m_0$. It follows from the definition of $\beta(x,y)$ and Lemma \ref{lem:finiteness_geodesics} that $\beta(x,y)=0$ if $y>\sqrt{m_0}$. Combined with Proposition \ref{prop:Fourier_expansion_lift}, this implies $\frac{1}{(4\pi)^{j+1}}\Psi^j(f,(x,y))=F(x,y)$ for all $(x,y)\in \H_W$ with $y>\sqrt{2 m_0}$. We conclude that $\frac{1}{(4\pi)^{j+1}}\Psi^j(f,(x,y))-F(x,y)$ is a real-analytic function on $\H_W$ that vanishes identically for $y>\sqrt{2 m_0}$, hence it vanishes identically on $\H_W$ by the identity theorem for real-analytic functions. This completes the proof of Theorem \ref{thm:finite_sum_representation}.
\end{proof}

\subsection{Geometric interpretation of the lift}\label{sec:geometric-interpretation}

The following corollary is a direct application of Theorem \ref{thm:finite_sum_representation} and Corollary \ref{cor-distance}, together with the identity
$$(-1)^{r}(j+1-\kappa)_{r}\Gamma\left(-\frac{1}{2}-r\right)=-2\sqrt{\pi} 
\frac{\left(j+\frac{1+q}{2}\right)_r}{(3/2)_r} \quad \text{for }r\in \{0,1,\ldots,j\}.$$

\begin{corollary}\label{cor:geo_interpretation}
Let $j\in \N_0$ and $f\in M^!_{\kappa-2j,L}$. Then,  for every $P=(x,y)\in \H_W$ we have
    \begin{equation}\label{eq:geom_interpretation}
    \begin{split}
     \frac{1}{(4\pi)^{j+1}}\Psi^j(f,P)
   &=  \frac{y}{\left(\frac{q-1}{2}+j\right)(j+1)}\sum_{Y\in K'}|Q(Y)|^{j+1} a_{f}(Y,Q(Y))\\
   &\quad +  j! \sum_{Y \in K'}  a_f(Y,Q(Y))\sum_{r=0}^j\frac{\left(j+\frac{1+q}{2}\right)_r|Q(Y)|^{j-r}}{(j-r)!(2+2r)!y^{1+2r}}  \mathbb{B}_{2+2r}\left((Y,x)\right) \\
 & \quad  -4\sum_{\substack{Y\in L'\\ Q(Y)<0,a_Y<0\\ \sign_Y(P)=1}}a_{f}(Y,Q(Y)) |Q(Y)|^{j+\frac{1}{2}} \\
 & \qquad \qquad \qquad  \times \sum_{r=0}^j\binom{j}{r} 
\frac{\left(j+\frac{1+q}{2}\right)_r}{\left(3/2\right)_r}
   \sinh d(P,S_Y)^{1+2r}.
 %& \quad  +\frac{2}{\sqrt{\pi}} \sum_{\substack{Y\in L'\\ Q(Y)<0,a_Y<0\\ \sign_Y(P)=1}}a_{f}(Y,Q(Y)) |Q(Y)|^{j+\frac{1}{2}} \sum_{r=0}^j(-1)^{r}{j \choose r}(j+1-\kappa)_{r}
 % \\
 % & \qquad \qquad \qquad  \qquad \times \Gamma\left(-\frac{1}{2}-r\right)\sinh d(P,S_Y)^{1+2r}.
  \end{split}
\end{equation}
\end{corollary}

We observe that the second line in \eqref{eq:geom_interpretation} reproduces the singularities of $\Psi^j(f,P)$ along $Y^\perp$ for those $Y\in L'$ with $Q(Y)<0$ and $a_Y=0$, while the last two lines reproduce the singularities along $Y^\perp$ for those  $Y\in L'$ with $Q(Y)<0$ and $a_Y\neq 0$. Moreover, the conditions $Y\in L',Q(Y)<0,\, a_Y<0$ and $\sign_Y(P)=1$ mean that $S_Y$ is a hemisphere and that $P$ lies in the bounded component of $\H_W\setminus S_Y$. Finally, we observe that for every $Y\in K'\setminus \{0\}$ and every $P_0=(x_0,y_0)\in \H_W$, writing $(Y,x_0)=n+\varepsilon$ with $n\in \Z$ and $\varepsilon\in (-1,1)$ gives
$$\mathbb{B}_{2+2r}((Y,x_0))=B_{2+2r}(|(Y,x_0)-n|)=B_{2+2r}\left(2y_0\sqrt{|Q(Y)|}\sinh d(P_0,S_{\widetilde{Y}})\right)$$
where $\widetilde{Y}:=Y-n\ell \in L'$. This shows that locally around $P_0$ the function $\mathbb{B}_{2+2r}((Y,x))$ is a polynomial in $y\sinh d(P,S_{\widetilde{Y}})$ with $d(P,S_{\widetilde{Y}})$ the hyperbolic distance between $P$ and the vertical hyperplane $S_{\widetilde{Y}}$. If $(Y,x_0)\not \in \Z$, then there are two possible choices for $n$ corresponding to the two vertical hyperplanes $S_{\widetilde{Y}}$ closest to $P_0$ and parallel to $S_Y$. Using that $B_{2+2r}(1-x)=B_{2+2r}(x)$ one can show that the above expression does not depend on the choice of $S_{\widetilde{Y}}$. If $(Y,x_0) \in \Z$, then there is a unique such choice.

\section{Hypersurface integrals of Green's functions}\label{sec:Hypersurface integrals of Green's functions}

\subsection{The resolvent kernel}

In this section we use the classical model
\begin{equation*}
\mathbb H^q=
\left\{
(x_1,\ldots,x_{q-1},y)\in
\mathbb R^{q}:y>0
\right\}.
\end{equation*}
The hyperbolic Laplace operator in $\H^q$ is
$$\Delta:=-y^2\left(\frac{\partial^2}{\partial x_1^2}+\ldots+\frac{\partial^2}{\partial x_{q-1}^2}+\frac{\partial^2}{\partial y^2}\right)+(q-2)y\frac{\partial}{\partial y}.$$ 
The resolvent kernel for the hyperbolic Laplacian can be found in \cite[Equation (2.1)]{borthwick-marzuola}. After adjusting their normalization  to match ours, we find that 
for $s\in \C$ with $\Re(s)>\frac{3-q}{2}$ and real $t>1$, the resolvent kernel is given in terms of the function 
\begin{equation*}%\label{def:g_s}
    g_s(t):=(2\pi)^{-\frac{q}{2}}e^{-\pi i \frac{(q-2)}{2}} (t^2-1)^{-\frac{q-2}{4}}Q_{s-2+\frac{q}{2}}^{\frac{q-2}{2}}(t)
\end{equation*}
where $Q_{\nu}^\mu$ denotes the Legendre function of the second kind. Using the hypergeometric representation of $Q_{\nu}^\mu$ (see \cite[Equation 8.703]{GR07}) %(see DLMF, Eq.~14.3.7), 
the function $g_s(t)$ can also be written as
\begin{equation}\label{eq:g_s-hypergeometric}
    g_s(t)
=
A_q(s)\,
t^{-(s+q-2)}
{}_2F_1\!\left(
\frac{s+q-1}{2},
\frac{s+q-2}{2};
s+\frac{q-1}{2};
\frac1{t^2}
\right),
\end{equation}
where
\begin{equation}\label{eq:A_q(s)}
A_q(s)
:=
\frac{\Gamma(s+q-2)}
{2^{\,s+q-1}\pi^{\frac{q-1}{2}}
\Gamma\!\left(s+\frac{q-1}{2}\right)}.
\end{equation}
For $P,Q\in \H^q$ with $P\neq Q$ we have
$$\Delta_P g_s(\cosh \, d(P,Q))=(s+q-2)(1-s)g_s(\cosh \, d(P,Q)).$$

%\todo[inline]{If $q=2$ for example the Gamma factor in $A_q$ would give a pole if $\Re(s)<(1-q)/2$. Do we want this? Seba: I think $\Re(s)>\frac{3-q}{2}$ according to \cite{borthwick-marzuola} (after we match their notation to ours. I corrected this restriction above)} 

\begin{example}
When $q=2$ we have that $g_s(t)=\frac{1}{2\pi}Q_{s-1}(t)$. %This is a constant multiple of the resolvent kernels used in \cite{GZ85,GZ86}.
%\todo[inline]{Claudia: Is this the same normalization as in the intro? Or do we use GZ there?} Now it should be consisten with the intro
\end{example}

\begin{example}
For \(q=3\)
we have, by \eqref{eq:g_s-hypergeometric},
\[
g_s(t)
=
\frac1{2^{s+2}\pi}
t^{-s-1}
{}_2F_1\!\left(
\frac{s+2}{2},
\frac{s+1}{2};
s+1;
\frac1{t^2}
\right).
\]
By \cite[p.~39, Line 2]{MOS} we have %the quadratic transformation
\[
t^{-s-1}
{}_2F_1\!\left(
\frac{s+2}{2},
\frac{s+1}{2};
s+1;
\frac1{t^2}
\right)
=
2^s
(t^2-1)^{-1/2}
\left(t+\sqrt{t^2-1}\right)^{-s},
\]
hence
\[
g_s(t)
=
\frac1{4\pi}
(t^2-1)^{-1/2}
\left(t+\sqrt{t^2-1}\right)^{-s},
\]
which is a constant multiple of the resolvent kernel on
\(\mathbb H^3\) used in \cite{EGM98} and \cite{HIPS25}.
\end{example}

For later use we state here the following lemma.

\begin{lemma}\label{lem:asymptotics_resolvent_kernel}
\begin{enumerate}
    \item[$(i)$] If $q=2$, then
    $$g_s(t)\sim -\frac{1}{4\pi}\log(t-1) \quad \text{ as }t\to 1^+,$$
    while if $q>2$ then
    $$g_s(t)\sim (t-1)^{1-\frac{q}{2}}\frac{\Gamma\left(\frac{q}{2}-1\right)}{2^{1+\frac{q}{2}}\pi^{\frac{q}{2}}} \quad \text{ as }t\to 1^+.$$
    \item[$(ii)$] $g_s(t)=O_s(t^{-(\Re(s)+q-2)})$ as $t\to \infty$.
\end{enumerate}
\end{lemma}
\begin{proof}
    Item $(i)$ follows from the fact that for $a,b,c\in \C$ with $\Re(c)>0$ the hypergeometric series satisfies
    $$_2F_1(a,b;c;x)\sim -\log(1-x)\frac{\Gamma(a+b)}{\Gamma(a)\Gamma(b)} \quad \text{ as }x\to 1^- \text{ if }c=a+b$$
    and 
    $$_2F_1(a,b;c;x)\sim (1-x)^{-(a+b-c)}\frac{\Gamma(c)\Gamma(a+b-c)}{\Gamma(a)\Gamma(b)} \quad \text{ as }x\to 1^- \text{ if }\Re(a+b-c)>0$$
    (see \cite[Theorem 2.1.3]{AAR99}), together with \eqref{eq:g_s-hypergeometric} and Legendre's duplication formula $\Gamma\left(s\right)\Gamma\left(s+\frac{1}{2}\right)=2^{1-2s}\sqrt{\pi}\Gamma(2s)$. Item $(ii)$ follows directly from \eqref{eq:g_s-hypergeometric}.
\end{proof}

\subsection{Hypersurface integrals of the resolvent kernel}

The $q$-dimensional upper half-space $\H^q$ is equipped with the hyperbolic metric
\[
\d s^2:=\frac{\d x_1^2+\cdots+\d x_{q-1}^2+\d y^2}{y^2}
\]
and with the hyperbolic measure
\[
\d\mu := \frac{\d x_1\cdots \d x_{q-1}\d y}{y^{q}}.
\]
If $S\subset \H^q$ is a geodesic hypersurface, we denote by $d\omega$ the Euclidean measure on $S$. The induced hyperbolic hypersurface measure on $S$ is 
$$\d\sigma:=\frac{\d\omega}{y^{q-1}}.$$

Using the notation from Section \ref{sec:quad_spaces_(1,q)_hyp_geometry}, we can write $\H^q=\H_W$ where $W=\R^{q-1}$ and $V=\R \ell +\R \ell'+W$ is a quadratic space with quadratic form 
$$Q(a\ell'+c\ell+(x_1,\ldots,x_{q-1}))=ac-(x_1^2+\ldots+x_{q-1}^2).$$
For every vector $Y\in V$  with $Q(Y)<0$ we have the geodesic hypersurface $S_Y\subset \H^q$ defined in Section \ref{sec:points_and_geodesic_hypersurfaces}.

\begin{example}
Take $q=2$ and $Y=\ell'-\ell$. Then $S_Y=\{\tau\in \H:|\tau|^2=1\}$ is the unit arc in $\H$. The hyperbolic arc element is $\d\sigma=\frac{\sqrt{\d x^2+\d y^2}}{y}$ and 
for every $f:S_Y\to \C$ we have
$$\int_{S_Y}f(\tau)\d\sigma(\tau)=\int_0^{\pi}f\left(e^{i\theta}\right)\frac{\d\theta}{\sin\theta}.$$
If we instead take $Y=1\in \mathbb{R}=W$, then $S_{Y}=\{it:t>0\}$ and 
$$\int_{S_{Y}}f(\tau)\d\sigma(\tau)=\int_0^\infty f(it)\frac{\d t}{t}.$$
\end{example}

For $x=(x_1,\ldots,x_{q-1})\in \R^{q-1}$ let $|x|:=\sqrt{x_1^2+\ldots+x_{q-1}^2}$. For every $\rho>0$ the geodesic hemisphere
\[
S_\rho:=
\left\{
(x,y)\in\mathbb H^q:
|x|^2+y^2=\rho^2
\right\}
\]
is parametrized by
\[
  (x,y)=
%\begin{cases}
(\rho\sin\theta\,\omega,\rho\cos\theta) \quad \text{where }
%&\text{where }
\theta\in \left[0,\frac{\pi}{2}\right[, \omega\in S^{q-2}.
%\text{ if }q>2,\\
%(\rho \sin\theta,\rho\cos\theta) & \text{where }
%\theta\in \left]-\frac{\pi}{2},\frac{\pi}{2}\right[  \text{ if }q=2.
%\end{cases}
\]
Here $S^{q-2}\subset \R^{q-1}$ is the Euclidean unit sphere. When $q>2$, writing
\(
\omega=(\omega_1,\ldots,\omega_{q-1})\in S^{q-2}
\)
in spherical coordinates gives
\begin{equation*}%\label{eq:w_spherical}
\begin{split}
\omega_1&=\cos\phi_1,\\
\omega_2&=\sin\phi_1\cos\phi_2,\\
\omega_3&=\sin\phi_1\sin\phi_2\cos\phi_3,\\
&\ \vdots\\
\omega_{q-2}
&=
\sin\phi_1\cdots
\sin\phi_{q-3}
\cos\phi_{q-2},\\
\omega_{q-1}
&=
\sin\phi_1\cdots
\sin\phi_{q-3}
\sin\phi_{q-2},
\end{split}
\end{equation*}
where
\(
0\le\phi_j\le\pi\)
for
\(1\le j\le q-3,
\)
and
\(
0\le\phi_{q-2}<2\pi.
\)
% Hence the hemisphere $S_\rho$ is parametrized by
% \begin{equation}\label{eq:parametrization}
% \sigma(\theta,\phi_1,\ldots,\phi_{q-2})=\begin{cases}
%     (\rho \sin \theta\, w,\rho \cos \theta) & \text{if }q>2,\\
%      (\rho \sin \theta,\rho \cos \theta) & \text{if }q=2,
% \end{cases}
% \end{equation}
% where $\theta,\omega,\phi_1,\ldots,\phi_{q-2}$ are as above. 
%\[
%\begin{aligned}
% x_1&=\rho\sin\theta\cos\phi_1,\\
% x_2&=\rho\sin\theta\sin\phi_1\cos\phi_2,\\
% x_3&=\rho\sin\theta\sin\phi_1\sin\phi_2\cos\phi_3,\\
% &\ \vdots\\
% x_{q-2}
% &=
% \rho\sin\theta
% \sin\phi_1\cdots
% \sin\phi_{q-3}
% \cos\phi_{q-2},\\
% x_{q-1}
% &=
% \rho\sin\theta
% \sin\phi_1\cdots
% \sin\phi_{q-3}
% \sin\phi_{q-2},\\
% y&=\rho\cos\theta.
% \end{aligned}
% \]
In this case the Euclidean measure on $S^{q-2}$ is
\[
\d\omega
:=
\sin^{q-3}\phi_1\,
\sin^{q-4}\phi_2
\cdots
\sin\phi_{q-3}\,
\d\phi_1\cdots \d\phi_{q-2},
\]
and the induced hyperbolic surface measure on
$S_\rho$ is
\begin{equation}\label{eq:measure_dsigma}
\d\sigma
:=
\frac{\sin^{q-2}\theta}{\cos^{q-1}\theta}
\d\theta\d\omega.
\end{equation}
When $q=2$ the same formula holds if we put $\d\omega:=\delta_1+\delta_{-1}$ on $S^{q-2}=\{-1,1\}$ where $\delta_x$ denotes the Dirac measure concentrated at $x$.

To compute the traces of automorphic Green's functions over geodesic hypersurfaces in $\H^q$, we need the following proposition.

\begin{proposition}\label{prop:integral_gs_over_S1}
Let 
$S$ be a geodesic hypersurface %of Euclidean radius \(1\) 
in $\H^q$, let
 \(
P\in\mathbb H^q
\), and put 
\(
\Lambda
:=\cosh\,d(P,S).
\)
Then for every $P\in \H^q$ and every $s\in \C$ with $\Re(s)>0$ we have
\[
\int_{S} g_s(\cosh\, d(P,Q))\d\sigma(Q)
={}
\frac{
\Gamma\left(\frac{s}{2}\right)
\Gamma\left(\frac{s+q-2}{2}\right)
}{
4\sqrt{\pi}\,
\Gamma\left(s+\frac{q-1}{2}\right)
}
\Lambda^{-(s+q-2)}
{}_2F_1\left(
\frac{s}{2},
\frac{s+q-2}{2};
s+\frac{q-1}{2};
\frac{1}{\Lambda^2}
\right)
\]
where the integral on the left hand side and the hypergeometric series on the right hand side both converge absolutely.
\end{proposition}

For the proof of Proposition \ref{prop:integral_gs_over_S1} we start with the following lemma. Recall that $S^{q-2}\subset \R^{q-1}$ denotes the Euclidean unit sphere. In the case $q=2$ we have $S^0=\{-1,1\}$.

\begin{lemma}\label{lemma:distance-sphere}
For every $y>0$, $\omega \in S^{q-2}$ and $\theta\in \left]-\frac{\pi}{2},\frac{\pi
}{2}\right[$, the hyperbolic distance between $P=(0,y)\in \H^q$ and $Q=(\sin \theta \,\omega,\cos \theta)\in S_1$ satisfies
$$\cosh\, d(P,Q)=\frac{y^2+1}{2y\cos(\theta)}.$$
In particular,
$$\cosh\, d(P,S_1)=\frac{y^2+1}{2y}.$$
\end{lemma}
\begin{proof}
We can write $P=(0,y)=P_X$ with $X=\ell'+y^2\ell$, which has $Q(X)=y^2$. Similarly, the point $Q=(\sin \theta \,\omega,\cos \theta)$ can be written as $Q=P_Z$ with $Z=\frac{1}{\sqrt 2 \cos \theta}(\ell'+\ell+\sin \theta\, \omega)$, which has $Q(Z)=\frac{1}{2}$. Then
\begin{eqnarray*}
    \cosh\, d(P,Q)=\frac{(X,Z)}{2\sqrt{Q(X)|Q(Z)|}}=\frac{1+y^2}{2y\cos(\theta)}.
\end{eqnarray*}
This proves the first statement. For the second statement, note that
\begin{eqnarray*}
    \cosh\, d(P,S_1)=\min_{\theta\in [0,\pi/2[}\frac{1+y^2}{2y\cos(\theta)}=\frac{1+y^2}{2y}.
\end{eqnarray*}
This completes the proof of the lemma.
\end{proof}

\begin{proof}[Proof of Proposition \ref{prop:integral_gs_over_S1}]
We start by choosing an isometry $\gamma_0:\H^q\to \H^q$ such that $\gamma_0P=(0,y)$ and $\gamma_0S=S_1$. Since
\[\int_{S} g_s(\cosh\, d(P,Q))\d\sigma(Q)=\int_{S_1} g_s(\cosh \, d((0,y),Q))\d\sigma(Q)\]
and $d(P,S)=d((0,y),S_1)$, we can assume $P=(0,y)$ and $S=S_1$. We parametrize the hemisphere \(S_1\) by
\[
Q(\theta,\omega)
=
\left(
\sin\theta\,\omega,
\cos\theta
\right), \quad \text{ where }
0\leq\theta<\frac{\pi}{2},\,
\omega\in S^{q-2}.
\]
The induced hyperbolic measure $\d\sigma$ on \(S_1\) is given by \eqref{eq:measure_dsigma}. By Lemma~\ref{lemma:distance-sphere}, the hyperbolic distance between \(P\) and \(Q(\theta,\omega)\) satisfies
\[
\cosh \, d\bigl(P,Q(\theta,\omega)\bigr)
=
\frac{\Lambda}{\cos\theta},
\]
with $\Lambda=\cosh \, d(P,S_1)$. This implies
\[\int_{S_1} g_s(\cosh \, d((0,y),Q))\d\sigma(Q)
=\mathrm{vol}(S^{q-2})\int_0^{\pi/2} g_s\left(\frac{\Lambda}{\cos\theta}\right)\frac{\sin^{q-2}\theta}{\cos^{q-1}\theta}\d\theta.
\]
By Lemma \ref{lem:asymptotics_resolvent_kernel} this integral is absolutely convergent for every $\Lambda\geq 1$ and $s\in \C$ with $\Re(s)>0$. Now, using \eqref{eq:g_s-hypergeometric} we have
\[
g_s\left(\frac{\Lambda}{\cos(\theta)}\right)
={}
A_q(s)
\left(
\frac{\cos\theta}{\Lambda}
\right)^{s+q-2}
{}_2F_1\left(
\frac{s+q-1}{2},
\frac{s+q-2}{2};
s+\frac{q-1}{2};
\frac{\cos^2\theta}{\Lambda^2}
\right),
\]
for all $\theta\in \left[0,\pi/2\right[$. We obtain
\[
\begin{aligned}
&  \int_{S_1}g_s(\cosh\, d(P,Q))\d\sigma(Q) \\
&\qquad =
\mathrm{vol}(S^{q-2})
A_q(s)\Lambda^{-(s+q-2)}
\\
&
\qquad\qquad \times\int_0^{\frac{\pi}{2}}
\sin^{q-2}\theta\,
\cos^{s-1}\theta\,
% \\
% &\qquad\qquad\times
{}_2F_1\left(
\frac{s+q-1}{2},
\frac{s+q-2}{2};
s+\frac{q-1}{2};
\frac{\cos^2\theta}{\Lambda^2}
\right)
\d\theta.
\end{aligned}
\]
Put
\[
a:=\frac{s+q-1}{2},
\qquad
b:=\frac{s+q-2}{2},
\qquad
c:=s+\frac{q-1}{2}.
\]
Expanding the hypergeometric function into its defining series gives
\[
{}_2F_1\left(
a,b;c;
\frac{\cos^2\theta}{\Lambda^2}
\right)
=
\sum_{k=0}^{\infty}
\frac{(a)_k(b)_k}{(c)_k\,k!}
\frac{\cos^{2k}\theta}{\Lambda^{2k}},
\]
hence
\begin{equation}\label{eq:before_exchanging_int_and_sum}
\begin{split}
&\int_0^{\frac{\pi}{2}}
\sin^{q-2}\theta\,
\cos^{s-1}\theta
{}_2F_1\left(
a,b;c;
\frac{\cos^2\theta}{\Lambda^2}
\right)
\d\theta\\
&\qquad \qquad=
\int_0^{\frac{\pi}{2}}\sum_{k=0}^{\infty}
\frac{(a)_k(b)_k}{(c)_k\,k!}
\Lambda^{-2k}
\sin^{q-2}\theta\,
\cos^{s+2k-1}\theta
\d\theta.
\end{split}
\end{equation}
By \cite[Equation 3.621.5]{GR07}, for $\Re(s)>0$ we have
\begin{equation}\label{eq:trigonometric_integral}
\int_0^{\frac{\pi}{2}}
\sin^{q-2}\theta\,
\cos^{s+2k-1}\theta
\,\d\theta
=\frac12 \frac{
\Gamma\left(\frac{q-1}{2}\right)
\Gamma\left(\frac{s}{2}+k\right)
}{
\Gamma\left(\frac{s+q-1}{2}+k\right)
} =
\frac12
\frac{
\Gamma\left(\frac{q-1}{2}\right)
\Gamma\left(\frac{s}{2}\right)
}{
\Gamma\left(\frac{s+q-1}{2}\right)
}
\frac{
\left(\frac{s}{2}\right)_k
}{
\left(\frac{s+q-1}{2}\right)_k
}.
\end{equation}
Using $|(a)_k|=\frac{|\Gamma(a+k)|}{|\Gamma(a)|}\leq \frac{\Gamma (\Re(a)+k)}{|\Gamma(a)|}$ and \eqref{eq:trigonometric_integral} with $s$ replaced by $\Re(s)$, this implies
\[
\begin{aligned}
 \int_0^{\frac{\pi}{2}}\sum_{k=0}^{\infty}
 \left| \frac{(a)_k(b)_k}{(c)_k\,k!}
\Lambda^{-2k}
\sin^{q-2}\theta\,
\cos^{s+2k-1}\theta \right|
\d\theta & \leq \frac{\Gamma\left(\frac{q-1}{2}\right)}{2|\Gamma(a)| }\sum_{k=0}^\infty \frac{\Gamma\left(\frac{\Re(s)}{2}+k\right) |(b)_k|}{|(c)_k|k!}
\Lambda^{-2k}  \\
&  =
\frac{\Gamma\left(\frac{q-1}{2}\right)\Gamma(\Re(a'))}{2|\Gamma(a)| }\sum_{k=0}^\infty \frac{(\Re(a'))_k |(b)_k|}{|(c)_k|k!}
\Lambda^{-2k} 
\end{aligned}
\]
where $a':=\frac{s}{2}$. Since $\Re(c-a'-b)=\frac{1}{2}>0$, the series above converges absolutely for every $\Lambda\geq 1$ (see, e.g., \cite[Theorems 2.1.1 and 2.1.2]{AAR99}). Hence, we can exchange summation and integration in \eqref{eq:before_exchanging_int_and_sum}. Since
\(
a=\frac{s+q-1}{2},
\)
after using \eqref{eq:trigonometric_integral} the factor \((a)_k\) cancels, and therefore
\[
\begin{aligned}
\int_0^{\frac{\pi}{2}}
&\sin^{q-2}\theta\,
\cos^{s-1}\theta\,
{}_2F_1\left(
a,b;c;
\frac{\cos^2\theta}{\Lambda^2}
\right)
\d\theta
 =
\frac12
\frac{
\Gamma\left(\frac{q-1}{2}\right)
\Gamma\left(\frac{s}{2}\right)
}{
\Gamma\left(\frac{s+q-1}{2}\right)
}
{}_2F_1\left(
a',
b;
c;
\frac{1}{\Lambda^2}
\right).
\end{aligned}
\]
Using
\(
\mathrm{vol} (S^{q-2})
=
\frac{
2\pi^{\frac{q-1}{2}}
}{
\Gamma\left(\frac{q-1}{2}\right)
},
\)
we find
\[
\begin{aligned}
\int_{S_1}g_s(\cosh\, d(P,Q))\d\sigma(Q)
={}&
\frac{A_q(s)
\pi^{\frac{q-1}{2}}
\Gamma\left(\frac{s}{2}\right)
}{
\Gamma\left(\frac{s+q-1}{2}\right)
}
\Lambda^{-(s+q-2)}
{}_2F_1\left(
a',
b;
c;
\frac{1}{\Lambda^2}
\right).
\end{aligned}
\]
As shown above, the hypergeometric series ${}_2F_1\left(
a',
b;
c;
\frac{1}{\Lambda^2}
\right)$ converges absolutely for every $\Lambda\geq 1$, hence for every $P\in \H^q$.
Now, substituting the definition of \(A_q(s)\) found in \eqref{eq:A_q(s)} gives
\[
\begin{aligned}
\int_{S_1}g_s(\cosh\, d(P,Q))\d\sigma(Q)
={}&
\frac{
\Gamma(s+q-2)
\Gamma\left(\frac{s}{2}\right)
}{
2^{s+q-1}
\Gamma\left(s+\frac{q-1}{2}\right)
\Gamma\left(\frac{s+q-1}{2}\right)
}
\Lambda^{-(s+q-2)}
\\
&\qquad \qquad \times
{}_2F_1\left(
\frac{s}{2},
\frac{s+q-2}{2};
s+\frac{q-1}{2};
\frac{1}{\Lambda^2}
\right).
\end{aligned}
\]
Finally, the duplication formula
\[
\Gamma(s+q-2)
=
\frac{
2^{s+q-3}
}{
\sqrt{\pi}
}
\Gamma\left(\frac{s+q-2}{2}\right)
\Gamma\left(\frac{s+q-1}{2}\right)
\]
yields the desired result. This completes the proof of the proposition.
\end{proof}

\subsection{Traces over hypersurface integrals of Green's functions}\label{sec:traces_Greens}

Let $(L,Q)$ be an even lattice of signature $(1,q)$ as in Section \ref{sec:finite_sum_resps}, so that $L=(\ell \Z\oplus \ell'\Z)\oplus^\perp K$, where $\ell,\ell'$ are isotropic vectors with $(\ell,\ell')=1$ and $K\subset L$ is a negative definite sublattice of rank $q-1$. Recall the notation $V=L\otimes_{\Z}\R$ and $W=K\otimes \R$, and the identifications $\Gr(V)\simeq \H_W\simeq \H^q$ from Section \ref{sec:hyperbolic-model}.

Since $V$ has signature $(1,q)$, the special orthogonal group $\SO(V)$ has two connected components. We denote by $\SO(V)^+$ the connected component of the identity and define $\SO(L)^+:= \SO(L)\cap \SO(V)^+$.

Let $\Gamma$ be a finite-index subgroup of $\SO(L)^+$ acting trivially on $L'/L$. Then $\Gamma$ acts discontinuously on $\Gr(V)\simeq \H^q$ and this action preserves orientation. The automorphic Green's function associated with $\Gamma$ is defined, for $s\in \C$ with $\Re(s)>1$ and $P,Q\in \H^q$ with $P \not \equiv Q$ mod $\Gamma$, by
$$G_s(P,Q)=\sum_{\gamma \in \Gamma}g_s(\cosh \, d(P,\gamma Q)).$$
This series converges absolutely and uniformly for $(P,Q)$  on compact subsets of $(\H^q\times \H^q)\setminus \{(P,Q): P\equiv Q \text{ mod }\Gamma\}$. Indeed, this can be checked using Lemma \ref{lem:asymptotics_resolvent_kernel} and the fact that the series
\begin{equation*}%\label{eq:series_coshd}
 \sum_{\gamma \in \Gamma}(\cosh \,d(P,\gamma Q))^{-(s+q-2)}   
\end{equation*}
converges absolutely and uniformly for $(P,Q)$ on compact subsets of $ \H^q\times  \H^q$ when $\Re(s)>1$ following the arguments in \cite[Section 3.1]{EGM98} (see also \cite{Mat99}).

\begin{remark}\label{rmk:singularities_Greens}
    It follows from Lemma \ref{lem:asymptotics_resolvent_kernel} that the automorphic Green's function has a singularity at the diagonal of the form 
    $$G_s(P,Q)\sim -\frac{|\Gamma_Q|}{2\pi}\log(d(P,Q)) \quad \text{ as }P\to Q$$
    when $q=2$, and of the form
    $$G_s(P,Q)\sim \frac{|\Gamma_Q|\,\Gamma\left(\frac{q}{2}-1\right)}{4\pi^{\frac{q}{2}}d(P,Q)^{q-2}} \quad \text{ as }P\to Q$$
    when $q>2$. Here $\Gamma_Q$ denotes the stabilizer subgroup of $Q$ in $\Gamma$. This implies that for every $s\in \C$ with $\Re(s)>1$ and every $Q\in \H^q$ the function $P\mapsto G_s(P,Q)$ is locally absolutely integrable over geodesic hypersurfaces.
\end{remark}

For every $\mu \in L'/L$ and $m\in \Z-Q(\mu)$ with $m>0$, the group $\Gamma$ acts on the set $\{Y\in L+\mu:Q(Y)=-m\}$ with finitely many orbits (see, e.g., \cite[Satz 30.2]{Kne02}).

Recall from Section \ref{sec:points_and_geodesic_hypersurfaces} that every $Y\in V$ with $Q(Y)<0$ defines a geodesic hypersurface $S_Y$ given by \eqref{eq:def_S_Y}, which we now consider as a subset of $\H^q$ under the identification $\H_W\simeq \H^q$.

For $Y\in L'$ with $Q(Y)<0$, the action of $\Gamma$ on $\H^q$ restricts to an action of the stabilizer subgroup $\Gamma_Y$ on $S_Y$; the quotient $\Gamma_Y\backslash S_Y$ has finite hyperbolic volume if $q\geq 3$. For $q=2$ this holds if and only if the orthogonal complement of $Y$ in $L\otimes \mathbb{Q}$ is anisotropic (otherwise $S_Y$ joins two cusps and $\Gamma_Y\backslash S_Y$ has infinite length). In all cases $\Gamma_Y$ acts faithfully on $S_Y$.

For every $\mu\in L'/L$, $m\in \Z-Q(\mu)$ with $m>0$ %we regard the set $H(\mu,m)$ defined in \eqref{eq:def_H(mu,m)} as a subset of $\H^q$. For every $\mu \in L'/L$, $m\in \Z-Q(\mu)$ with $m>0$ 
and $P\in \H^q$ we define
\begin{equation}\label{eq:def_traces_mu,m_Greens}
\tr_{\mu,m}\,G_{s}(P,\cdot)
:=
\sum_{\substack{[Y]\in \Gamma \backslash(L+\mu)\\ Q(Y)=-m}}
\int _{\Gamma_Y \backslash S_Y}
G_s(P,Q)\d\sigma(Q).
\end{equation}
Note that the sum over $[Y]$ is a finite sum. Moreover, we show below that each one of the integrals in \eqref{eq:def_traces_mu,m_Greens} converges absolutely when $\Re(s)>1$.

\begin{proposition}\label{prop:tracegreenshypergeometric}
For $P\in \H^q$ and $Y\in V$ with $Q(Y)<0$, write $\Lambda_{P,Y}\coloneq\cosh\, d(P,S_Y)$. Then for every $s\in \C$ with $\Re(s)>1$, every $\mu \in L'/L$, every $m\in \Z-Q(\mu)$ with $m>0$ and every $P\in \H^q$, each integral in \eqref{eq:def_traces_mu,m_Greens} converges absolutely and 
\[
\tr_{\mu,m}\,G_{s}(P,\cdot)=\frac{
\Gamma\left(\frac{s}{2}\right)
\Gamma\left(\frac{s+q-2}{2}\right)
}{
4\sqrt{\pi}\,
\Gamma\left(s+\frac{q-1}{2}\right)
}
\sum_{\substack{Y\in L+\mu\\ Q(Y)=-m}}
\Lambda_{P,Y}^{-(s+q-2)}
{}_2F_1\left(
\frac{s}{2},
\frac{s+q-2}{2};
s+\frac{q-1}{2};
\frac{1}{\Lambda_{P,Y}^2}
\right).
\]
Moreover, the series on the right hand side converges absolutely.
\end{proposition}
\begin{proof}
Let $Y\in V$ with $Q(Y)<0$. For $\Re(s)>1$, the definition of the Green's function gives
\begin{align}\label{eq:integral_double_sum}
      \int_{\Gamma_Y\backslash S_Y}G_{s}(P,Q)\d \sigma(Q)
    &=\int_{\Gamma_Y\backslash S_Y}\sum_{\gamma_1  \in \Gamma/\Gamma_Y}\sum_{\gamma_2\in \Gamma_Y}g_{s}(\cosh\, d(P,\gamma_1\gamma_2 Q))\d \sigma(Q).
\end{align}
Now
\begin{align*}
    \int_{\Gamma_Y\backslash S_Y}\sum_{\gamma_1  \in \Gamma/\Gamma_Y}\sum_{\gamma_2\in \Gamma_Y}|g_{s}(\cosh\, d(P,\gamma_1\gamma_2 Q))|\d \sigma(Q) 
    %=
    %\sum_{\gamma_1  \in \Gamma/\Gamma_Y}\int_{ S_Y}|g_{s}(\cosh\, d(P,\gamma_1 Q))|d \sigma(Q) \\
    =
    \sum_{\gamma_1  \in \Gamma/\Gamma_Y}\int_{ S_{\gamma_1Y}}|g_{s}(\cosh\, d(P, Q))|\d \sigma(Q).
\end{align*}
From the proof of the absolute convergence of the hypersurface integral in Proposition \ref{prop:integral_gs_over_S1} it follows that
$$\int_{ S_{\gamma_1 Y}}|g_{s}(\cosh\, d(P, Q))|\d \sigma(Q)=O_{s}\left(\Lambda_{P,\gamma_1Y}^{-(\Re(s)+q-2)}\right)$$
for all $P\in \H^q$. Letting $P_0:=(0,1)\in \H^q$ we have $d(P_0,S_{\gamma_1Y})\leq d(P_0,P)+d(P,S_{\gamma_1Y})$ by the triangular inequality. Using that $\cosh(a+b)\leq 2\,\cosh(a)\cosh(b)$ for all $a,b\in \R$ and the monotonicity of $\cosh(x)$ for $x\geq 0$, we get $\Lambda_{P_0,\gamma_1Y}\leq 2\,\cosh(d(P_0,P))\Lambda_{P,\gamma_1Y}$. This implies
\begin{align*}
&    \sum_{\gamma_1  \in  \Gamma/\Gamma_Y}\int_{ S_{\gamma_1Y} }|g_{s}(\cosh\, d(P, Q))|\d\sigma(Q) \\
& \qquad \qquad \qquad  = \cosh(d(P_0,P))^{\Re(s)+q-2}\cdot O_s\left(  \sum_{\gamma_1 \in  \Gamma/\Gamma_Y} \Lambda_{P_0,\gamma_1Y}^{-(\Re(s)+q-2)}\right)\\
& \qquad \qquad \qquad \leq  \cosh(d(P_0,P))^{\Re(s)+q-2}\cdot O_s\left(  \sum_{\substack{Y\in L'\\ Q(Y)=-m}} \Lambda_{P_0,Y}^{-(\Re(s)+q-2)}\right).
\end{align*}
Now, for every $Y=a\ell'+c\ell+b\in L'$ with $Q(Y)=-m$ we have $\Lambda_{P_0,Y}^2=1+\frac{(a+c)^2}{4m}$ by Corollary \ref{cor-distance}, hence
\begin{equation}\label{eq:series_over_Y}
    \sum_{\substack{Y\in L'\\ Q(Y)=-m}} \Lambda_{P_0,Y}^{-(\Re(s)+q-2)}=\sum_{\substack{a,c\in \Z,b\in K'\\ac+Q(b)=-m}}\left(1+\frac{(a+c)^2}{4m}\right)^{-(\Re(s)+q-2)}.
\end{equation}
For every $R>0$ we can bound the number of possible values of $a,c\in \Z$, $b\in K'$ satisfying $ac+Q(b)=-m$ and $|a+c|=R$, as follows: noting that $|Q(b)|=m+ac\leq m+\frac{(a+c)^2}{4}= m+\frac{R^2}{4}$ we see that there are $O(R^{q-1})$ possibilities for $b$ (since $K'$ is negative definite of dimension $q-1$). For fixed $b$ we have $ac=|Q(b)|-m$, which together with $|a+c|= R$ gives at most 4 possibilities for $(a,c)$.
%satisfying $|Q(b)|\neq  m$ we have $ac=|Q(b)|-m\neq 0$ giving at most 4 possibilities for $(a,c)$, whereas for those  $b$ satisfying  $|Q(b)|=m$ we have $ac=0$ hence $O(R)$ possibilities for $(a,c)$. Since there are only finitely many $b\in K'$ with $|Q(b)|=m$,
%$O_\varepsilon(R^{q-1+\varepsilon})$ 
Hence, we get $O(R^{q-1})$
possible values for $a,c\in \Z$, $b\in K'$ satisfying $ac+Q(b)=-m$ and $|a+c|= R$. This implies that \eqref{eq:series_over_Y} is bounded above by a positive constant multiple of
$$\sum_{R=1}^\infty R^{-2(\Re(s)+q-2)}R^{q-1}=\sum_{R=1}^\infty R^{-2\Re(s)-q+3},$$
which converges since $\Re(s)>1$.
Having proved the absolute convergence of the right hand side of \eqref{eq:integral_double_sum}, we can exchange summation and integration and get
\begin{align*}
      \int_{\Gamma_Y\backslash S_Y}G_{s}(P,Q)\d\sigma(Q)
    &=\sum_{\gamma_1  \in \Gamma/\Gamma_Y}\sum_{\gamma_2\in \Gamma_Y}\int_{\Gamma_Y\backslash S_Y}g_{s}(\cosh\, d(P,\gamma_1\gamma_2 Q))\d\sigma(Q)\\
    % &= \sum_{\gamma_1  \in \Gamma/\Gamma_Y}\int_{S_Y}g_{s}(\cosh\, d(\gamma_1^{-1}P, Q))\d\sigma(Q)\\
    % &= \sum_{\gamma_1  \in \Gamma_Y\backslash \Gamma}\int_{S_Y}g_{s}(\cosh\, d(\gamma_1 P, Q))\d\sigma(Q)\\
    &= \sum_{\gamma_1  \in  \Gamma/\Gamma_Y}g_{\gamma_1Y,s}( P),
\end{align*}
where
$$g_{\gamma_1Y,s}(P):=\int_{S_{\gamma_1Y}} g_s(\cosh\, d(P,Q))\d\sigma(Q).$$
%satisfies $g_{Y,s}(\gamma P)=g_{\gamma^{-1} Y,s}(P)$ for all $\gamma\in \Gamma$. 
This implies
\begin{align*}
    \tr_{\mu,m}\,G_{s}(P,\cdot) 
    %& 
    %=\sum_{\substack{[Y]\in \Gamma\backslash (L+\mu)\\ Q(Y)=-m}}\sum_{\gamma_1  \in \Gamma_Y\backslash \Gamma}g_{Y,s}(\gamma_1 P)\\
     =\sum_{\substack{[Y]\in \Gamma\backslash (L+\mu)\\ Q(Y)=-m}}\sum_{\gamma_1  \in  \Gamma/\Gamma_Y}g_{\gamma_1Y,s}( P)
     =\sum_{\substack{Y\in L+\mu\\ Q(Y)=-m}}g_{Y,s}(P).
\end{align*}
%Given $Y\in L+\mu$ with $Q(Y)=-m$ we can choose $\gamma_0\in \mathrm{Iso}^+(\H^q)$ such that $\gamma_0 S_Y=S_1$.
%is the vertical line going from $0$ to $i\infty$, so that $\gamma_0 S_X=S_{X_0}$ with $X_0=[0,b_0,0]$ with $b_0>0$. 
%Then
%\begin{align*}
%g_{Y,s}(P)& =\int_{\gamma_0^{-1}S_1} g_s(\cosh \, d(P,Q))\d\sigma(Q)
%=\int_{S_1} g_s(\cosh\, d(\gamma_0 P,Q))\d\sigma(Q).
%\end{align*}
By Proposition \ref{prop:integral_gs_over_S1}, %and the fact that $d(\gamma_0P,S_1)=d(P,S_Y)$ 
 $g_{Y,s}(P)$ equals
%\todo[inline]{The Proposition was proved only for points of the form $(0,...)$. Should we say how it extends to arbitrary $P$?} corrected!
\[
\frac{
\Gamma\left(\frac{s}{2}\right)
\Gamma\left(\frac{s+q-2}{2}\right)
}{
4\sqrt{\pi}\,
\Gamma\left(s+\frac{q-1}{2}\right)
}
(\cosh\, d(P,S_Y) )^{-(s+q-2)}
{}_2F_1\left(
\frac{s}{2},
\frac{s+q-2}{2};
s+\frac{q-1}{2};
\frac{1}{(\cosh\, d(P,S_Y))^2}
\right).
\]
This proves the desired formula for $\tr_{\mu,m}\,G_{s}(P,\cdot)$. Finally, the fact that the right hand side of the formula converges absolutely 
%locally uniformly on $P\in \H^q$ 
follows from the absolute convergence of \eqref{eq:series_over_Y}  and the comparison $\Lambda_{P_0,Y}\leq 2\,\cosh(d(P_0,P))\Lambda_{P,Y}$. This
 completes the proof of the proposition.
\end{proof}

Using Proposition \ref{prop:tracegreenshypergeometric} we obtain the following theorem.

\begin{theorem}\label{thm:traces_Greens_equals_theta_lift}
Let $\mu\in L'/L$, $m\in \Z-Q(\mu)$ with $m>0$ and $P\in \H^q\setminus H(\mu,m)$. Then, for every $s\in \C$ with $\Re(s)>\max\left\{1,\frac{5-q}{2}\right\}$ we have
$$\tr_{\mu,m}\,G_{s}(P,\cdot)= \frac{\Gamma\left(\frac{s}{2}\right)}{16\pi \sqrt{m}} \, \Phi_{\mu,m}\left(P,\frac{s}{2}+\frac{q-1}{4}\right).$$
%\left(\frac{16\pi \sqrt{m}}{\Gamma\left(\frac{s}{2}\right)}\right)^{-1}\times  \Psi(F_{\mu,m,\kappa,\frac{s}{2}+\frac{q-1}{4}},z).$$

%{\claudia The constant might be (I hope ;) ) $\left(\frac{16\pi \sqrt{m}}{\Gamma\left(\frac{s}{2}\right)}\right)^{-1}$.}
\end{theorem}
\begin{proof}
    Let $z\in \Gr(L)$ be the point corresponding to $P$. Recall from Corollary \ref{cor:1/cosh} that
    \[
    \left|Q(Y_{z^\perp})\right| =m \Lambda_{P,Y}^2,
    \]
    where $\Lambda_{P,Y}=\cosh\, d(P,S_Y)$. Substituting $s\mapsto \frac{s}{2}+\frac{q-1}{4}$ in Theorem \ref{thm:Siegel lift of PC series} gives
    \begin{align*}
        \Phi_{\mu,m}\left(P,\frac{s}{2}+\frac{q-1}{4}\right) &= 4\sqrt{\pi m } \frac{\Gamma\left(\frac{s}{2}+\frac{q}{2}-1\right)}{\Gamma\left(s+\frac{q-1}{2}\right)} \\
        &\quad \times \sum_{\substack{Y\in L+\mu\\Q(Y)=-m}} \Lambda_{P,Y}^{2-q-s} \, _2F_1\left( \frac{s+q}{2}-1,\frac{s}{2},s+\frac{q-1}{2}; \frac{1}{\Lambda^2_{P,Y}}\right).
    \end{align*}
    Comparing with Proposition \ref{prop:tracegreenshypergeometric} gives the result.
\end{proof}

\begin{corollary}\label{cor:linear_combinatio_traces_is_theta_lift}
    Let $\kappa=\frac{1-q}{2}$, $j\in \N_0$ and $f\in M^!_{\kappa-2j,L}$. Then, for every $P\in \H^q$ we have
    \begin{equation}\label{eq:evaluation_at_CM_traces_over_geodesics_G_even}
        \sum_{\mu \in L'/L}\sum_{\substack{m\in \Z-Q(\mu)\\m>0}} m^{j+\frac{1}{2}}a_f(\mu,-m)\,\mathrm{tr}_{\mu,m}G_{2+2j}(P,\cdot)=\frac{1}{2(4\pi)^{j+1}}\Psi^j(f,P).
    \end{equation}
\end{corollary}
\begin{proof}
Choose $s=2+2j$ and let $s_0=1-\frac{\kappa}{2}+j=\frac{3+q}{4}+j$ so that $\frac{s}{2}+\frac{q-1}{4}=s_0$. First, assume that $P\in \H^q$ lies outside the union of the sets $H(\mu,m)$ for those  $\mu \in L'/L$ and $m>0$ with $a_f(\mu,-m)\neq 0$. In this case  the result follows directly from \eqref{eq:theta_lift_with_raising_operator} and Theorem \ref{thm:traces_Greens_equals_theta_lift}. The general case then follows by continuity in $P$ of the left and right hand sides of \eqref{eq:evaluation_at_CM_traces_over_geodesics_G_even}.
\end{proof}

We conclude this section with the following theorem which is a direct consequence of Corollary \ref{cor:linear_combinatio_traces_is_theta_lift} and Corollary \ref{cor:geo_interpretation}

\begin{theorem}\label{thm:main_general}
   Let $\kappa=\frac{1-q}{2}$, $j\in \N_0$ and $f\in M^!_{\kappa-2j,L}$. Then, for every $P\in \H^q$ we have
   \[
   \begin{aligned}
     & \sum_{\mu \in L'/L}\sum_{\substack{m\in \Z-Q(\mu)\\m>0}} m^{j+\frac{1}{2}}a_f(\mu,-m)\,\mathrm{tr}_{\mu,m}G_{2+2j}(P,\cdot) \\
     &\quad =\frac{y}{\left(q-1+2j\right)(j+1)}\sum_{Y\in K'}|Q(Y)|^{j+1} a_{f}(Y,Q(Y))\\
   &\quad\quad +  \frac{j!}{2} \sum_{Y \in K'}  a_f(Y,Q(Y))\sum_{r=0}^j\frac{\left(j+\frac{1+q}{2}\right)_r|Q(Y)|^{j-r}}{(j-r)!(2+2r)!y^{1+2r}}  \mathbb{B}_{2+2r}\left((Y,x)\right) \\
 & \quad\quad\quad -2\sum_{\substack{Y\in L'\\ Q(Y)<0,a_Y<0\\ \sign_Y(P)=1}}a_{f}(Y,Q(Y)) |Q(Y)|^{j+\frac{1}{2}}  \sum_{r=0}^j\binom{j}{r} 
\frac{\left(j+\frac{1+q}{2}\right)_r}{\left(3/2\right)_r}
   \sinh d(P,S_Y)^{1+2r}.
   \end{aligned}
   \]
\end{theorem}

\subsection{Evaluation at special points}

We continue with the notation of the previous section, so $L=\Z \ell+\Z\ell'+K$ has signature $(1,q)$ and $W=K\otimes \R$. 

For every $X=a\ell'+c\ell+b\in L\otimes \Q$ with $Q(X)>0$, the point $z=\R X\in \Gr(L)$ is special in the sense of Definition \ref{def:special_point} and corresponds under $\Gr(L)\simeq \H_W$ to
$$P_X=\left(\frac{b}{a},\frac{\sqrt{Q(X)}}{|a|}\right)\in (K\otimes \Q)\times \sqrt{Q(X)}\Q^+\subset \H_{W}.$$

As a consequence of Theorem \ref{thm:finite_sum_representation} we obtain the following result.

\begin{theorem}\label{thm:algebraicity_traces}
Let $\kappa=\frac{1-q}{2}$, $j\in \N_0$, and let $f\in M^!_{\kappa-2j,L}$ have rational principal part. Then for every $X\in L\otimes \Q$ with $Q(X)>0$ we have
$$\frac{1}{(4\pi)^{j+1}}\Psi^j(f,P_X)\in \sqrt{Q(X)}\Q.$$
In particular
\[\sum_{\mu \in L'/L}\sum_{\substack{m\in \Z-Q(\mu)\\m>0}} m^{j+\frac{1}{2}}a_f(\mu,-m)\,\mathrm{tr}_{\mu,m}G_{2+2j}(P_X,\cdot)\in \sqrt{Q(X)}\Q.\]
\end{theorem}
\begin{proof}
Write $P_X=(x,y)=\left(\frac{b}{a},\frac{\sqrt{Q(X)}}{|a|}\right)$. Then $x\in K\otimes \Q$ and $y\in \sqrt{Q(X)}\Q^+$. Theorem \ref{thm:finite_sum_representation} expresses $\frac{1}{(4\pi)^{j+1}}\Psi^j(f,P_X)$ as a finite sum.  
The first assertion follows from this expression together with the rationality of the coefficients $a_f\left(Y,Q(Y)\right)$ for all $Y\in L'$ (Lemma \ref{lemma:mcgrawextended}) and Remark \ref{rmk:Gamma(-1/2-r)}. The second result follows from the first together with Corollary \ref{cor:linear_combinatio_traces_is_theta_lift}.% This completes our proof.
\end{proof}

\begin{corollary}\label{cor:algebraicity_traces}
Let $\kappa=\frac{1-q}{2}$, $j\in \N_0$ and let $\{\lambda(\mu,m)\}$ be a rational relation for $S_{2-\kappa+2j,L^-}$. Then for every $X\in L\otimes \Q$ with $Q(X)>0$ we have
\[\sum_{\mu \in L'/L}\sum_{\substack{m\in \Z-Q(\mu)\\m>0}} m^{j+\frac{1}{2}}\lambda(\mu,m)\,\mathrm{tr}_{\mu,m}G_{2+2j}(P_X,\cdot)\in \sqrt{Q(X)}\Q.\]
\end{corollary}
\begin{proof}
Let $f\in M^!_{\kappa-2j,L}$ be the form constructed in Proposition \ref{prop:rational_relations_define_weakly_holomorphic_forms}. Then $f$ has rational principal part and $a_f(\mu,-m)=2\lambda(\mu,m)$ for all $\mu\in L'/L$ and all $m\in \Z-Q(\mu)$ with $m>0$ (see \eqref{eq:fourier_expansion_rational_relation}). The result now follows directly from Theorem \ref{thm:algebraicity_traces}.
\end{proof}

\section{Complementary results: theta splitting and cycle integrals of meromorphic modular forms}\label{sec:remarks}

% In this section we consider the even lattice
%     \begin{equation*}
%  L=\left\{X=\begin{pmatrix}
%     b & c \\ -a & -b
% \end{pmatrix}:a,b,c\in \Z\right\} \text{ with }Q(X)=ac-b^2
% \end{equation*}
% of signature $(1,2)$. Then $L=\Z \ell+\Z\ell'+K$ where
% $$\ell:=\left(\begin{matrix}
% 0 & 0 \\ -1 & 0
% \end{matrix}\right), \quad \ell':=\left(\begin{matrix}
% 0 & 1 \\ 0 & 0
% \end{matrix}\right)$$
% and $K=\left\{\left(\begin{smallmatrix}
%     b & 0 \\ 0 & -b
% \end{smallmatrix}\right):b\in \Z\right\}$ can be identified with $(\Z,Q(b)=-b^2)$. Moreover
% \begin{equation*}
%  L'=\left\{X=\begin{pmatrix}
%     b/2 & c \\ -a & -b/2
% \end{pmatrix}:a,b,c\in \Z\right\}
% \end{equation*}
% and $K'\simeq (\frac{1}{2}\Z,Q(b)=-b^2)$.

\subsection{Another evaluation of the theta lift at special points}\label{sec:theta_splitting}

In this section we give an alternative proof of Theorem \ref{thm:algebraicity_traces} and Corollary \ref{cor:algebraicity_traces} valid for any even lattice $L$ of signature $(1,q)$ without the hypothesis of splitting a hyperbolic plane over $\Z$. This alternative proof uses an evaluation of Males \cite{males} for the regularized theta lift $\Psi^j(f,z)$ at special points $z \in \Gr(L)$ in terms of Fourier coefficients of harmonic Maass forms. It also relies on results of Li and Schwagenscheidt \cite{LS22} on the existence of harmonic Maass forms with rational Fourier coefficients that are preimages under $\xi_{3/2}$ of unary theta series.

Recall from Definition \ref{def:special_point} that $z\in \Gr(L)$ is called a special point if it is defined over $\Q$, that is, if there exists a positive definite $\Q$-line $z_0\subseteq L \otimes \Q$ such that $z=z_0\otimes \R$. 
For such a special point % $z \in \Gr(L)$ its orthogonal complement $z^{\perp}$ in $V=L\otimes \R$ is also defined over $\Q$ and 
we obtain the rational splitting $L \otimes \Q = z_0 \oplus z_0^{\perp}$ which yields the positive definite lattice $P$ and the negative definite lattice $N$ given by
\[
P = L \cap z_0, \qquad N = L \cap z_0^{\perp}.
\]
Note that $P \oplus N$ is a sublattice of $L$ of finite index. The Siegel theta functions associated with $L$ and $P \oplus N$ are related by
\begin{align*}%\label{eq theta relation}
\Theta_{L} = \left(\Theta_{P\oplus N}\right)^{L},
\end{align*}
where the trace operator is as defined in Section~\ref{section:operatorsandrankin}. Moreover, identifying $\C[(P\oplus N)'/(P\oplus N)]$ with $\C[P'/P] \otimes \C[N'/N]$, we see that the Siegel theta function associated with $P \oplus N$ splits as a tensor product
\begin{align*}\label{splitting siegel theta}
\Theta_{P \oplus N}(\tau,z) = \Theta_{P}(\tau) \otimes \Theta_{N}(\tau),
\end{align*}
where
\begin{equation*}%\label{eq usual holomorphic theta}
    \Theta_{P}(\tau) := \sum_{\mu \in P'/P}\sum_{X \in P+\mu}e(Q(X)\tau)\e_\mu \in M_{1/2,P}
\end{equation*}
is the usual holomorphic (vector-valued) theta series associated with $P$, and
\[
\Theta_N(\tau) =  \Im(\tau)^{q/2}\overline{\Theta_{N^-}(\tau)}
\]
with the holomorphic theta series $\Theta_{N^-} \in M_{q/2,N^-}$, where $N^-=(N,-Q)$. 

Now let $f\in M^{!}_{\kappa-2j,L}$, let $z\in \Gr(L)$ be a special point and let $P,N$ be as above. Then by \cite[Theorem 3.3]{males} we have the evaluation
    \begin{align}\label{eq:evaluation_theta_lift_as_CT}
        \Psi^j(f,z)&= (4\pi)^j\frac{\pi^{\frac{1}{2}}\Gamma(1+j)}{2\Gamma\left(\frac{3}{2}+j\right) }\mathrm{CT}\left(\left\langle f_{P\oplus N}, \left[\tilde{\mathcal G}_P^+, \Theta_{N^-}\right]_j\right\rangle\right),
    \end{align}
where $\mathrm{CT}$ denotes the constant term of a Fourier series, $f_{P\oplus N}$ denotes the image of $f$ under the restriction map defined in Section \ref{section:operatorsandrankin}, and $\tilde{\mathcal G}_P\in H_{3/2,P^-}$ is any $\xi_{3/2}$-preimage of $\Theta_P$.

Since $P$ is a 1-dimensional sublattice of $L$, we can find $X_0\in L$ such that $P=\Z X_0$. Letting $N_0:=Q(X_0)\in \Z$, we have that the lattice $(P,Q)$ is isomorphic to $(\Z,x\mapsto N_0x^2)$. Following the notation in \cite{LS22}, the holomorphic theta series $\Theta_P$ can then be identified with the unary theta series
$$\theta_{N_0}(\tau;0)=\sum_{\mu \in \Z/2N_0\Z}\sum_{n\in 2N_0\Z+\mu}q^{\frac{n^2}{4N_0}}\e_\mu.$$
By \cite[Theorem 1.1(2)]{LS22} there exists a harmonic Maass form $\widetilde{\theta}_{N_0}(\tau;0)\in H_{\frac{3}{2},P^-}$ such that $\widetilde{\theta}^+_{N_0}(\tau;0)$ has rational Fourier coefficients and $\xi_{\frac{3}{2}}\left(\widetilde{\theta}_{N_0}(\tau;0)\right) =\frac{\sqrt{N_0}}{2\pi}\theta_{N_0}(\tau;0)$. It follows that choosing $\tilde{\mathcal G}_P$ as $\frac{2\pi}{\sqrt{N_0}}\widetilde{\theta}_{N_0}(\tau;0)$ gives a preimage of $\Theta_P$ such that
$\tilde{\mathcal G}_P^+$ has Fourier coefficients in $\pi \sqrt{Q(X_0)}\Q$. 

If $f$ has rational principal part, then by Lemma \ref{lemma:mcgrawextended} all the Fourier coefficients of $f$, and hence of $f_{P\oplus N}$, are rational. Since the Fourier coefficients of $\Theta_{N^-}$ are integers we get
\[
\mathrm{CT}\left(\left\langle f_{P\oplus N}, \left[\tilde{\mathcal G}_P^+, \Theta_{N^-}\right]_j\right\rangle\right)\in \pi \sqrt{Q(X_0)}\,\Q.
\]
Together with \eqref{eq:evaluation_theta_lift_as_CT} and the fact that $\Gamma\left(\frac{3}{2}+j\right)\in \sqrt{\pi}\,\Q$, this proves that $\frac{1}{(4\pi)^{j+1}}\Psi^j(f,z)\in \sqrt{Q(X_0)}\,\Q$. 

Finally, given $X\in L\otimes \Q$ with $Q(X)>0$ the special point $z=z_0\otimes \R$ with $z_0:=\Q X$ corresponds under $\Gr(L)\simeq \H^q$ to $P_X$. Then $P=\Z X_0=L\cap z_0=L\cap \Q X_0$ implies $\sqrt{Q(X_0)}\in \sqrt{Q(X)}\, \Q$ and therefore $\frac{1}{(4\pi)^{j+1}}\Psi^j(f,P_X)\in \sqrt{Q(X)}\,\Q$, as desired. This reproves Theorem \ref{thm:algebraicity_traces}, and, as a direct consequence, Corollary \ref{cor:algebraicity_traces}.

\subsection{Cycle integrals of meromorphic modular forms}\label{sec:cycle_integral_meromorphic_mod_forms}

In this section we let $L$ denote an even lattice of full rank in the quadratic space
\begin{equation*}
 V:=\left\{X=\begin{pmatrix}
    -b & -c \\ a & b
\end{pmatrix}:a,b,c\in \R\right\} \text{ with }Q(X)=\det(X)=ac-b^2
\end{equation*}
of signature $(1,2)$; see Example \ref{ex:V_signature_(1,2)}. Recall that every $g\in \SL_2(\R)$ acts as an isometry on $V$ via $X\mapsto gXg^{-1}$. This induces an isomorphism $\PSL_2(\R)\simeq \SO(V)^+$. Moreover, $\Gr(V)\simeq \H$ via $\left(\begin{smallmatrix}
    b & c \\ -a & -b
\end{smallmatrix}\right)\R \mapsto -\frac{b}{a}+i\frac{\sqrt{ac-b^2}}{|a|}$ and this identification is $\SL_2(\R)$-equivariant.

For each $X=\begin{pmatrix}
    -b & -c \\ a & b
\end{pmatrix}\in V$ we define $Q_X(x,y):=ax^2+2bxy+cy^2$.

Let $\Gamma$ be a finite-index subgroup of $\SO(L)^+$ acting trivially on $L'/L$. Then $\Gamma$ can be seen as a subgroup of $\PSL_2(\R)$ acting on $\Gr(V)\simeq \H$. Given $X_0\in V$ with $Q(X_0)>0$ and $j\in \N_0$ let $k:=2+2j$ and
$$f_{k,X_0}(\tau):=\frac{(4Q(X_0))^{\frac{k+1}{2}}}{\pi}\sum_{X\in \Gamma X_0}Q_X(\tau,1)^{-k}.$$
Then $f_{k,X_0}(\tau)$ is a weight $2k$ meromorphic modular form for $\Gamma$ with poles of order $k$ at the points
$$\tau_X:=-\frac{b}{a}+i\frac{\sqrt{Q(X)}}{|a|}\in \H \quad \text{ for each } X= \begin{pmatrix}b & c \\ -a & -b
\end{pmatrix}\in \Gamma X_0.$$

Rationality properties of the cycle integrals of the functions $f_{k,X_0}(\tau)$ have been studied by several authors (see, e.g., \cite{abs20, ABMS21, ls}). The purpose of this section is to recover such results from the rationality of the cycle integrals of the even higher Green's function on $\H$ (Corollary \ref{cor:algebraicity_traces} in the case $q=2$).

We start by observing that the function $f_{k,X_0}(\tau)$ is related to the automorphic Green's function as follows.

\begin{lemma}{\cite[Equation (5.8)]{BKvP19}}\label{lemma:bkvp}
For all even integers $k\geq 2$ we have
\[
R_0^k(G_k(\tau,\tau_{X_0}))=2^{k-3}(k-1)!\frac{|\Gamma_{X_0}|}{\sqrt{Q(X_0)}}f_{k,X_0}(\tau).
\]
\end{lemma}

Let $Y=\begin{pmatrix}
    -b & -c \\ a & b
\end{pmatrix}\in V$ with $Q(Y)<0$. If $a\neq 0$ the geodesic $S_Y$ is a semi-circle that we take with counter-clockwise orientation if $a>0$ and clockwise orientation if $a<0$. If $a=0$, then $S_Y$ is the vertical line $\Re(\tau)=-\frac{c}{2b}$ that we take with south-to-north orientation if $b>0$ and north-to-south orientation if $b<0$. 

For each $\mu \in L'/L$ and each $m\in \Z-Q(\mu)$ with $ m>0$ and such that $f_{k,X_0}(\tau)$ has no poles on $H(\mu,-m)=\bigcup_{\substack{Y\in L+\mu\\ Q(Y)=-m }}S_Y$, we define
\begin{equation}\label{eq:traces_weight_k}
    \tr_{\mu,m}f_{k,X_0}:=\sum_{\substack{[Y]\in \Gamma \backslash (L+\mu)\\Q(Y)=-m}}\int_{\Gamma_Y\backslash S_Y}f_{k,X_0}(\tau)\left(\frac{2\sqrt{m}}{Q_Y(\tau,1)}\right)^{1-k}\d\tau.
\end{equation}

\begin{remark}
   For every $\Gamma_Y$-invariant function $f:S_Y\to \C$ that is continuous almost everywhere  one has
   $$\int_{\Gamma_Y\backslash S_Y}f(\tau)\frac{2\sqrt{m}}{Q_Y(\tau,1)}\d\tau= \int_{\Gamma_Y\backslash S_Y}f(\tau)\d\sigma(\tau),$$
   hence the definition \eqref{eq:traces_weight_k} is consistent with \eqref{eq:def_traces_mu,m_Greens}.
\end{remark}

Recalling that $\Delta G_k(\cdot,\tau_{X_0})=k(1-k)G_k(\cdot,\tau_{X_0})$, we can use  \cite[Theorem 1.1]{AS20} to relate, for every integer $r\geq 2$, the cycle integrals of $R_0^r(G_k(\cdot,\tau_{X_0}))$ to those of $R_0^{r-2}G_k(\cdot,\tau_{X_0})$, as follows.

\begin{lemma}\label{lem:cycle_recursion_Gk}
Let $X_0,Y\in V$ with $Q(X_0)>0$ and $Q(Y)=-m<0$, and $r\geq 2$ an integer. Assume that $\Gamma \tau_{X_0}\cap  S_Y=\emptyset$. Then, we have
\begin{align*}
 &   \int_{\Gamma_Y\backslash S_Y}R_0^{r}G_k(\tau,\tau_{X_0}) \left(\frac{2\sqrt{m}}{Q_Y(\tau,1)}\right)^{1-r}\d\tau
\\
&\qquad  =(k(k-1)-(r-1)(r-2))
 \int_{\Gamma_Y\backslash S_Y}R_0^{r-2}G_k(\tau,\tau_{X_0}) \left(\frac{2\sqrt{m}}{Q_Y(\tau,1)}\right)^{3-r}\d\tau.
\end{align*}
\end{lemma}

For $k=2+2j$ with $j\in \N_0$, using Lemma \ref{lem:cycle_recursion_Gk} several times gives
\begin{align*}
 \int_{\Gamma_Y\backslash S_Y}R_0^{k}G_k(\tau,\tau_{X_0}) \left(\frac{2\sqrt{m}}{Q_Y(\tau,1)}\right)^{1-k}\d\tau
= \frac{2( (2j+1)!)^2}{(j!)^2}
 \int_{\Gamma_Y\backslash S_Y}G_k(\tau,\tau_{X_0}) \d\sigma(\tau)
\end{align*}
Combined with Lemma \ref{lemma:bkvp} this leads to the identity
\[\int_{\Gamma_Y\backslash S_Y}f_{k,X_0}(\tau)\left(\frac{2\sqrt{m}}{Q_Y(\tau,1)}\right)^{1-k}\d\tau=
\frac{(2j+1)! \sqrt{Q(X_0)}}{2^{k-4}(j!)^2 |\Gamma_{X_0}|} \int_{\Gamma_Y\backslash S_Y}G_{2+2j}(\tau,\tau_{X_0}) \d\sigma(\tau)\]
and therefore
\[\tr_{\mu,m}f_{k,X_0}=\frac{(2j+1)! \sqrt{Q(X_0)}}{2^{k-4}(j!)^2 |\Gamma_{X_0}|}\,\tr_{\mu,m}(G_{2+2j}(\cdot,\tau_{X_0})).\]
Finally, as a direct consequence of Corollary \ref{cor:algebraicity_traces} we obtain the following rationality result: Let $\{\lambda(\mu,m)\}$ be a rational relation for $S_{\frac{1}{2}+k,L^-}$ and  $X_0\in L\otimes \Q$ with $Q(X_0)>0$. Assume that $\tau_{X_0}\not \in H(\mu,m)$ for all $\mu\in L'/L$ and $m\in \Z- Q(\mu)$ with $\lambda(\mu,m)\neq 0$. Then, we have
\[\sum_{\mu \in L'/L}\sum_{\substack{m\in \Z-Q(\mu)\\m>0}} m^{\frac{k-1}{2}}\lambda(\mu,m)\,\mathrm{tr}_{\mu,m}f_{k,X_0} \in \Q.\]

% Note that this argument does not work for odd $k$ since we cannot reduce to $R_0^0 G_k$. 

\bibliographystyle{alpha}
\bibliography{references} 

@misc{Mel08,
      title={{H}igher {G}reen's functions for modular forms}, 
      author={Anton Mellit},
      year={2008},
      note   =  {arXiv:0804.3184},
      eprint={0804.3184},
      archivePrefix={arXiv},
      primaryClass={math.NT},
      url={https://arxiv.org/abs/0804.3184}, 
}

@misc{Via11,
  author        = {Viazovska, Maryna},
  title         = {{CM} {V}alues of {H}igher {G}reen's {F}unctions},
  year          = {2011},
  note   =  {arXiv:1110.4654},
  eprint        = {1110.4654},
  archivePrefix = {arXiv},
  primaryClass  = {math.NT},
  doi           = {10.48550/arXiv.1110.4654},
  url           = {https://arxiv.org/abs/1110.4654}
}

@article {Via19,
    AUTHOR = {Viazovska, Maryna},
     TITLE = {Petersson inner products of weight-one modular forms},
   JOURNAL = {J. Reine Angew. Math.},
  FJOURNAL = {Journal f\"ur die Reine und Angewandte Mathematik. [Crelle's
              Journal]},
    VOLUME = {749},
      YEAR = {2019},
     PAGES = {133--159},
      ISSN = {0075-4102,1435-5345},
   MRCLASS = {11F27 (11G50 11N75 11R27)},
  MRNUMBER = {3935901},
MRREVIEWER = {Shaul\ Zemel},
       DOI = {10.1515/crelle-2016-0042},
       URL = {https://doi.org/10.1515/crelle-2016-0042},
}

@article {Li23,
    AUTHOR = {Li, Yingkun},
     TITLE = {Algebraicity of higher {G}reen functions at a {CM} point},
   JOURNAL = {Invent. Math.},
  FJOURNAL = {Inventiones Mathematicae},
    VOLUME = {234},
      YEAR = {2023},
    NUMBER = {1},
     PAGES = {375--418},
      ISSN = {0020-9910,1432-1297},
   MRCLASS = {11F41 (11G15)},
  MRNUMBER = {4635835},
MRREVIEWER = {Seewoo\ Lee},
       DOI = {10.1007/s00222-023-01205-5},
       URL = {https://doi.org/10.1007/s00222-023-01205-5},
}

@article {Li22,
    AUTHOR = {Li, Yingkun},
     TITLE = {Average {CM}-values of higher {G}reen's function and
              factorization},
   JOURNAL = {Amer. J. Math.},
  FJOURNAL = {American Journal of Mathematics},
    VOLUME = {144},
      YEAR = {2022},
    NUMBER = {5},
     PAGES = {1241--1298},
      ISSN = {0002-9327,1080-6377},
   MRCLASS = {11F03 (11G15 11G18 11G40)},
  MRNUMBER = {4494182},
MRREVIEWER = {Rolf\ Berndt},
       DOI = {10.1353/ajm.2022.0029},
       URL = {https://doi.org/10.1353/ajm.2022.0029},
}

@article {BLY25,
    AUTHOR = {Bruinier, Jan H. and Li, Yingkun and Yang, Tonghai},
     TITLE = {Deformations of theta integrals and a conjecture of
              {G}ross-{Z}agier},
   JOURNAL = {Forum Math. Sigma},
  FJOURNAL = {Forum of Mathematics. Sigma},
    VOLUME = {13},
      YEAR = {2025},
     PAGES = {Paper No. e54, 59},
      ISSN = {2050-5094},
   MRCLASS = {11F37 (11F41 11G15)},
  MRNUMBER = {4878630},
MRREVIEWER = {Wei\ Xiong},
       DOI = {10.1017/fms.2024.139},
       URL = {https://doi.org/10.1017/fms.2024.139},
}

@article {BKvP19,
    AUTHOR = {Bringmann, Kathrin and Kane, Ben and von Pippich, Anna-Maria},
     TITLE = {Regularized inner products of meromorphic modular forms and
              higher {G}reen's functions},
   JOURNAL = {Commun. Contemp. Math.},
  FJOURNAL = {Communications in Contemporary Mathematics},
    VOLUME = {21},
      YEAR = {2019},
    NUMBER = {5},
     PAGES = {1850029, 35},
      ISSN = {0219-1997,1793-6683},
   MRCLASS = {11F37 (11F11)},
  MRNUMBER = {3980689},
MRREVIEWER = {Paul\ M.\ Jenkins},
       DOI = {10.1142/S0219199718500293},
       URL = {https://doi.org/10.1142/S0219199718500293},
}

@article {Mat99,
    AUTHOR = {Matthes, Roland},
     TITLE = {On some {P}oincar\'e-series on hyperbolic space},
   JOURNAL = {Forum Math.},
  FJOURNAL = {Forum Mathematicum},
    VOLUME = {11},
      YEAR = {1999},
    NUMBER = {4},
     PAGES = {483--502},
      ISSN = {0933-7741,1435-5337},
   MRCLASS = {11F55 (11F30 11F37)},
  MRNUMBER = {1699170},
MRREVIEWER = {Hiroshi\ Saito},
       DOI = {10.1515/form.1999.010},
       URL = {https://doi.org/10.1515/form.1999.010},
}

@book {AAR99,
    AUTHOR = {Andrews, George E. and Askey, Richard and Roy, Ranjan},
     TITLE = {Special functions},
    SERIES = {Encyclopedia of Mathematics and its Applications},
    VOLUME = {71},
 PUBLISHER = {Cambridge University Press, Cambridge},
      YEAR = {1999},
     PAGES = {xvi+664},
      ISBN = {0-521-62321-9; 0-521-78988-5},
   MRCLASS = {33-01 (33-02)},
  MRNUMBER = {1688958},
MRREVIEWER = {Bruce\ C.\ Berndt},
       DOI = {10.1017/CBO9781107325937},
       URL = {https://doi.org/10.1017/CBO9781107325937},
}

@book {Kne02,
    AUTHOR = {Kneser, Martin},
     TITLE = {Quadratische {F}ormen},
      NOTE = {Revised and edited in collaboration with Rudolf Scharlau},
 PUBLISHER = {Springer-Verlag, Berlin},
      YEAR = {2002},
     PAGES = {viii+164},
      ISBN = {3-540-64650-7},
   MRCLASS = {11-02 (11Exx 11H55)},
  MRNUMBER = {2788987},
       DOI = {10.1007/978-3-642-56380-5},
       URL = {https://doi.org/10.1007/978-3-642-56380-5},
}

@book {EGM98,
    AUTHOR = {Elstrodt, J. and Grunewald, F. and Mennicke, J.},
     TITLE = {Groups acting on hyperbolic space},
    SERIES = {Springer Monographs in Mathematics},
      NOTE = {Harmonic analysis and number theory},
 PUBLISHER = {Springer-Verlag, Berlin},
      YEAR = {1998},
     PAGES = {xvi+524},
      ISBN = {3-540-62745-6},
   MRCLASS = {11F72 (11F12 11M36 22E40 57M50 57S30)},
  MRNUMBER = {1483315},
MRREVIEWER = {Stefan\ K\"uhnlein},
       DOI = {10.1007/978-3-662-03626-6},
       URL = {https://doi.org/10.1007/978-3-662-03626-6},
}

@book {BFOR17,
    AUTHOR = {Bringmann, Kathrin and Folsom, Amanda and Ono, Ken and Rolen,
              Larry},
     TITLE = {Harmonic {M}aass forms and mock modular forms: theory and
              applications},
    SERIES = {American Mathematical Society Colloquium Publications},
    VOLUME = {64},
 PUBLISHER = {American Mathematical Society, Providence, RI},
      YEAR = {2017},
     PAGES = {xv+391},
      ISBN = {978-1-4704-1944-8},
   MRCLASS = {11F03 (11F11 11F27 11F30 11F37 11F50 11P81 11P82)},
  MRNUMBER = {3729259},
MRREVIEWER = {Jeremy\ Lovejoy},
       DOI = {10.1090/coll/064},
       URL = {https://doi.org/10.1090/coll/064},
}

@article {HIPS25,
    AUTHOR = {Herrero, S. and Imamo\=glu, \"O. and von Pippich, A.-M. and
              Schwagenscheidt, M.},
     TITLE = {Special values of {G}reen's functions on hyperbolic 3-space},
   JOURNAL = {Duke Math. J.},
  FJOURNAL = {Duke Mathematical Journal},
    VOLUME = {174},
      YEAR = {2025},
    NUMBER = {15},
     PAGES = {3245--3316},
}

@article {AS20,
    AUTHOR = {Alfes-Neumann, Claudia and Schwagenscheidt, Markus},
     TITLE = {Identities of cycle integrals of weak {M}aass forms},
   JOURNAL = {Ramanujan J.},
  FJOURNAL = {Ramanujan Journal. An International Journal Devoted to the
              Areas of Mathematics Influenced by Ramanujan},
    VOLUME = {52},
      YEAR = {2020},
    NUMBER = {3},
     PAGES = {683--688},
      ISSN = {1382-4090,1572-9303},
   MRCLASS = {11F37 (11F11 11F25 11F67)},
  MRNUMBER = {4122768},
MRREVIEWER = {Soumya\ Das},
       DOI = {10.1007/s11139-019-00143-2},
       URL = {https://doi.org/10.1007/s11139-019-00143-2},
}

@article {BS21,
    AUTHOR = {Bruinier, Jan Hendrik and Schwagenscheidt, Markus},
     TITLE = {Theta lifts for {L}orentzian lattices and coefficients of mock
              theta functions},
   JOURNAL = {Math. Z.},
  FJOURNAL = {Mathematische Zeitschrift},
    VOLUME = {297},
      YEAR = {2021},
    NUMBER = {3-4},
     PAGES = {1633--1657},
      ISSN = {0025-5874,1432-1823},
   MRCLASS = {11F27},
  MRNUMBER = {4229617},
MRREVIEWER = {Wei\ Xiong},
       DOI = {10.1007/s00209-020-02572-y},
       URL = {https://doi.org/10.1007/s00209-020-02572-y},
}

@article {McG03,
    AUTHOR = {McGraw, William J.},
     TITLE = {The rationality of vector valued modular forms associated with
              the {W}eil representation},
   JOURNAL = {Math. Ann.},
  FJOURNAL = {Mathematische Annalen},
    VOLUME = {326},
      YEAR = {2003},
    NUMBER = {1},
     PAGES = {105--122},
      ISSN = {0025-5831,1432-1807},
   MRCLASS = {11F30 (11F27)},
  MRNUMBER = {1981614},
MRREVIEWER = {Jan\ Hendrik\ Bruinier},
       DOI = {10.1007/s00208-003-0413-1},
       URL = {https://doi.org/10.1007/s00208-003-0413-1},
}

@article {LS22,
    AUTHOR = {Li, Yingkun and Schwagenscheidt, Markus},
     TITLE = {Mock modular forms with integral {F}ourier coefficients},
   JOURNAL = {Adv. Math.},
  FJOURNAL = {Advances in Mathematics},
    VOLUME = {399},
      YEAR = {2022},
     PAGES = {Paper No. 108264, 30},
      ISSN = {0001-8708,1090-2082},
   MRCLASS = {11F27 (11F30 11F37)},
  MRNUMBER = {4384613},
MRREVIEWER = {Larry\ Rolen},
       DOI = {10.1016/j.aim.2022.108264},
       URL = {https://doi.org/10.1016/j.aim.2022.108264},
}

@article {borcherds,
    AUTHOR = {Borcherds, Richard E.},
     TITLE = {Automorphic forms with singularities on {G}rassmannians},
   JOURNAL = {Invent. Math.},
  FJOURNAL = {Inventiones Mathematicae},
    VOLUME = {132},
      YEAR = {1998},
    NUMBER = {3},
     PAGES = {491--562},
      ISSN = {0020-9910,1432-1297},
   MRCLASS = {11F37 (11F22 14J28 17B67 57R57)},
  MRNUMBER = {1625724},
MRREVIEWER = {I.\ Dolgachev},
       DOI = {10.1007/s002220050232},
       URL = {https://doi.org/10.1007/s002220050232},
}

@article {ABMS21,
    AUTHOR = {Alfes-Neumann, Claudia and Bringmann, Kathrin and Males, J. and
              Schwagenscheidt, M.},
     TITLE = {Cycle integrals of meromorphic modular forms and coefficients
              of harmonic {M}aass forms},
   JOURNAL = {J. Math. Anal. Appl.},
  FJOURNAL = {Journal of Mathematical Analysis and Applications},
    VOLUME = {497},
      YEAR = {2021},
    NUMBER = {2},
     PAGES = {Paper No. 124898, 15},
      ISSN = {0022-247X,1096-0813},
   MRCLASS = {11F37 (11F27 11F67)},
  MRNUMBER = {4196572},
MRREVIEWER = {Larry\ Rolen},
       DOI = {10.1016/j.jmaa.2020.124898},
       URL = {https://doi.org/10.1016/j.jmaa.2020.124898},
}

@article {abs20,
    AUTHOR = {Alfes-Neumann, Claudia and Bringmann, Kathrin and
              Schwagenscheidt, Markus},
     TITLE = {On the rationality of cycle integrals of meromorphic modular
              forms},
   JOURNAL = {Math. Ann.},
  FJOURNAL = {Mathematische Annalen},
    VOLUME = {376},
      YEAR = {2020},
    NUMBER = {1-2},
     PAGES = {243--266},
      ISSN = {0025-5831,1432-1807},
   MRCLASS = {11F27 (11F37 30D30)},
  MRNUMBER = {4055161},
MRREVIEWER = {Jeremy\ Lovejoy},
       DOI = {10.1007/s00208-019-01914-4},
       URL = {https://doi.org/10.1007/s00208-019-01914-4},
}

@article {ls,
    AUTHOR = {L\"obrich, Steffen and Schwagenscheidt, Markus},
     TITLE = {Meromorphic modular forms with rational cycle integrals},
   JOURNAL = {Int. Math. Res. Not. IMRN},
  FJOURNAL = {International Mathematics Research Notices. IMRN},
      YEAR = {2022},
    NUMBER = {1},
     PAGES = {312--342},
      ISSN = {1073-7928,1687-0247},
   MRCLASS = {11F67 (11F12)},
  MRNUMBER = {4366018},
MRREVIEWER = {Seokho\ Jin},
       DOI = {10.1093/imrn/rnaa104},
       URL = {https://doi.org/10.1093/imrn/rnaa104},
}

@article {BEY21,
    AUTHOR = {Bruinier, Jan Hendrik and Ehlen, Stephan and Yang, Tonghai},
     TITLE = {C{M} values of higher automorphic {G}reen functions for
              orthogonal groups},
   JOURNAL = {Invent. Math.},
  FJOURNAL = {Inventiones Mathematicae},
    VOLUME = {225},
      YEAR = {2021},
    NUMBER = {3},
     PAGES = {693--785},
      ISSN = {0020-9910,1432-1297},
   MRCLASS = {11G18 (11F37 11G15)},
  MRNUMBER = {4296349},
MRREVIEWER = {\'Oscar\ Rivero},
       DOI = {10.1007/s00222-021-01038-0},
       URL = {https://doi.org/10.1007/s00222-021-01038-0},
}

@book {EMT81,
    AUTHOR = {Erd\'elyi, Arthur and Magnus, Wilhelm and Oberhettinger, Fritz
              and Tricomi, Francesco G.},
     TITLE = {Higher transcendental functions. {V}ol. {I}},
      NOTE = {Based on notes left by Harry Bateman,
              With a preface by Mina Rees,
              With a foreword by E. C. Watson,
              Reprint of the 1953 original},
 PUBLISHER = {Robert E. Krieger Publishing Co., Inc., Melbourne, FL},
      YEAR = {1981},
     PAGES = {xiii+302},
      ISBN = {0-89874-069-X},
   MRCLASS = {33-02 (01A75)},
  MRNUMBER = {698779},
}

@book {MOS,
    AUTHOR = {Magnus, Wilhelm and Oberhettinger, Fritz and Soni, Raj Pal},
     TITLE = {Formulas and theorems for the special functions of
              mathematical physics},
    SERIES = {Die Grundlehren der mathematischen Wissenschaften},
    VOLUME = {Band 52},
   EDITION = {enlarged},
 PUBLISHER = {Springer-Verlag New York, Inc., New York},
      YEAR = {1966},
     PAGES = {viii+508},
   MRCLASS = {33.00},
  MRNUMBER = {232968},
MRREVIEWER = {Mary\ L.\ Boas},
}

@book {GR07,
    AUTHOR = {Gradshteyn, I. S. and Ryzhik, I. M.},
     TITLE = {Table of integrals, series, and products},
   EDITION = {Seventh},
      NOTE = {Translated from the Russian,
              Translation edited and with a preface by Alan Jeffrey and
              Daniel Zwillinger,
              With one CD-ROM (Windows, Macintosh and UNIX)},
 PUBLISHER = {Elsevier/Academic Press, Amsterdam},
      YEAR = {2007},
     PAGES = {xlviii+1171},
      ISBN = {978-0-12-373637-6; 0-12-373637-4},
   MRCLASS = {00A22 (33-00 65-00 65A05)},
  MRNUMBER = {2360010},
}

@article {GKZ87,
    AUTHOR = {Gross, B. and Kohnen, W. and Zagier, D.},
     TITLE = {Heegner points and derivatives of {$L$}-series. {II}},
   JOURNAL = {Math. Ann.},
  FJOURNAL = {Mathematische Annalen},
    VOLUME = {278},
      YEAR = {1987},
    NUMBER = {1-4},
     PAGES = {497--562},
      ISSN = {0025-5831,1432-1807},
   MRCLASS = {11G40 (11F30 11G05 11R29 14G10)},
  MRNUMBER = {909238},
MRREVIEWER = {Lawrence\ Washington},
       DOI = {10.1007/BF01458081},
       URL = {https://doi.org/10.1007/BF01458081},
}

@article {GZ86,
    AUTHOR = {Gross, Benedict H. and Zagier, Don B.},
     TITLE = {Heegner points and derivatives of {$L$}-series},
   JOURNAL = {Invent. Math.},
  FJOURNAL = {Inventiones Mathematicae},
    VOLUME = {84},
      YEAR = {1986},
    NUMBER = {2},
     PAGES = {225--320},
      ISSN = {0020-9910,1432-1297},
   MRCLASS = {11G40 (11F11 11G05 14G10)},
  MRNUMBER = {833192},
MRREVIEWER = {Loren\ D.\ Olson},
       DOI = {10.1007/BF01388809},
       URL = {https://doi.org/10.1007/BF01388809},
}

@article {GZ85,
    AUTHOR = {Gross, Benedict H. and Zagier, Don B.},
     TITLE = {On singular moduli},
   JOURNAL = {J. Reine Angew. Math.},
  FJOURNAL = {Journal f\"{u}r die Reine und Angewandte Mathematik. [Crelle's
              Journal]},
    VOLUME = {355},
      YEAR = {1985},
     PAGES = {191--220},
      ISSN = {0075-4102},
   MRCLASS = {11F03 (11G05 14K15)},
  MRNUMBER = {772491},
MRREVIEWER = {R. W. K. Odoni},
}

@book{brhabil,
	Address = {Berlin},
	Author = {Bruinier, J. H.},
	Isbn = {3-540-43320-1},
	Mrclass = {11F55 (11F23 11F27 11G18)},
	Mrreviewer = {Rainer Schulze-Pillot},
	Pages = {viii+152},
	Publisher = {Springer-Verlag},
	Series = {Lecture Notes in Mathematics},
	Title = {Borcherds products on {O}(2, {$l$}) and {C}hern classes of {H}eegner divisors},
	Volume = {1780},
	Year = {2002}}

@Article{	  brfu04,
  author	= {Bruinier, Jan H. and Funke, Jens},
  title		= {On two geometric theta lifts},
  journal	= {Duke Math. J.},
  fjournal	= {Duke Mathematical Journal},
  volume	= {125},
  year		= {2004},
  number	= {1},
  pages		= {45--90},
  issn		= {0012-7094},
  coden		= {DUMJAO},
  mrclass	= {11F55 (11F27 11F32)},
  mrreviewer	= {Matthew G. Boylan},
  doi		= {10.1215/S0012-7094-04-12513-8},
  url		= {http://dx.doi.org/10.1215/S0012-7094-04-12513-8}
}

@article {BY09,
    AUTHOR = {Bruinier, Jan Hendrik and Yang, Tonghai},
     TITLE = {Faltings heights of {CM} cycles and derivatives of
              {$L$}-functions},
   JOURNAL = {Invent. Math.},
  FJOURNAL = {Inventiones Mathematicae},
    VOLUME = {177},
      YEAR = {2009},
    NUMBER = {3},
     PAGES = {631--681},
      ISSN = {0020-9910,1432-1297},
   MRCLASS = {11G18 (11F67 11G40 14G35 14G40)},
  MRNUMBER = {2534103},
MRREVIEWER = {Robin\ de Jong},
       DOI = {10.1007/s00222-009-0192-8},
       URL = {https://doi.org/10.1007/s00222-009-0192-8},
}

@book {abramowitz1988handbook,
    AUTHOR = {Abramowitz, Milton and Stegun, Irene A.},
     TITLE = {Handbook of mathematical functions with formulas, graphs, and
              mathematical tables},
    SERIES = {National Bureau of Standards Applied Mathematics Series},
    VOLUME = {No. 55},
      NOTE = {For sale by the Superintendent of Documents},
 PUBLISHER = {U. S. Government Printing Office, Washington, DC},
      YEAR = {1964},
     PAGES = {xiv+1046},
   MRCLASS = {33.00 (65.05)},
  MRNUMBER = {167642},
MRREVIEWER = {D.\ H.\ Lehmer},
}

@article {males,
    AUTHOR = {Males, Joshua},
     TITLE = {Higher {S}iegel theta lifts on {L}orentzian lattices, harmonic
              {M}aass forms, and {E}ichler-{S}elberg type relations},
   JOURNAL = {Math. Z.},
  FJOURNAL = {Mathematische Zeitschrift},
    VOLUME = {301},
      YEAR = {2022},
    NUMBER = {4},
     PAGES = {3555--3569},
      ISSN = {0025-5874,1432-1823},
   MRCLASS = {11F27 (11F70 32N10)},
  MRNUMBER = {4449720},
MRREVIEWER = {Larry\ Rolen},
       DOI = {10.1007/s00209-022-03023-6},
       URL = {https://doi.org/10.1007/s00209-022-03023-6},
}

@article {borthwick-marzuola,
    AUTHOR = {Borthwick, David and Marzuola, Jeremy L.},
     TITLE = {Dispersive estimates for scalar and matrix {S}chr\"{o}dinger
              operators on {$\Bbb{H}^{n+1}$}},
   JOURNAL = {Math. Phys. Anal. Geom.},
  FJOURNAL = {Mathematical Physics, Analysis and Geometry. An International
              Journal Devoted to the Theory and Applications of Analysis and
              Geometry to Physics},
    VOLUME = {18},
      YEAR = {2015},
    NUMBER = {1},
     PAGES = {Art. 22, 26},
      ISSN = {1385-0172,1572-9656},
   MRCLASS = {35Q55 (58J50)},
  MRNUMBER = {3376154},
MRREVIEWER = {Pietro\ d'Avenia},
       DOI = {10.1007/s11040-015-9191-8},
       URL = {https://doi.org/10.1007/s11040-015-9191-8},
}

\end{document}